\documentclass[reqno,10pt,centertags]{amsart}
\usepackage{amsmath,amsthm,amscd,amssymb,latexsym,esint,upref,enumerate,color,verbatim,
yfonts}
\date{\today}
\usepackage{hyperref}

\newcommand*{\mailto}[1]{\href{mailto:#1}{\nolinkurl{#1}}}

\makeatletter
\def\theequation{\@arabic\c@equation}

\newcommand{\bb}[1]{{\mathbb{#1}}}
\newcommand{\mc}[1]{{\mathcal{#1}}}
\newcommand{\oT}{H}
\newcommand{\D}{\mathbb{D}}
\newcommand{\bbN}{{\mathbb{N}}}
\newcommand{\bbR}{{\mathbb{R}}}

\newcommand{\bbZ}{{\mathbb{Z}}}
\newcommand{\bbC}{{\mathbb{C}}}

\newcommand{\cB}{{\mathcal B}}

\newcommand{\cD}{{\mathcal D}}

\newcommand{\cF}{{\mathcal F}}

\newcommand{\cH}{{\mathcal H}}
\newcommand{\cI}{{\mathcal I}}

\newcommand{\cK}{{\mathcal K}}

\newcommand{\cM}{{\mathcal M}}
\newcommand{\cN}{{\mathcal N}}

\newcommand{\cS}{{\mathcal S}}

\newcommand{\cU}{{\mathcal U}}
\newcommand{\cV}{{\mathcal V}}

\newcommand{\cX}{{\mathcal X}}

\newcommand{\no}{\nonumber}
\newcommand{\lb}{\label}
\newcommand{\f}{\frac}
\newcommand{\ul}{\underline}
\newcommand{\ol}{\overline}
\newcommand{\bs}{\backslash}

\newcommand{\wti}{\widetilde}

\newcommand{\la}{\lambda}
\newcommand{\al}{\alpha}

\newcommand{\ve}{\varepsilon}
\newcommand{\Oh}{O}
\newcommand{\oh}{o}

\newcommand{\loc}{\text{\rm{loc}}}

\newcommand{\ran}{\operatorname{ran}}

\newcommand{\dom}{\operatorname{dom}}
\newcommand{\supp}{\operatorname{supp}}

\renewcommand{\Re}{\operatorname{Re}}
\renewcommand{\Im}{\operatorname{Im}}

\renewcommand{\oint}{\ointctrclockwise}

\newcommand{\slimes}{\text{\rm{l.i.m.}}}

\newcommand{\bi}{\bibitem}
\newcommand{\hatt}{\widehat}

\newcommand{\essran}{\text{\rm ess.ran}}
\newcommand{\dD}{{\partial\hspace*{.2mm}\mathbb{D}}}

\DeclareMathOperator*{\slim}{s-lim}
\DeclareMathOperator*{\wlim}{w-lim}

\numberwithin{equation}{section}

\newtheorem{theorem}{Theorem}[section]

\newtheorem{lemma}[theorem]{Lemma}
\newtheorem{corollary}[theorem]{Corollary}
\newtheorem{hypothesis}[theorem]{Hypothesis}

\theoremstyle{definition}
\newtheorem{definition}[theorem]{Definition}
\newtheorem{remark}[theorem]{Remark}

\begin{document}

\title[Spectral Theory and Operator-Valued Potentials]{On Spectral
Theory for Schr\"odinger Operators with Operator-Valued Potentials}
\author[F.\ Gesztesy, R.\ Weikard, and M.\ Zinchenko]{Fritz
Gesztesy, Rudi Weikard, and Maxim Zinchenko}

\address{Department of Mathematics,
University of Missouri, Columbia, MO 65211, USA}
\email{\mailto{gesztesyf@missouri.edu}}
\urladdr{\url{http://www.math.missouri.edu/personnel/faculty/gesztesyf.html}}

\address{Department of Mathematics, University of
Alabama at Birmingham, Birmingham, AL 35294, USA}
\email{\mailto{rudi@math.uab.edu}}
\urladdr{\url{http://www.math.math.uab.edu/~rudi/index.html}}

\address{Department of Mathematics,
University of New Mexico, Albuquerque, NM 87131, USA}
\email{\mailto{maxim@math.unm.edu}}
\urladdr{\url{http://www.math.unm.edu/~maxim/}}

\date{\today}

\begin{abstract}
Given a complex, separable Hilbert space $\cH$, we consider differential expressions
of the type $\tau =  - (d^2/dx^2) + V(x)$, with $x \in (a,\infty)$ or $x \in \bbR$. Here $V$
denotes a bounded operator-valued potential $V(\cdot) \in \cB(\cH)$ such that $V(\cdot)$ is
weakly measurable and the operator norm $\|V(\cdot)\|_{\cB(\cH)}$ is locally integrable.

We consider self-adjoint half-line $L^2$-realizations $H_{\alpha}$ in $L^2((a,\infty); dx; \cH)$
associated with $\tau$, assuming $a$ to be a regular endpoint necessitating a
boundary condition of the type $\sin(\alpha)u'(a) + \cos(\alpha)u(a)=0$,
indexed by the self-adjoint operator $\alpha = \alpha^* \in \cB(\cH)$. In addition, we study
self-adjoint full-line $L^2$-realizations $H$ of $\tau$ in $L^2(\bbR; dx; \cH)$. In either case
we treat in detail basic spectral theory associated with $H_{\alpha}$ and $H$, including
Weyl--Titchmarsh theory, Green's function structure, eigenfunction expansions, diagonalization,
and a version of the spectral theorem.
\end{abstract}

\maketitle


\section{Introduction} \lb{s1}

The principal topic of this paper centers around basic spectral theory, including
Weyl--Titchmarsh theory, Green's function structure, eigenfunction expansions, diagonalization,
and a version of the spectral theorem
 for self-adjoint Schr\"odinger operators with bounded operator-valued
potentials on a half-line as well as on the full real line. More precisely, given a complex,
separable Hilbert space $\cH$, we consider differential expressions $\tau$ of the type
\begin{equation}
\tau =  - (d^2/dx^2) + V(x),    \lb{1.1}
\end{equation}
with $x \in (a,\infty)$ or $x \in \bbR$, and $V$ a bounded
operator-valued potential $V(\cdot) \in \cB(\cH)$ such that $V(\cdot)$ is weakly measurable
and the operator norm $\|V(\cdot)\|_{\cB(\cH)}$ is locally integrable. The self-adjoint operators
in question are then half-line $L^2$-realizations of $\tau$ in $L^2((a,\infty); dx; \cH)$,
with $a$ assumed to be a regular endpoint for $\tau$, and hence with appropriate boundary
conditions at $a$ (cf.\ \eqref{1.2}) on one hand, and full-line $L^2$-realizations of $\tau$ in
$L^2(\bbR; dx; \cH)$ on the other.

The case of Schr\"odinger operators with operator-valued potentials under various
continuity or smoothness hypotheses on $V(\cdot)$, and under various self-adjoint
boundary conditions on bounded and unbounded open intervals, received considerable attention
in the past. In the special case where $\dim(\cH)<\infty$, that is, in the case of Schr\"odinger
operators with matrix-valued potentials, the literature is so voluminous that we cannot
possibly describe individual references and hence we primarily refer to the monographs
\cite{AM63}, \cite{RK05}, and the references cited therein. We note that the
finite-dimensional case, $\dim(\cH) < \infty$, as discussed in \cite{BL00}, is of
considerable interest as it represents an important ingredient in some proofs of
Lieb--Thirring inequalities (cf.\ \cite{LW00}). In addition, the constant coefficient case, where
$\tau$ is of the special form $\tau = - (d^2/dx^2) + A$, has received overwhelming attention.
But since this is not the focus of this paper we just refer to \cite{GG73}, \cite[Chs.\ 3, 4]{GG91},
\cite{MN11a}, and the literature cited therein.

In the particular case of Schr\"odinger-type operators corresponding to the differential expression
$\tau = - (d^2/dx^2) + A + V(x)$ on a bounded interval $(a,b) \subset \bbR$ with either
$A=0$ or $A$ a self-adjoint operator satisfying $A\geq c I_{\cH}$ for some $c>0$,
unique solvability of boundary value problems, the asymptotic behavior of
eigenvalues, and trace formulas
in connection with various self-adjoint realizations of $\tau = - (d^2/dx^2) + A + V(x)$
on a bounded interval $(a,b)$ are discussed, for instance, in \cite{ABK91}, \cite{ABK93},
\cite{AK92}, \cite{As08},
\cite{Go68}, \cite{GG69}, \cite{GM76}, \cite{Gu08}, \cite{Mo07}, \cite{Mo10} (for the
case of spectral parameter dependent separated boundary conditions, see also
\cite{Al06}, \cite{Al10}, \cite{BA10}).

For earlier results on various aspects of boundary value problems, spectral theory,
and scattering theory in the half-line case $(a,b) =(0,\infty)$, the situation closely
related to the principal topic of this paper, we refer, for instance,
to \cite{Al06a}, \cite{AM10}, \cite{De08}, \cite{Go68}--\cite{Go71}, \cite{GM76},
\cite{KL67}, \cite{Mo07}, \cite{Mo10}, \cite{Ro60}, \cite{Sa71}, \cite{Tr00}
(the case of the real line is discussed in \cite{VG70}). While our treatment of initial value
problems associated with $\tau$ given by \eqref{1.1} in \cite{GWZ13} was originally inspired
by the one in Sait{\= o} \cite{Sa71}, we do permit a more general local behavior of $V(\cdot)$.
With respect to spectral theory for self-adjoint half-line realizations of $\tau$ in
$L^2((a,\infty); dx; \cH)$ we refer to the fundamental paper by Gorbachuk \cite{Go68}. Our
treatment in this context again permits more general potentials $V(\cdot)$, we also provide
all details in connection with the derivation of \eqref{2.28} (cf.\ \eqref{2.32}--\eqref{2.38}), not
present in \cite{Go68}. Our $2 \times 2$ block operator approach in Section \ref{s5} in connection
with full-line realizations of $\tau$ in $L^2(\bbR; dx; \cH)$, with special emphasis on the structure
of the Green's function \eqref{2.63} and the Weyl--Titchmarsh matrix \eqref{2.71} appears to be
new, in particular, Theorems \ref{t2.9} and \ref{t2.10}, represent the principal new results in this
paper in this operator-valued setting.

Next we briefly turn to the content of each section: Section \ref{s2} recalls our basic results in
\cite{GWZ13} on the initial value problem associated with Schr\"odinger operators with bounded
operator-valued potentials. We use this section to introduce some of the basic notation employed
subsequently and note that our conditions on $V(\cdot)$ (cf.\ Hypothesis \ref{h2.7}) are the most
general to date with respect to the local behavior of the latter. Also Section \ref{s3} is of preparatory
nature. Again following our detailed treatment in \cite{GWZ13}, we introduce maximal and minimal
operators associated with $\tau = - (d^2/dx^2) + V(\cdot)$ on the interval $(a,b) \subset \bbR$
(eventually aiming at the case of a half-line $(a,\infty)$), and assuming that the left end point
$a$ is regular for $\tau$ and that $\tau$ is in the limit point case at the end point $b$ we discuss
a family of self-adjoint extensions $H_{\alpha}$ in $L^2((a,b); dx; \cH)$ corresponding to boundary
conditions of the type
\begin{equation}
\sin(\alpha)u'(a) + \cos(\alpha)u(a)=0,     \lb{1.2}
\end{equation}
indexed by the self-adjoint operator $\alpha = \alpha^* \in \cB(\cH)$ with $u \in \cH$ lying in the
domain of the maximal operator $H_{\max}$ corresponding to $\tau$.
In addition, we recall elements of Weyl--Titchmarsh theory, culminating in the introduction of the
operator-valued Weyl--Titchmarsh function $m_{\alpha}(\cdot) \in \cB(\cH)$ and the Green's
function $G_{\alpha}(z,\cdot,\cdot) \in \cB(\cH)$ of $H_{\alpha}$. Section \ref{s4} then presents our
first set of principal spectral results for the right half-line $(a,\infty)$, denoting the corresponding
self-adjoint right half-line operator in $L^2((a,\infty); dx; \cH)$ by $H_{+,\alpha}$: Theorem \ref{t2.5}
and especially,
Theorem \ref{t2.6}, then yield a diagonalization of $H_{+,\alpha}$ and contain its underlying
generalized eigenfunction expansion, including a description of support properties of the
$\cB(\cH)$-valued half-line
spectral measure $d\rho_{+,\alpha}$. In particular, they illustrate the spectral theorem for
$F(H_{+,\alpha})$, $F \in C(\bbR)$. Our final Section \ref{s5} then derives the analogous results
for full-line Schr\"odinger operators $H$ in $L^2(\bbR; dx; \cH)$, employing a $2 \times 2$ block
operator representation of the associated Weyl--Titchmarsh $M_{\alpha} (\cdot,x_0)$-matrix and its
$\cB\big(\cH^2\big)$-valued spectral measure $d\Omega_{\alpha}(\cdot,x_0)$, decomposing $\bbR$
into a left and right half-line with reference point $x_0 \in \bbR$, $(-\infty, x_0] \cup [x_0, \infty)$. The
latter decomposition is familiar from the scalar and matrix-valued ($\dim(\cH) < \infty$) special cases.
Our principal new results, Theorems \ref{t2.9} and \ref{t2.10} again yield a diagonalization of $H$
and the corresponding generalized eigenfunction expansion, illustrating the spectral theorem for
$F(H)$ and support properties of $d\Omega_{\alpha}(\cdot,x_0)$. Appendix \ref{sA} collects basic
facts on operator-valued Herglotz functions, some of which are of interest in their own right.
Appendix \ref{sD} recalls several equivalent definitions of direct integrals of Hilbert spaces and
constructions of the model Hilbert space $L^2(\bbR; d\Sigma;\cK)$ associated with a $\cB(\cK)$-valued measure $d\Sigma)$ described in \cite{GKMT01} and \cite{GWZ13a} and also describes a new connection with a construction due to Sait{\= o} \cite{Sa71}. The topics in both appendices are frequently used throughout this manuscript and we hope they render this paper sufficiently self-contained.

We should also add that while this paper completes our project on Schr\"odinger operators with
bounded operator-valued potentials, it simultaneously represents the basis for the next step in this
program: This step aims at certain classes of unbounded operator-valued potentials $V$,
applicable to multi-dimensional Schr\"odinger operators in $L^2(\bbR^n; d^n x)$, $n \in \bbN$,
$n \geq 2$, generated by differential expressions of the type $\Delta + V(\cdot)$. It was precisely the connection between multi-dimensional Schr\"odinger
operators and one-dimensional Schr\"odinger operators with unbounded operator-valued potentials
which originally motivated our interest in this circle of ideas. This connection was already employed by
Kato \cite{Ka59} in 1959; for more recent applications of this connection between one-dimensional
Schr\"odinger operators with unbounded operator-valued potentials and multi-dimensional
Schr\"odinger operators we refer, for instance, to \cite{ACH99}, \cite{Cr01},
\cite{Ja70}, \cite{LNS05}, \cite{MN11a}, \cite{Mi76}, \cite{Mi83}, \cite{Mi83a},
\cite{Sa08}, \cite{SS07}, \cite{Sa71a}--\cite{Sa79}, and the references cited therein.

Finally, we comment on the notation used in this paper: Throughout, $\cH$
denotes a separable, complex Hilbert space with inner product and norm
denoted by $(\cdot,\cdot)_{\cH}$ (linear in the second argument) and
$\|\cdot \|_{\cH}$, respectively. The identity operator in $\cH$ is written as
$I_{\cH}$. We denote by
$\cB(\cH)$ (resp., $\cB_{\infty}(\cH)$) the Banach space of linear bounded (resp., compact)
operators in $\cH$. The domain, range, kernel (null space), resolvent set, and spectrum
of a linear operator will be denoted by $\dom(\cdot)$,
$\ran(\cdot)$, $\ker(\cdot)$, $\rho(\cdot)$, and $\sigma(\cdot)$, respectively. The closure
of a closable operator $S$ in $\cH$ is denoted by $\ol S$.

By $\mathfrak{B}(\bbR)$ we denote the collection of Borel subsets of $\bbR$.

\section{The Initial Value Problem Associated With Schr\"odinger Operators
with Operator-Valued Potentials Revisited} \label{s2}

In this section we recall the basic results about initial value problems for second-order differential equations of the form $-y''+Qy=f$ on an arbitrary open interval
$(a,b) \subseteq \bbR$ with a bounded operator-valued coefficient $Q$, that is, when $Q(x)$ is a bounded operator on a separable, complex Hilbert space $\cH$ for a.e.\ $x\in(a,b)$. In fact, we are interested in two types of situations: In the first one $f(x)$ is an element of the Hilbert space $\cH$ for a.e.\ $x\in (a,b)$, and the solution sought is to take values in $\cH$. In the second situation, $f(x)$ is a bounded operator on $\cH$ for a.e.\ $x\in(a,b)$, as is the proposed solution $y$.

All results recalled in this section were proved in detail in \cite{GWZ13}.

We start with some necessary preliminaries:
Let $(a,b) \subseteq \bbR$ be a finite or infinite interval and $\cX$ a Banach space.
Unless explicitly stated otherwise (such as in the context of operator-valued measures in
Herglotz representations, cf.\ Appendix \ref{sA}), integration of $\cX$-valued functions on $(a,b)$ will
always be understood in the sense of Bochner (cf., e.g., \cite[p.\ 6--21]{ABHN01},
\cite[p.\ 44--50]{DU77}, \cite[p.\ 71--86]{HP85}, \cite[Ch.\ III]{Mi78}, \cite[Sect.\ V.5]{Yo80} for
details). In particular, if $p\ge 1$, the symbol $L^p((a,b);dx;\cX)$ denotes the set of equivalence classes of strongly measurable $\cX$-valued functions which differ at most on sets of Lebesgue measure zero, such that $\|f(\cdot)\|_{\cX}^p \in L^1((a,b);dx)$. The
corresponding norm in $L^p((a,b);dx;\cX)$ is given by
\begin{equation}
\|f\|_{L^p((a,b);dx;\cX)} = \bigg(\int_{(a,b)} dx\, \|f(x)\|_{\cX}^p \bigg)^{1/p}
\end{equation}
and $L^p((a,b);dx;\cX)$ is a Banach space.

If $\cH$ is a separable Hilbert space, then so is $L^2((a,b);dx;\cH)$ (see, e.g.,
\cite[Subsects.\ 4.3.1, 4.3.2]{BW83}, \cite[Sect.\ 7.1]{BS87}).

One recalls that by a result of Pettis \cite{Pe38}, if $\cX$ is separable, weak
measurability of $\cX$-valued functions implies their strong measurability.

If $g \in L^1((a,b);dx;\cX)$, $f(x)= \int_{x_0}^x dx' g(x')$, $x_0, x \in (a,b)$, then $f$ is
strongly differentiable a.e.\ on $(a,b)$ and
\begin{equation}
f'(x) = g(x) \, \text{ for a.e.\ $x \in (a,b)$}.
\end{equation}
In addition,
\begin{equation}
\lim_{t\downarrow 0} \f{1}{t} \int_x^{x+t} dx' \|g(x') - g(x)\|_{\cX} = 0 \,
\text{ for a.e.\ $x \in (a,b)$,}
\end{equation}
in particular,
\begin{equation}
\slim_{t\downarrow 0}\f{1}{t} \int_x^{x+t} dx' g(x') = g(x)
\, \text{ for a.e.\ $x \in (a,b)$.}
\end{equation}

Sobolev spaces $W^{n,p}((a,b); dx; \cX)$ for $n\in\bbN$ and $p\geq 1$ are defined as follows: $W^{1,p}((a,b);dx;\cX)$ is the set of all
$f\in L^p((a,b);dx;\cX)$ such that there exists a $g\in L^p((a,b);dx;\cX)$ and an
$x_0\in(a,b)$ such that
\begin{equation}
f(x)=f(x_0)+\int_{x_0}^x dx' \, g(x') \, \text{ for a.e.\ $x \in (a,b)$.}
\end{equation}
In this case $g$ is the strong derivative of $f$, $g=f'$. Similarly,
$W^{n,p}((a,b);dx;\cX)$ is the set of all $f\in L^p((a,b);dx;\cX)$ so that the first $n$ strong
derivatives of $f$ are in $L^p((a,b);dx;\cX)$. For simplicity of notation one also introduces
$W^{0,p}((a,b);dx;\cX)=L^p((a,b);dx;\cX)$. Finally, $W^{n,p}_{\rm loc}((a,b);dx;\cX)$ is
the set of $\cX$-valued functions defined on $(a,b)$ for which the restrictions to any
compact interval $[\alpha,\beta]\subset(a,b)$ are in $W^{n,p}((\alpha,\beta);dx;\cX)$.
In particular, this applies to the case $n=0$ and thus defines $L^p_{\rm loc}((a,b);dx;\cX)$.
If $a$ is finite we may allow $[\alpha,\beta]$ to be a subset of $[a,b)$ and denote the
resulting space by $W^{n,p}_{\rm loc}([a,b);dx;\cX)$ (and again this applies to the case
$n=0$).

Following a frequent practice (cf., e.g., the discussion in \cite[Sect.\ III.1.2]{Am95}), we
will call elements of $W^{1,1} ([c,d];dx;\cX)$, $[c,d] \subset (a,b)$ (resp.,
$W^{1,1}_{\rm loc}((a,b);dx;\cX)$), strongly absolutely continuous $\cX$-valued functions
on $[c,d]$ (resp., strongly locally absolutely continuous $\cX$-valued functions
on $(a,b)$), but caution the reader that unless $\cX$ posseses the Radon--Nikodym
(RN) property, this notion differs from the classical definition
of $\cX$-valued absolutely continuous functions (we refer the interested reader
to \cite[Sect.\ VII.6]{DU77} for an extensive list of conditions equivalent to $\cX$ having the
RN property). Here we just mention that reflexivity of $\cX$ implies the RN property.

In the special case where $\cX = \bbC$, we omit $\cX$ and just write
$L^p_{(\loc)}((a,b);dx)$, as usual.

{\bf A Remark on notational convention:} To avoid possible confusion later on between
two standard
notions of strongly continuous operator-valued functions $F(x)$,
$x \in (a,b)$, that is, strong continuity of $F(\cdot) h$ in $\cH$ for all $h \in\cH$ (i.e.,
pointwise continuity of $F(\cdot)$), versus strong continuity of $F(\cdot)$ in the norm
of $\cB(\cH)$ (i.e., uniform continuity of $F(\cdot)$), we will always mean pointwise continuity of $F(\cdot)$ in $\cH$. The same pointwise conventions will apply to the notions of strongly differentiable and strongly measurable operator-valued functions throughout this manuscript.
In particular, and unless explicitly stated otherwise, for operator-valued functions $Y$, the symbol $Y'$ will be understood in the strong sense; similarly,  $y'$ will denote the strong derivative for vector-valued functions $y$.

We start by recalling the following elementary, yet useful lemma:

\begin{lemma} \label{l2.1}
Let $(a,b)\subseteq\bbR$. Suppose $Q:(a,b)\to\cB(\cH)$ is a weakly
measurable operator-valued function with $\|Q(\cdot)\|_{\cB(\cH)}\in L^1_\loc((a,b);dx)$ and $g:(a,b)\to\cH$ is $($weakly$)$ measurable. Then $Qg$ is $($strongly$)$ measurable. Moreover, if $g$ is strongly continuous, then there exists a set $E\subset(a,b)$ with zero Lebesgue measure, depending only on $Q$, such that for every $x_0\in(a,b)\bs E$,
\begin{equation}
\lim_{t\downarrow0}\f{1}{t}\int_{x_0}^{x_0+t} dx \, \|Q(x)g(x) - Q(x_0)g(x_0)\|_\cH = 0,
\lb{2.6A}
\end{equation}
in particular,
\begin{equation}
\slim_{t\downarrow0} \f{1}{t}\int_{x_0}^{x_0+t} dx \, Q(x)g(x) = Q(x_0)g(x_0).
\lb{2.7A}
\end{equation}
In addition, the set of Lebesgue points of $Q(\cdot)g(\cdot)$ can be
chosen independently of $g$.
\end{lemma}

In connection with \eqref{2.7A} we also refer to \cite[Theorem\ II.2.9]{DU77},
\cite[Subsect.\ III.3.8]{HP85}, \cite[Theorem\ V.5.2]{Yo80}.

\begin{definition} \lb{d2.2}
Let $(a,b)\subseteq\bbR$ be a finite or infinite interval and
$Q:(a,b)\to\cB(\cH)$ a weakly measurable operator-valued function with
$\|Q(\cdot)\|_{\cB(\cH)}\in L^1_\loc((a,b);dx)$, and suppose that
$f\in L^1_{\loc}((a,b);dx;\cH)$. Then the $\cH$-valued function
$y: (a,b)\to \cH$ is called a (strong) solution of
\begin{equation}
- y'' + Q y = f   \lb{2.15A}
\end{equation}
if $y \in W^{2,1}_\loc((a,b);dx;\cH)$ and \eqref{2.15A} holds a.e.\ on $(a,b)$.
\end{definition}

We recall our notational convention that vector-valued solutions of \eqref{2.15A} will always be viewed as strong solutions.

One verifies that $Q:(a,b)\to\cB(\cH)$ satisfies the conditions in
Definition \ref{d2.2} if and only if $Q^*$ does (a fact that will play a role later on, cf.\
the paragraph following \eqref{2.33A}).

\begin{theorem} \lb{t2.3}
Let $(a,b)\subseteq\bbR$ be a finite or infinite interval and
$V:(a,b)\to\cB(\cH)$ a weakly measurable operator-valued function with
$\|V(\cdot)\|_{\cB(\cH)}\in L^1_\loc((a,b);dx)$. Suppose that
$x_0\in(a,b)$, $z\in\bbC$, $h_0,h_1\in\cH$, and $f\in
L^1_{\loc}((a,b);dx;\cH)$. Then there is a unique $\cH$-valued
solution $y(z,\cdot,x_0)\in W^{2,1}_\loc((a,b);dx;\cH)$ of the initial value problem
\begin{equation}
\begin{cases}
- y'' + (V - z) y = f \, \text{ on } \, (a,b)\bs E,  \\
\, y(x_0) = h_0, \; y'(x_0) = h_1,
\end{cases}     \lb{2.1}
\end{equation}
where the exceptional set $E$ is of Lebesgue measure zero and independent
of $z$.

Moreover, the following properties hold:
\begin{enumerate}[$(i)$]
\item For fixed $x_0,x\in(a,b)$ and $z\in\bbC$, $y(z,x,x_0)$ depends jointly continuously on $h_0,h_1\in\cH$, and $f\in L^1_{\loc}((a,b);dx;\cH)$ in the sense that
\begin{align}
\begin{split}
& \big\|y\big(z,x,x_0;h_0,h_1,f\big) - y\big(z,x,x_0;\wti h_0,\wti h_1,\wti f\big)\big\|_{\cH}    \\
& \quad \leq C(z,V)
\big[\big\|h_0 - \wti h_0\big\|_{\cH} + \big\|h_1 - \wti h_1\big\|_{\cH}
+ \big\|f - \wti f\big\|_{L^1([x_0,x];dx;\cH)}\big],    \lb{2.1A}
\end{split}
\end{align}
where $C(z,V)>0$ is a constant, and the dependence of
$y$ on the initial data $h_0, h_1$ and the inhomogeneity $f$ is displayed
in \eqref{2.1A}.
\item For fixed $x_0\in(a,b)$ and $z\in\bbC$, $y(z,x,x_0)$ is strongly continuously differentiable with respect to $x$ on $(a,b)$.
\item For fixed $x_0\in(a,b)$ and $z\in\bbC$, $y'(z,x,x_0)$ is strongly differentiable with respect to $x$ on $(a,b)\bs E$.
\item For fixed $x_0,x \in (a,b)$, $y(z,x,x_0)$ and $y'(z,x,x_0)$
are entire with respect to $z$.
\end{enumerate}
\end{theorem}

For classical references on initial value problems we refer, for instance, to
\cite[Chs.\ III, VII]{DK74} and \cite[Ch.\ 10]{Di60}, but we emphasize again that our approach minimizes the smoothness hypotheses on $V$ and $f$.

\begin{definition} \lb{d2.4}
Let $(a,b)\subseteq\bbR$ be a finite or infinite interval and assume that
$F,\,Q:(a,b)\to\cB(\cH)$ are two weakly measurable operator-valued functions such
that $\|F(\cdot)\|_{\cB(\cH)},\,\|Q(\cdot)\|_{\cB(\cH)}\in L^1_\loc((a,b);dx)$. Then the
$\cB(\cH)$-valued function $Y:(a,b)\to\cB(\cH)$ is called a solution of
\begin{equation}
- Y'' + Q Y = F   \lb{2.26A}
\end{equation}
if $Y(\cdot)h\in W^{2,1}_\loc((a,b);dx;\cH)$ for every $h\in\cH$ and $-Y''h+QYh=Fh$ holds
a.e.\ on $(a,b)$.
\end{definition}

\begin{corollary} \lb{c2.5}
Let $(a,b)\subseteq\bbR$ be a finite or infinite interval, $x_0\in(a,b)$, $z\in\bbC$, $Y_0,\,Y_1\in\cB(\cH)$, and suppose $F,\,V:(a,b)\to\cB(\cH)$ are two weakly measurable operator-valued functions with
$\|V(\cdot)\|_{\cB(\cH)},\,\|F(\cdot)\|_{\cB(\cH)}\in L^1_\loc((a,b);dx)$. Then there is a
unique $\cB(\cH)$-valued solution $Y(z,\cdot,x_0):(a,b)\to\cB(\cH)$ of the initial value
problem
\begin{equation}
\begin{cases}
- Y'' + (V - z)Y = F \, \text{ on } \, (a,b)\bs E,  \\
\, Y(x_0) = Y_0, \; Y'(x_0) = Y_1.
\end{cases} \lb{2.3}
\end{equation}
where the exceptional set $E$ is of Lebesgue measure zero and independent
of $z$. Moreover, the following properties hold:
\begin{enumerate}[$(i)$]
\item For fixed $x_0 \in (a,b)$ and $z \in \bbC$, $Y(z,x,x_0)$ is continuously
differentiable with respect to $x$ on $(a,b)$ in the $\cB(\cH)$-norm.
\item For fixed $x_0 \in (a,b)$ and $z \in \bbC$, $Y'(z,x,x_0)$ is strongly differentiable with respect to $x$ on $(a,b)\bs E$.
\item For fixed $x_0, x \in (a,b)$, $Y(z,x,x_0)$ and $Y'(z,x,x_0)$ are entire in $z$ in
the $\cB(\cH)$-norm.
\end{enumerate}
\end{corollary}

Various versions of Theorem \ref{t2.3} and Corollary \ref{c2.5} exist in the literature
under varying assumptions on $V$ and $f, F$. For instance, the case where $V(\cdot)$ is continuous in the $\cB(\cH)$-norm and $F=0$ is discussed in \cite[Theorem\ 6.1.1]{Hi69}.
The case, where $\|V(\cdot)\|_{\cB(\cH} \in L^1_{\loc}([a,c];dx)$ for all $c>a$ and
$F=0$ is discussed in detail in \cite{Sa71} (it appears that a measurability assumption
of $V(\cdot)$ in the $\cB(\cH)$-norm is missing in the basic set of
hypotheses of \cite{Sa71}). Our extension to $V(\cdot)$ weakly measurable and
$\|V(\cdot)\|_{\cB(\cH} \in L^1_{\loc}([a,b);dx)$ in \cite{GWZ13} may well be the most general one
published to date.

\begin{definition} \lb{d2.6}
Pick $c \in (a,b)$.
The endpoint $a$ (resp., $b$) of the interval $(a,b)$ is called {\it regular} for the operator-valued differential expression $- (d^2/dx^2) + Q(\cdot)$ if it is finite and if $Q$ is weakly measurable and $\|Q(\cdot)\|_{\cB(\cH)}\in  L^1_{\loc}([a,c];dx)$ (resp.,
$\|Q(\cdot)\|_{\cB(\cH)}\in  L^1_{\loc}([c,b];dx)$) for some $c\in (a,b)$. Similarly,
$- (d^2/dx^2) + Q(\cdot)$ is called {\it regular at $a$} (resp., {\it regular at $b$}) if
$a$ (resp., $b$) is a regular endpoint for $- (d^2/dx^2) + Q(\cdot)$.
\end{definition}

We note that if $a$ (resp., $b$) is regular for $- (d^2/dx^2) + Q(x)$, one may allow for
$x_0$ to be equal to $a$ (resp., $b$) in the existence and uniqueness Theorem \ref{t2.3}.

If $f_1, f_2$ are strongly continuously differentiable $\cH$-valued functions, we define the Wronskian of $f_1$ and $f_2$ by
\begin{equation}
W_{*}(f_1,f_2)(x)=(f_1(x),f'_2(x))_\cH - (f'_1(x),f_2(x))_\cH,    \lb{2.31A}
\quad x \in (a,b).
\end{equation}
If $f_2$ is an $\cH$-valued solution of $-y''+Qy=0$ and $f_1$ is an $\cH$-valued
solution of $-y''+Q^*y=0$, their Wronskian $W_{*}(f_1,f_2)(x)$ is $x$-independent, that is,
\begin{equation}
\f{d}{dx} W_{*}(f_1,f_2)(x) = 0, \, \text{ for a.e.\ $x \in (a,b)$.}   \lb{2.32A}
\end{equation}
Equation \eqref{2.52A} will show that the right-hand side of \eqref{2.32A} actually
vanishes for all $x \in (a,b)$.

We decided to use the symbol $W_{*}(\cdot,\cdot)$ in \eqref{2.31A} to indicate its
conjugate linear behavior with respect to its first entry.

Similarly, if $F_1,F_2$ are strongly continuously differentiable $\cB(\cH)$-valued
functions, their Wronskian is defined by
\begin{equation}
W(F_1,F_2)(x) = F_1(x) F'_2(x) - F'_1(x) F_2(x), \quad x \in (a,b).    \lb{2.33A}
\end{equation}
Again, if $F_2$ is a $\cB(\cH)$-valued solution of  $-Y''+QY = 0$ and $F_1$ is a
$\cB(\cH)$-valued solution of $-Y'' + Y Q = 0$ (the latter is equivalent to
$- {(Y^{*})}^{\prime\prime} + Q^* Y^* = 0$ and hence can be handled in complete analogy
via Theorem \ref{t2.3} and Corollary \ref{c2.5}, replacing $Q$ by $Q^*$) their Wronskian will be $x$-independent,
\begin{equation}
\f{d}{dx} W(F_1,F_2)(x) = 0 \, \text{ for a.e.\ $x \in (a,b)$.}
\end{equation}

Our main interest is in the case where $V(\cdot)=V(\cdot)^* \in \cB(\cH)$ is self-adjoint,
that is, in the differential equation $\tau \eta=z \eta$, where $\eta$ represents an $\cH$-valued, respectively, $\cB(\cH)$-valued solution (in the sense of Definitions \ref{d2.2},
resp., \ref{d2.4}), and where $\tau$ abbreviates the operator-valued differential expression
\begin{equation} \label{2.4}
\tau = - (d^2/dx^2) + V(\cdot).
\end{equation}
To this end, we now introduce the following basic assumption:

\begin{hypothesis} \lb{h2.7}
Let $(a,b)\subseteq\bbR$, suppose that $V:(a,b)\to\cB(\cH)$ is a weakly
measurable operator-valued function with $\|V(\cdot)\|_{\cB(\cH)}\in L^1_\loc((a,b);dx)$,
and assume that $V(x) = V(x)^*$ for a.e.\ $x \in (a,b)$.
\end{hypothesis}

Moreover, for the remainder of this section we assume that $\alpha \in \cB(\cH)$ is a
self-adjoint operator,
\begin{equation}
\alpha = \alpha^* \in \cB(\cH).      \lb{2.4A}
\end{equation}

Assuming Hypothesis \ref{h2.7} and \eqref{2.4A}, we introduce the standard fundamental systems of operator-valued solutions of $\tau y=zy$ as follows: Since $\alpha$ is a bounded self-adjoint operator, one may define the self-adjoint operators $A=\sin(\alpha)$ and $B=\cos(\alpha)$ via the spectral theorem. One then concludes that
$\sin^2(\alpha) + \cos^2(\alpha) = I_\cH$ and $[\sin\alpha,\cos\alpha]=0$ (here
$[\cdot,\cdot]$ represents the commutator symbol). The spectral theorem implies also
that the spectra of $\sin(\alpha)$ and $\cos(\alpha)$ are contained in $[-1,1]$ and that the spectra of $\sin^2(\alpha)$ and $\cos^2(\alpha)$ are contained in $[0,1]$. Given such an operator $\alpha$ and a point $x_0\in(a,b)$ or a regular endpoint for $\tau$, we now
define $\theta_\alpha(z,\cdot, x_0,), \phi_\alpha(z,\cdot,x_0)$ as those $\cB(\cH)$-valued
solutions of $\tau Y=z Y$ (in the sense of Definition \ref{d2.4}) which satisfy the initial
conditions
\begin{equation}
\theta_\alpha(z,x_0,x_0)=\phi'_\alpha(z,x_0,x_0)=\cos(\alpha), \quad
-\phi_\alpha(z,x_0,x_0)=\theta'_\alpha(z,x_0,x_0)=\sin(\alpha).    \lb{2.5}
\end{equation}

By Corollary 2.5\,$(iii)$, for any fixed $x, x_0\in(a,b)$, the functions
$\theta_{\alpha}(z,x,x_0)$ and $\phi_{\alpha}(z,x,x_0)$ as well as their strong $x$-derivatives are entire with respect to $z$ in the $\cB(\cH)$-norm. The same is true for the functions $z\mapsto\theta_{\alpha}(\ol{z},x,x_0)^*$ and
$z\mapsto\phi_{\alpha}(\ol{z},x,x_0)^*$.

Since $\theta_{\alpha}(\bar z,\cdot,x_0)^*$ and $\phi_{\alpha}(\bar z,\cdot,x_0)^*$ satisfy
the adjoint equation $-Y''+YV=z Y$ and the same initial conditions as $\theta_\alpha$ and
$\phi_\alpha$, respectively, one obtains the following identities from the constancy of Wronskians:
\begin{align}
\theta_{\alpha}' (\bar z,x,x_0)^*\theta_{\alpha} (z,x,x_0)-
\theta_{\alpha} (\bar z,x,x_0)^*\theta_{\alpha}' (z,x,x_0)&=0, \label{2.7f}
\\
\phi_{\alpha}' (\bar z,x,x_0)^*\phi_{\alpha} (z,x,x_0)-
\phi_{\alpha} (\bar z,x,x_0)^*\phi_{\alpha}' (z,x,x_0)&=0, \label{2.7g}
\\
\phi_{\alpha}' (\bar z,x,x_0)^*\theta_{\alpha} (z,x,x_0)-
\phi_{\alpha} (\bar z,x,x_0)^*\theta_{\alpha}' (z,x,x_0)&=I_{\cH}, \label{2.7h}
\\
\theta_{\alpha} (\bar z,x,x_0)^*\phi_{\alpha}' (z,x,x_0)
- \theta_{\alpha}' (\bar z,x,x_0)^*\phi_{\alpha} (z,x,x_0)&=I_{\cH}. \label{2.7i}
\end{align}
Equations \eqref{2.7f}--\eqref{2.7i} are equivalent to the statement that the block operator
\begin{equation}
\Theta_{\alpha}(z,x,x_0)=\begin{pmatrix}\theta_{\alpha}(z,x,x_0)&\phi_{\alpha}(z,x,x_0)\\ \theta_{\alpha}'(z,x,x_0)&\phi_{\alpha}'(z,x,x_0) \end{pmatrix}    \label{2.7ia}
\end{equation}
has a left inverse given by
\begin{equation}
\begin{pmatrix}\phi_{\alpha}'(\bar z,x,x_0)^*&-\phi_{\alpha}(\bar z,x,x_0)^*\\
-\theta_{\alpha}'(\bar z,x,x_0)^*&\theta_{\alpha}(\bar z,x,x_0)^*
\end{pmatrix}.    \label{2.7ib}
\end{equation}
Thus the operator $\Theta_{\alpha}(z,x,x_0)$ is injective. It is also surjective as will be shown next: Let $(f_1,g_1)^\top$ be an arbitrary element of $\cH\oplus\cH$ and let $y$ be an $\cH$-valued solution of the initial value problem
\begin{equation}
\begin{cases} \tau y=zy, \\ y(x_1)=f_1, \; y'(x_1)=g_1, \end{cases}
\end{equation}
for some given $x_1\in(a,b)$. One notes that due to the initial conditions specified in
\eqref{2.5}, $\Theta_{\alpha}(z,x_0,x_0)$ is bijective. We now assume that $(f_0,g_0)^\top$ are given by
\begin{equation}
\Theta_{\alpha}(z,x_0,x_0)\begin{pmatrix}f_0\\ g_0\end{pmatrix}=\begin{pmatrix}y(x_0)\\ y'(x_0)\end{pmatrix}.   \label{2.7ic}
\end{equation}
The existence and uniqueness Theorem \ref{t2.3} then yields that
\begin{equation}
\Theta_{\alpha}(z,x_1,x_0)\begin{pmatrix}f_0\\ g_0\end{pmatrix}=\begin{pmatrix}f_1\\ g_1\end{pmatrix}.    \label{2.7id}
\end{equation}
This establishes surjectivity of $\Theta_{\alpha}(z,x_1,x_0)$ which therefore has a right inverse too, also given by \eqref{2.7ib}. This fact then implies the following identities:
\begin{align}
\phi_{\alpha} (z,x,x_0)\theta_{\alpha} (\bar z,x,x_0)^*-
\theta_{\alpha} (z,x,x_0)\phi_{\alpha} (\bar z,x,x_0)^*&=0, \label{2.7j}
\\
\phi_{\alpha}' (z,x,x_0)\theta_{\alpha}' (\bar z,x,x_0)^*-
\theta_{\alpha}' (z,x,x_0)\phi_{\alpha}' (\bar z,x,x_0)^*&=0, \label{2.7k}
\\
\phi_{\alpha}' (z,x,x_0)\theta_{\alpha} (\bar z,x,x_0)^*-
\theta_{\alpha}' (z,x,x_0)\phi_{\alpha} (\bar z,x,x_0)^*&=I_{\cH}, \label{2.7l}
\\
\theta_{\alpha} (z,x,x_0)\phi_{\alpha}' (\bar z,x,x_0)^*-
\phi_{\alpha} (z,x,x_0)\theta_{\alpha}' (\bar z,x,x_0)^*&=I_{\cH}. \label{2.7m}
\end{align}

Having established the invertibility of $\Theta_\alpha(z,x_1,x_0)$ we can now show that
for any $x_1\in(a,b)$, any $\cH$-valued solution of $\tau y=zy$ may be expressed in terms of
$\theta_{\alpha}(z,\cdot,x_1)$ and $\phi_{\alpha}(z,\cdot,x_1)$, that is,
\begin{equation}
y(x)=\theta_\alpha(z,x,x_1)f+\phi_\alpha(z,x,x_1)g
\end{equation}
for appropriate vectors $f,g\in\cH$ or $\cB(\cH)$.

We also recall several versions of Green's formula (also called Lagrange's identity).

\begin{lemma} \label{l2.9}
Let $(a,b)\subseteq\bbR$ be a finite or infinite interval and $[x_1,x_2]\subset(a,b)$. \\
$(i)$ Assume that $f,g\in W^{2,1}_{\rm loc}((a,b);dx;\cH)$. Then
\begin{equation}
\int_{x_1}^{x_2} dx \, [((\tau f)(x),g(x))_\cH-(f(x),(\tau g)(x))_\cH]
= W_{*}(f,g)(x_2)-W_{*}(f,g)(x_1).     \lb{2.52A}
\end{equation}
$(ii)$ Assume that $F:(a,b)\to\cB(\cH)$ is absolutely continuous, that $F'$ is again differentiable, and that $F''$ is weakly measurable. Also assume that $\|F''\|_\cH \in L^1_\loc((a,b);dx)$ and $g\in W^{2,1}_{\rm loc}((a,b);dx;\cH)$. Then
\begin{equation}
\int_{x_1}^{x_2} dx \, [(\tau F^*)^*(x)g(x)-F(x)(\tau g)(x)]
= (Fg'-F'g)(x_2)-(Fg'-F'g)(x_1).     \lb{2.52B}
\end{equation}
$(iii)$ Assume that $F,\,G:(a,b)\to\cB(\cH)$ are absolutely continuous operator-valued functions such that $F',\,G'$ are again differentiable and that $F''$, $G''$ are weakly measurable. In addition, suppose that $\|F''\|_\cH,\, \|G''\|_\cH \in L^1_\loc((a,b);dx)$. Then
\begin{equation}
\int_{x_1}^{x_2} dx \, [(\tau F^*)(x)^*G(x) - F(x) (\tau G)(x)] = (FG'-F'G)(x_2)-(FG'-F'G)(x_1).
\lb{2.53A}
\end{equation}
\end{lemma}

\section{Half-Line Weyl--Titchmarsh Theory for Schr\"odinger Operators
with Operator-Valued Potentials Revisited} \label{s3}

In this section we recall the basics of Weyl--Titchmarsh theory for self-adjoint Schr\"odinger
operators $H_{\alpha}$ in $L^2((a,b); dx; \cH)$ associated with the operator-valued differential
expression $\tau =-(d^2/dx^2)+V(\cdot)$, assuming regularity of the
left endpoint $a$ and the limit point case at the right endpoint $b$ (see
Definition \ref{d3.6}). We discuss the existence of Weyl--Titchmarsh solutions, introduce
the corresponding Weyl--Titchmarsh $m$-function, and determine the structure of the Green's
function of $H_{\alpha}$.

All results recalled in this section were proved in detail in \cite{GWZ13}.

As before, $\cH$ denotes a separable Hilbert space and $(a,b)$ denotes a finite or infinite interval. One recalls that $L^2((a,b);dx;\cH)$ is separable (since $\cH$ is)
and that
\begin{equation}
(f,g)_{L^2((a,b);dx;\cH)} =\int_a^b dx \, (f(x),g(x))_\cH, \quad f,g\in L^2((a,b);dx;\cH).
\end{equation}

Assuming Hypothesis \ref{h2.7} throughout this section, we are interested in
studying certain self-adjoint operators in $L^2((a,b);dx;\cH)$ associated with the
operator-valued differential expression $\tau =-(d^2/dx^2)+V(\cdot)$. These will be suitable restrictions of the {\it maximal} operator $\oT_{\max}$ in $L^2((a,b);dx;\cH)$ defined by
\begin{align}
& \oT_{\max} f = \tau f,   \no \\
& f\in \dom(\oT_{\max})=\big\{g\in L^2((a,b);dx;\cH) \,\big|\, g\in W^{2,1}_{\rm loc}((a,b);dx;\cH); \\
& \hspace*{6.6cm} \tau g\in L^2((a,b);dx;\cH)\big\}.     \no
\end{align}
We also introduce the operator $\dot \oT_{\min}$ in $L^2((a,b);dx;\cH)$ as the restriction of $\oT_{\max}$ to the domain
\begin{equation}
\dom(\dot \oT_{\min})=\{g\in\dom(\oT_{\max})\,|\,\supp (u) \, \text{is compact in} \, (a,b)\}.
\end{equation}
Finally, the {\it minimal} operator $\oT_{\min}$ in $L^2((a,b);dx;\cH)$ associated with $\tau$ is then defined as the closure of $\dot \oT_{\min}$,
\begin{equation}
\oT_{\min} = \ol{\dot \oT_{\min}}.
\end{equation}

Next, we intend to show that $\oT_{\max}$ is the adjoint of $\dot \oT_{\min}$ (and hence that of $\oT_{\min}$), implying, in particular, that $\oT_{\max}$ is closed. To this end, we first establish the following two preparatory lemmas for the case where $a$ and $b$ are both regular endpoints for $\tau$ in the sense of Definition \ref{d2.6}.

\begin{lemma} \label{l3.1}
In addition to Hypothesis \ref{h2.7} suppose that $a$ and $b$ are regular endpoints for $\tau$. Then
\begin{align} \label{3.5A}
\begin{split}
& \ker(\oT_{\max}-z I_{L^2((a,b);dx;\cH)})  \\
& \quad =\{[\theta_0(z,\cdot,a)f+\phi_0(z,\cdot,a)g]
\in L^2((a,b);dx;\cH) \,|\, f,g\in\cH\}
\end{split}
\end{align}
is a closed subspace of $L^2((a,b);dx;\cH)$.
\end{lemma}

Of course, if $\cH$ is finite-dimensional (e.g., in the scalar case, $\dim(\cH)=1$), then
$\ker(\oT_{\max}-z I_{L^2((a,b);dx;\cH)})$ is finite-dimensional and hence automatically closed.

\begin{lemma} \label{l3.3}
In addition to Hypothesis \ref{h2.7} suppose that $a$ and $b$ are regular endpoints for $\tau$. Denote by $\oT_0$ the linear operator in $L^2((a,b);dx;\cH)$
defined by the restriction of
$\oT_{\max}$ to the space
\begin{equation}
\dom(\oT_0)=\{g\in\dom(\oT_{\max}) \,|\, g(a)=g(b)=g'(a)=g'(b)=0\}.
\lb{3.9a}
\end{equation}
Then
\begin{equation}
\ker(\oT_{\max})=[\ran(\oT_0)]^\perp,
\end{equation}
that is, the space of solutions $u$ of $\tau u=0$ coincides with the orthogonal complement of the collection of elements $\tau u_0$ satisfying $u_0\in \dom(\oT_0)$.
\end{lemma}

\begin{theorem} \label {t3.4}
Assume Hypothesis \ref{h2.7}. Then the operator $\dot \oT_{\min}$ is densely defined. Moreover, $\oT_{\max}$ is the adjoint of $\dot \oT_{\min}$,
\begin{equation}
\oT_{\max} = (\dot \oT_{\min})^*.   \lb{3.12a}
\end{equation}
In particular, $\oT_{\max}$ is closed. In addition, $\dot \oT_{\min}$ is symmetric and
$\oT_{\max}^*$ is the closure of $\dot \oT_{\min}$, that is,
\begin{equation}
\oT_{\max}^* = \ol{\dot \oT_{\min}} = \oT_{\min}.    \lb{3.13a}
\end{equation}
\end{theorem}

Lemmas \ref{l3.1}, \ref{l3.3}, and Theorem \ref{t3.4}, under additional hypotheses on $V$ (typically involving continuity assumptions) are of course well-known and go back to
Rofe-Beketov \cite{Ro69}, \cite{Ro69a} (see also \cite[Sect.\ 3.4]{GG91}, \cite[Ch.\ 5]{RK05}).

In the special case where $a$ and $b$ are regular endpoints for $\tau$, the operator $H_0$ introduced in \eqref{3.9a} coincides with the minimal
operator $\oT_{\min}$.

Using the dominated convergence theorem and Green's formula \eqref{2.52A} one can show that $\lim_{x\to a}W_*(u,v)(x)$ and $\lim_{x\to b}W_*(u,v)(x)$ both exist whenever
$u,v\in\dom(\oT_{\max})$. We will denote these limits by $W_*(u,v)(a)$ and $W_*(u,v)(b)$, respectively. Thus Green's formula also holds for $x_1=a$ and $x_2=b$ if $u$ and $v$ are in $\dom(\oT_{\max})$, that is,
\begin{equation}\label{3.17A}
(\oT_{\max}u,v)_{L^2((a,b);dx;\cH)}-(u,\oT_{\max}v)_{L^2((a,b);dx;\cH)}
= W_*(u,v)(b) - W_*(u,v)(a).
\end{equation}
This relation and the fact that $\oT_{\min}=\oT_{\max}^*$ is a restriction of $\oT_{\max}$ show that
\begin{align}
\begin{split}
& \dom(\oT_{\min})=\{u\in\dom(\oT_{\max}) \,|\, W_*(u,v)(b)=W_*(u,v)(a)=0   \\
& \hspace*{6.2cm} \text{ for all } v\in \dom(\oT_{\max})\}.     \label{3.18A}
\end{split}
\end{align}

\begin{definition} \lb{d3.6}
Assume Hypothesis \ref{h2.7}.
Then the endpoint $a$ (resp., $b$) is said to be of {\it limit-point type for $\tau$} if
$W_*(u,v)(a)=0$ (resp., $W_*(u,v)(b)=0$) for all $u,v\in\dom(\oT_{\max})$.
\end{definition}

Next, we introduce the subspaces
\begin{equation}
\cD_{z}=\{u\in\dom(\oT_{\max}) \,|\, \oT_{\max}u=z u\}, \quad z \in \bbC.
\end{equation}
For $z\in\bbC\backslash\bbR$, $\cD_{z}$ represent the deficiency subspaces of
$\oT_{\min}$. Von Neumann's theory of extensions of symmetric operators implies that
\begin{equation} \label{3.20A}
\dom(\oT_{\max})=\dom(\oT_{\min}) \dotplus \cD_i \dotplus \cD_{-i}
\end{equation}
where $\dotplus$ indicates the direct (but not necessarily orthogonal direct) sum.

We now set out to determine the self-adjoint restrictions of $\oT_{\max}$ assuming
that $a$ is a regular endpoint for $\tau$ and $b$ is of limit-point type for $\tau$.

\begin{hypothesis} \lb{h3.9}
In addition to Hypothesis \ref{h2.7} suppose that $a$ is a regular
endpoint for $\tau$ and $b$ is of limit-point type for $\tau$.
\end{hypothesis}

\begin{theorem} \lb{t3.10}
Assume Hypothesis \ref{h3.9}. If $\oT$ is a self-adjoint restriction of
$\oT_{\max}$, then there is a bounded and self-adjoint operator $\alpha\in\cB(\cH)$
such that
\begin{equation}
\dom(\oT)=\{u\in\dom(\oT_{\max}) \,|\, \sin(\alpha)u'(a) + \cos(\alpha)u(a)=0\}.  \lb{3.29A}
\end{equation}
Conversely, for every $\alpha \in \cB(\cH)$, \eqref{3.29A} gives rise to a self-adjoint restriction of $\oT_{\max}$ in $L^2((a,b);dx;\cH)$.
\end{theorem}

Henceforth, under the assumptions of Theorem \ref{t3.10}, we denote the operator $\oT$ in $L^2((a,b);dx;\cH)$ associated with the boundary condition induced by $\alpha = \alpha^* \in \cB(\cH)$, that is, the restriction of $\oT_{\max}$ to the set
\begin{equation}
\dom(H_{\alpha})=\{u\in\dom(\oT_{\max}) \,|\, \sin(\alpha)u'(a)+\cos(\alpha)u(a)=0\}
\end{equation}
by $H_{\alpha}$. For a discussion of boundary conditions at infinity, see, for instance,
\cite{MN11}, \cite{Mo09}, and \cite{RK85}.

Our next goal is to construct the square integrable solutions $Y(z,\cdot) \in\cB(\cH)$
of $\tau Y=zY$, $z\in\bbC\backslash\bbR$, the $\cB(\cH)$-valued Weyl--Titchmarsh solutions, under the assumptions that $a$ is a regular endpoint for $\tau$ and $b$ is of limit-point type for $\tau$.

Fix $c\in(a,b)$ and $z \in \rho(H_{\alpha})$. For any $f_0\in\cH$ let $f=f_0\chi_{[a,c]}\in L^2((a,b);dx;\cH)$ and $u(f_0,z,\cdot)=(H_{\alpha} - z I_{L^2((a,b);dx;\cH)})^{-1} f \in\dom(H_{\alpha})$. By the variation of constants formula,
\begin{align}
\begin{split}
u(f_0,z,x) &=\theta_{\alpha}(z,x,a)\bigg(g(z)+\int_x^c dx' \, \phi_{\alpha}(\ol z,x',a)^* f_0\bigg)\\
& \quad  +\phi_{\alpha}(z,x,a)\bigg(h(z)-\int_x^c dx' \, \theta_{\alpha}(\ol z,x',a)^* f_0\bigg)
 \end{split}
\end{align}
for suitable vectors $g(z) \in \cH$, $h(z) \in \cH$. Since
$u(f_0,z,\cdot)\in\dom(H_{\alpha})$, one infers that
\begin{equation} \label{3.40A}
g(z)=-\int_a^c dx' \, \phi_{\alpha}(\ol z,x',a)^* f_0, \quad z \in \rho(H_{\alpha}),
\end{equation}
and that
\begin{equation} \lb{3.40B}
h(z)=\cos(\alpha)u'(f_0,z,a) - \sin(\alpha)u(f_0,z,a)
+ \int_a^c dx' \, \theta_{\alpha}(\ol z,x',a)^* f_0, \quad z \in \rho(H_{\alpha}).
\end{equation}

\begin{lemma} \lb{l3.13}
Assume Hypothesis \ref{h3.9} and suppose that $\alpha \in \cB(\cH)$ is self-adjoint. In addition, choose $c \in (a,b)$ and introduce $g(\cdot)$ and $h(\cdot)$ as in \eqref{3.40A} and \eqref{3.40B}. Then the maps
\begin{equation}
C_{1,\alpha}(c,z):\begin{cases} \cH\to\cH, \\
f_0\mapsto g(z), \end{cases} \quad
C_{2,\alpha}(c,z): \begin{cases} \cH\to\cH, \\
f_0\mapsto h(z), \end{cases}  \quad z \in \rho(H_{\alpha}),
\end{equation}
are linear and bounded. Moreover, $C_{1,\alpha}(c,\cdot)$ is entire and
$C_{2,\alpha}(c,\cdot)$ is analytic on $\rho(H_{\alpha})$. In addition,
$C_{1,\alpha}(c,z)$ is boundedly invertible if $z\in \bbC\backslash\bbR$ and $c$
is chosen appropriately.
\end{lemma}

Using the bounded invertibility of $C_{1,\alpha}(c,z)$ we now define
\begin{equation}
\psi_{\alpha}(z,x)=\theta_{\alpha}(z,x,a)
+ \phi_{\alpha}(z,x,a)C_{2,\alpha}(c,z)C_{1,\alpha}(c,z)^{-1},
\quad z \in \bbC\backslash\bbR, \; x \in [a,b),     \lb{3.49A}
\end{equation}
still assuming Hypothesis \ref{h3.9} and $\alpha = \alpha^* \in \cB(\cH)$. By Lemma \ref{l3.13}, $\psi_{\alpha}(\cdot,x)$ is analytic on
$z \in \bbC\backslash\bbR$ for fixed $x \in [a,b]$.

Since $\psi_{\alpha}(z,\cdot) f_0$ is the solution of the initial value problem
\begin{equation}
\tau y =z y, \quad y(c)=u(f_0,z,c), \; y'(c)=u'(f_0,z,c), \quad z \in \bbC\backslash\bbR,
\end{equation}
the function $\psi_{\alpha}(z,x)C_{1,\alpha}(z,c)f_0$ equals
$u(f_0,z,x)$ for $x\geq c$, and thus is square integrable for every choice of $f_0\in\cH$.
In particular, choosing $c \in (a,b)$ such that $C_{1,\alpha}(z,c)^{-1} \in \cB(\cH)$, one infers that
\begin{equation}
\int_a^b dx \, \|\psi_{\alpha}(z,x) f\|_{\cH}^2 < \infty, \quad
f \in \cH, \; z \in \bbC\backslash\bbR.
\end{equation}

Every $\cH$-valued solution of $\tau y=z y$ may be written as
\begin{equation}
y=\theta_\alpha(z,\cdot,a)f_{\alpha,a} + \phi_\alpha(z,\cdot,a)g_{\alpha,a},
\end{equation}
with
\begin{equation}
f_{\alpha,a}=(\cos\alpha)y(a)+(\sin\alpha)y'(a), \quad
g_{\alpha,a}=-(\sin\alpha)y(a)+(\cos\alpha)y'(a).
\end{equation}
Hence we can define the maps
\begin{align}
& \mc C_{1,\alpha,z}:\begin{cases} \cD_z\to\cH, \\
\theta_\alpha(z,\cdot,a) f_{\alpha,a}
+ \phi_\alpha(z,\cdot,a) g_{\alpha,a} \mapsto f_{\alpha,a}, \end{cases} \\
& \mc C_{2,\alpha,z}: \begin{cases} \cD_z\to\cH, \\
\theta_\alpha(z,\cdot,a) f_{\alpha,a} + \phi_\alpha(z,\cdot,a) g_{\alpha,a} \mapsto g_{\alpha,a}. \end{cases}
\end{align}

\begin{lemma} \label{l3.14}
Assume Hypothesis \ref{h3.9}, suppose that $\alpha \in \cB(\cH)$ is self-adjoint, and
let $z\in\bbC\backslash\bbR$. Then the operators
$\mc C_{1,\alpha,z}$ and $\mc C_{2,\alpha,z}$ are linear bijections and hence
\begin{equation}
\mc C_{1,\alpha,z}, \, \mc C_{1,\alpha,z}^{-1}, \, \mc C_{2,\alpha,z}, \,
\mc C_{2,\alpha,z}^{-1} \in \cB(\cH).   \lb{3.57}
\end{equation}
\end{lemma}

At this point we are finally in the position to define the
Weyl--Titchmarsh $m$-function for $z\in\bbC\backslash\bbR$ by setting
\begin{equation} \label{3.57A}
m_{\alpha}(z)=\mc C_{2,\alpha,z}\mc C_{1,\alpha,z}^{-1}, \quad
z\in\bbC\backslash\bbR.
\end{equation}

\begin{theorem} \label{t3.15}
Assume Hypothesis \ref{h3.9} and that $\alpha \in \cB(\cH)$ is self-adjoint. Then
\begin{equation}
m_{\alpha}(z) \in \cB(\cH), \quad z\in\bbC\backslash\bbR,   \lb{3.57B}
\end{equation}
and $m_{\alpha}(\cdot)$ is analytic on $\bbC\backslash\bbR$. Moreover,
\begin{equation}
m_{\alpha}(z)=m_{\alpha}(\ol z)^*, \quad z \in \bbC\backslash\bbR.    \lb{3.59A}
\end{equation}
\end{theorem}

Thus, the $\cB(\cH)$-valued function $\psi_{\alpha}(z,\cdot)$
in \eqref{3.49A} can be rewritten in the form
\begin{equation} \label{3.58A}
\psi_{\alpha}(z,x)=\theta_{\alpha}(z,x,a)+\phi_{\alpha}(z,x,a)m_{\alpha}(z),
\quad z \in \bbC\backslash\bbR, \; x \in [a,b).
\end{equation}
In particular, this implies that $\psi_{\alpha}(z,\cdot)$ is independent of
the choice of the parameter $c \in (a,b)$ in \eqref{3.49A}.
Following the tradition in the scalar case ($\dim(\cH) = 1$), we will call
$\psi_{\alpha}(z,\cdot)$ the {\it Weyl--Titchmarsh} solution associated
with $\tau Y = z Y$.

We remark that, given a function $u\in\cD_z$, the operator $m_{0}(z)$ assigns the Neumann boundary data $u'(a)$ to the Dirichlet boundary data
$u(a)$, that is, $m_{0}(z)$ is the ($z$-dependent)
Dirichlet-to-Neumann map.

With the aid of the Weyl--Titchmarsh solutions we can now give a detailed description of the resolvent
$R_{z,\alpha} = (H_{\alpha} - z I_{L^2((a,b);dx;\cH)})^{-1}$ of $H_{\alpha}$.

\begin{theorem}\label{t3.16}
Assume Hypothesis \ref{h3.9} and that $\alpha \in \cB(\cH)$ is self-adjoint. Then the resolvent of $H_{\alpha}$ is an integral operator of the type
\begin{align}
\begin{split}
\big((H_{\alpha} - z I_{L^2((a,b);dx;\cH)})^{-1} u\big)(x)
= \int_a^b dx' \, G_{\alpha}(z,x,x')u(x'),& \\
u \in L^2((a,b);dx;\cH), \; z \in \rho(H_{\alpha}), \; x \in [a,b),&
\end{split}
\end{align}
with the $\cB(\cH)$-valued Green's function $G_{\alpha}(z,\cdot,\cdot) $ given by
\begin{equation} \label{3.63A}
G_{\alpha}(z,x,x') = \begin{cases}
\phi_{\alpha}(z,x,a) \psi_{\alpha}(\ol{z},x')^*, & a\leq x \leq x'<b, \\
\psi_{\alpha}(z,x) \phi_{\alpha}(\ol{z},x',a)^*, & a\leq x' \leq x<b,
\end{cases}  \quad z\in\bbC\backslash\bbR.
\end{equation}
\end{theorem}

One recalls from Definition \ref{dA.4} that a nonconstant function
$N:\bb C_+\to\cB(\cH)$ is called a (bounded) operator-valued Herglotz function, if
$z\mapsto (u,N(z)u)_{\cH}$ is analytic and
has a non-negative imaginary part for all $u\in\cH$.

\begin{theorem} \label{t3.17}
Assume Hypothesis \ref{h3.9} and suppose that $\alpha \in \cB(\cH)$ and
$\beta \in \cB(\cH)$ are self-adjoint.
Then the $\cB(\cH)$-valued function $m_{\alpha}(\cdot)$
is an operator-valued Herglotz function and explicitly determined by the Green's function
for $H_{\alpha}$ as follows,
\begin{align}
& m_{\alpha}(z) = \big(-\sin(\alpha), \cos(\alpha)\big)
\begin{pmatrix} G_{\alpha}(z,a,a) & G_{\alpha,x'} (z,a,a) \\
G_{\alpha, x} (z,a,a) & G_{\alpha, x,x'} (z,a,a) \end{pmatrix}
\begin{pmatrix} -\sin(\alpha) \\ \cos(\alpha) \end{pmatrix},    \no \\
& \hspace*{9.5cm}   z\in\bbC\backslash\bbR,     \label{2.19a}
\end{align}
where we denoted
\begin{align}
G_{\alpha,x} (z,a,a) &= \slim_{\substack{x' \to a \\ a<x<x'}}
\frac{\partial}{\partial x} G_{\alpha} (z,x,x'),   \no   \\
G_{\alpha,x'} (z,a,a) &= \slim_{\substack{x' \to a \\ a<x<x'}}
\frac{\partial}{\partial x'} G_{\alpha} (z,x,x'),     \\
G_{\alpha,x,x'} (z,a,a) &= \slim_{\substack{x' \to a \\ a<x<x'}}
\frac{\partial}{\partial x} \frac{\partial}{\partial x'} G_{\alpha} (z,x,x')    \no
\end{align}
$($the strong limits referring to the strong operator topology in $\cH$$)$.
In addition, $m_{\alpha}(\cdot)$ extends analytically to the resolvent set
of $H_{\alpha}$.

Moreover, $m_{\alpha}(\cdot)$ and $m_{\beta}(\cdot) $ are related by the
following linear fractional transformation,
\begin{equation}\label{3.67A}
m_{\beta}=(C+Dm_{\alpha})(A+Bm_{\alpha})^{-1},
\end{equation}
where
\begin{equation}
\begin{pmatrix}A&B\\ C&D\end{pmatrix}
= \begin{pmatrix}\cos(\beta) & \sin(\beta) \\ -\sin(\beta) & \cos(\beta) \end{pmatrix}
\begin{pmatrix}\cos(\alpha) & -\sin(\alpha) \\ \sin(\alpha) & \cos(\alpha) \end{pmatrix}.
\end{equation}
\end{theorem}

We also mention that $G_{\alpha}(\cdot,x,x)$ is a bounded Herglotz operator
in $\cH$ for each $x\in (a,b)$, as is clear from \eqref{2.7j}, \eqref{3.58A},
\eqref{3.63A}, and the Herglotz property of $m_{\alpha}$.

\begin{remark} \lb{r3.19}
The Weyl--Titchmarsh theory established in this section is modeled after
right half-lines $(a,b) = (0, \infty)$. Of course precisely the analogous theory applies
to left half-lines $(-\infty,0)$. Given the two half-line results, one then establishes
the full-line result on $\bbR$ in the usual fashion with $x=0$ a reference point and a
$2 \times 2$ block operator formalism as in the well-known scalar or
matrix-valued cases; we omit further details at this point as the basic results will
explicitly be derived in Section \ref{s5}.
\end{remark}

\section{Spectral Theory of Schr\"odinger Operators with Operator-Valued
Potentials on the Half-Line}  \lb{s4}

In this section we develop the basic spectral theory for Schr\"odinger operators $H_{+,\alpha}$
in $L^2((a,\infty); dx; \cH)$ on right a half-line $(a,\infty)$ with a bounded operator-valued
potential coefficient  in some complex, separable Hilbert space $\cH$, and with a regular left
endpoint $a$. We focus on a diagonalization of $H_{+,\alpha}$ and the corresponding
generalized eigenfunction expansion, including a description of support properties of the
underlying $\cB(\cH)$-valued half-line spectral measure. In particular, we illustrate the
spectral theorem for $F(H_{+,\alpha})$, $F \in C(\bbR)$ (cf.\ Theorems \ref{t2.5} and \ref{t2.6}).

In the special scalar and matrix-valued cases where $\dim(\cH)<\infty$, the material of this
section is standard. In particular, we refer to
\cite{Be08}, \cite{BE05}, \cite[Ch.\ 9]{CL85}, \cite[Sect.\ XIII.5]{DS88},
\cite[Ch.\ 2]{EK82}, \cite[Sect.\ III.10]{EE89}, \cite{Ev04}, \cite[Ch.\ 10]{Hi69}, \cite{HS98},
\cite{Ko49}, \cite{Le51},
\cite[Ch.\ 2]{LS75}, \cite[Ch.\ VI]{Na68}, \cite[Ch.\ 6]{Pe88}, \cite[Chs.\ II, III]{Ti62},
\cite[Ch.\ 8]{We80}, \cite[Sects.\ 7--10]{We87}, in the scalar case (i.e., for $\dim(\cH)=1$)
and to \cite{CG02}, \cite[Ch.\ 1, Appendix A]{RK05}, in the matrix-valued case (i.e., for
$\dim(\cH)<\infty$). While there exist a variety of results in the operator-valued case (i.e., for $\dim(\cH)=\infty$), \cite{Go68}, \cite[Chs.\ 3, 4]{GG91},
\cite[Sect.\ 10.7]{Hi69}, \cite{Mi76}, \cite{Mi83}, \cite{Mi83a},
\cite{Ro60}, \cite[Ch.\ 2]{RK05}, \cite{Sa71}, \cite{Sa71a}, \cite{Sa72}, \cite{Sa77}, \cite{Sa79},
\cite{Tr00}, \cite{VG70}, typically, under varying regularity hypotheses on $V(\cdot)$, we
emphasize that under our general Hypothesis \ref{h2.7}, the results obtained in this section are new.

We start with the following useful result, a version of Stone's formula in the weak sense (cf.,
e.g., \cite[p.\ 1203]{DS88}).

\begin{lemma} \lb{l2.4a}
Let $T$ be a self-adjoint operator in a complex separable Hilbert space
$\cK$ $($with inner product denoted by $(\cdot,\cdot)_\cK$, linear in the
second factor$)$ and denote by $\{E_T(\lambda)\}_{\lambda\in\bbR}$ the
family of self-adjoint right-continuous spectral projections associated
with $T$, that is, $E_T(\lambda)=\chi_{(-\infty,\lambda]}(T)$,
$\lambda\in\bbR$. Moreover, let $f,g \in\cK$, $\lambda_1,\lambda_2\in\bbR$,
$\lambda_1<\lambda_2$, and $F\in C(\bbR)$. Then,
\begin{align}
&\big(f,F(T)E_{T}((\lambda_1,\lambda_2])g\big)_{\cK} \no \\
& \quad = \lim_{\delta\downarrow 0}\lim_{\varepsilon\downarrow 0}
\frac{1}{2\pi i}
\int_{\lambda_1+\delta}^{\lambda_2+\delta} d\lambda \, F(\lambda)
\big[\big(f,(T-(\lambda+i\varepsilon) I_{\cK})^{-1}g\big)_{\cK}  \no \\
& \hspace*{4.9cm} - \big(f,(T-(\lambda-i\varepsilon)I_{\cK})^{-1}
g\big)_{\cK}\big]. \lb{2.26a}
\end{align}
\end{lemma}
\begin{proof}
First, assume $F\geq 0$. Then
\begin{align}
& \big(F(T)^{1/2}E_{T}((\lambda_1,\lambda_2])f, (T-z I_{\cK})^{-1}
F(T)^{1/2}E_{T}((\lambda_1,\lambda_2])f\big)_{\cK} \no \\
& \quad = \int_\bbR d\big(F(T)^{1/2}f,E_T(\lambda) F(T)^{1/2}f\big)_{\cK} \,
\chi_{(\lambda_1,\lambda_2]}(\lambda)(\lambda-z)^{-1} \no \\
& \quad = \int_\bbR d\big(f,E_T(\lambda)f\big)_{\cK} \, F(\lambda)
\chi_{(\lambda_1,\lambda_2]}(\lambda)(\lambda-z)^{-1} \no \\
& \quad = \int_\bbR \f{d\big(F(T)^{1/2}
\chi_{(\lambda_1,\lambda_2]}(T)f,E_T(\lambda)
F(T)^{1/2}\chi_{(\lambda_1,\lambda_2]}(T)f\big)_{\cK}}{(\lambda-z)},
\quad z\in\bbC_+,  \lb{2.26b}
\end{align}
is a Herglotz function and hence \eqref{2.26a} for $g=f$ follows from
the standard Stieltjes inversion formula in the scalar case. If $F$ is not
nonnegative, one decomposes $F$ as
$F=(F_1-F_2)+i(F_3-F_4)$ with $F_j\geq 0$, $1\leq j \leq 4$ and
applies \eqref{2.26b} to each $j\in\{1,2,3,4\}$. The general case
$g\neq f$ then follows from the case $g=f$ by polarization.
\end{proof}

Next, we replace the interval $(a,b)$ in Sections \ref{s2} and \ref{s3} by the right half-line
$(a,\infty)$ and indicate this change with the additional subscript $+$ in $H_{+,\alpha}$,
$m_{+,\alpha}(\cdot)$, $d\rho_{+,\alpha}(\cdot)$, etc., to distinguish these quantities
from the analogous objects on the left half-line $(-\infty, a)$ (later indicated with the
subscript $-$), which are needed in our subsequent Section \ref{s5}.

Our aim is to relate the family of spectral projections,
$\{E_{H_{+,\alpha}}(\lambda)\}_{\lambda\in\bbR}$, of the self-adjoint
operator $H_{+,\alpha}$ and the $\cB(\cH)$-valued spectral function
$\rho_{+,\alpha}(\lambda)$, $\lambda\in\bbR$, which generates the
operator-valued measure $d\rho_{+,\alpha}$ in the Herglotz representation
\eqref{2.25} of $m_{+,\alpha}$:
\begin{equation}
m_{+,\alpha}(z) = c_{+,\alpha}
+ \int_{\bbR}  d\rho_{+,\alpha}(\lambda) \Big[
\frac{1}{\lambda-z} -\frac{\lambda}{1+\lambda^2}\Big],  \quad
z\in\bbC\backslash \sigma(H_{+,\alpha}),   \lb{2.25}
\end{equation}
where
\begin{equation}
c_{+,\alpha}^* = c_{+,\alpha} \in \cB(\cH),
\end{equation}
and $d\rho_{+,\alpha}$ is a $\cB(\cH)$-valued measure satisfying
\begin{equation}
\int_{\bbR} \frac{d(e,\rho_{+,0}(\lambda)e)_{\cB(\cH)}}
{1+\lambda^2} < \infty, \lb{2.26}
\end{equation}
for all $e\in\cH$ (cf.\ Appendix \ref{sA} for details on Nevanlinna--Herglotz
functions).

We first note that for $F\in C(\bbR)$,
\begin{align}
&\big(f,F(H_{+,\alpha})g\big)_{L^2((a,\infty);dx;\cH)}= \int_{\bbR}
d \big(f,E_{H_{+,\alpha}}(\lambda)g\big)_{L^2((a,\infty);dx;\cH)}\,
F(\lambda), \no \\
& f, g \in\dom(F(H_{+,\alpha})) \lb{2.27} \\
& \qquad\ = \bigg\{h\in L^2((a,\infty);dx;\cH)
\,\bigg|\,
\int_{\bbR} d \|E_{H_{+,\alpha}}(\lambda)h\|_{L^2((a,\infty);dx;\cH)}^2
\, |F(\lambda)|^2 < \infty\bigg\}. \no
\end{align}
Equation \eqref{2.27} extends to measurable functions $F$ and holds
also in the strong sense, but the displayed weak version will suffice
for our purpose.

In the following, $C_0^\infty((c,d); \cH)$, $-\infty \leq c<d\leq \infty$,
denotes the usual space of infinitely differentiable $\cH$-valued functions of
compact support contained in $(c,d)$.

\begin{theorem} \lb{t2.5}
Assume Hypothesis \ref{h2.7} and let  $f,g \in C^\infty_0((a,\infty); \cH)$,
$F\in C(\bbR)$, and $\lambda_1, \lambda_2 \in\bbR$,
$\lambda_1<\lambda_2$. Then,
\begin{align}
\begin{split}
& \big(f,F(H_{+,\alpha})E_{H_{+,\alpha}}((\lambda_1,\lambda_2])g \big)_{L^2((a,\infty);dx;\cH)}
\\
& \quad =  \big(\hatt f_{+,\alpha},M_FM_{\chi_{(\lambda_1,\lambda_2]}} \hatt
g_{+,\alpha}\big)_{L^2(\bbR;d\rho_{+,\alpha};\cH)},    \lb{2.28}
\end{split}
\end{align}
where we introduced the notation
\begin{equation}
\hatt h_{+,\alpha}(\lambda)=\int_a^\infty dx \,
\phi_\alpha(\lambda,x,a)^* h(x), \quad \lambda \in\bbR, \;
h\in C^\infty_0((a,\infty); \cH), \lb{2.29}
\end{equation}
and $M_G$ denotes the maximally defined operator of multiplication by the
function $G \in C(\bbR)$ in the Hilbert space\footnote{We recall
that $L^2(\bbR;d\rho_{+,\alpha};\cH)$ is a convenient abbreviation for the Hilbert space
$L^2(\bbR;d\mu_{+,\alpha};\cM_{\rho_{+,\alpha}})$ discussed in detail in
Appendix \ref{sD}, with $d\mu_{+,\alpha}$ a control measure for the
$\cB(\cH)$-valued measure $d\rho_{+,\alpha}$. One recalls that
$\cM_{\rho_{+,\alpha}} \subset \cS(\{\cH_{\lambda}\}_{\lambda\in \bbR})$ is
generated by ${\ul \Lambda} (\cH)$ (or by ${\ul \Lambda} (\{e_n\}_{n\in\cI})$ for
any complete orthonormal system $\{e_n\}_{n\in\cI}$, $\cI \subseteq \bbN$,
in $\cH$).} $L^2(\bbR;d\rho_{+,\alpha};\cH)$,
\begin{align}
\begin{split}
& \big(M_G\hatt h\big)(\lambda)=G(\lambda)\hatt h(\lambda)
\, \text{ for $\rho_{+,\alpha}$-a.e.\ $\lambda\in\bbR$}, \lb{2.30} \\
& \hatt h\in\dom(M_G)=\big\{\hatt k \in L^2(\bbR;d\rho_{+,\alpha};\cH) \,\big|\,
G\hatt k \in L^2(\bbR;d\rho_{+,\alpha};\cH)\big\}.
\end{split}
\end{align}
Here $\rho_{+,\alpha}$ generates the operator-valued measure in the Herglotz representation of the operator-valued Weyl--Titchmarsh function
$m_{+,\alpha}(\cdot)\in\cB(\cH)$ $($cf.\ \eqref{2.25}$)$.
\end{theorem}
\begin{proof}
The point of departure for deriving \eqref{2.28} is Stone's formula
\eqref{2.26a} applied to $T=H_{+,\alpha}$,
\begin{align}
&\big(f,F(H_{+,\alpha})E_{H_{+,\alpha}}((\lambda_1,\lambda_2])
g\big)_{L^2((a,\infty);dx;\cH)}    \lb{2.31} \\
& \quad = \lim_{\delta\downarrow 0}\lim_{\varepsilon\downarrow 0}
\frac{1}{2\pi i} \int_{\lambda_1+\delta}^{\lambda_2+\delta}
d\lambda \, F(\lambda)
\big[\big(f,(H_{+,\alpha}-(\lambda+i\varepsilon)
I)^{-1}g\big)_{L^2((a,\infty);dx;\cH)}
\no \\
& \hspace*{4.9cm} - \big(f,(H_{+,\alpha}-(\lambda-i\varepsilon)
I)^{-1}g\big)_{L^2((a,\infty);dx;\cH)}\big].    \no
\end{align}
Expressing the resolvent in \eqref{2.31} in terms of the Green's function \eqref{3.63A} then yields the following:
\begin{align}
&\big(f,F(H_{+,\alpha})E_{H_{+,\alpha}}((\lambda_1,\lambda_2])
g\big)_{L^2((a,\infty);dx;\cH)} = \lim_{\delta\downarrow
0}\lim_{\varepsilon\downarrow 0} \frac{1}{2\pi i}
\int_{\lambda_1+\delta}^{\lambda_2+\delta} d\lambda \, F(\lambda) \no \\
& \quad \times \int_a^\infty dx
\bigg\{\bigg[\bigg(f(x), \psi_{+,\alpha}(\lambda
+i\varepsilon,x) \int_a^x dx'\,
\phi_\alpha(\lambda-i\varepsilon,x',a)^* g(x')\bigg)_{\cH} \no \\
& \hspace*{2.4cm} + \bigg(f(x), \phi_\alpha(\lambda+i\varepsilon,x,a)
\int_x^\infty dx'\, \psi_{+,\alpha}(\lambda-i\varepsilon,x')^* g(x')\bigg)_{\cH}
\bigg]  \no \\
& \hspace*{1.7cm}  -\bigg[\bigg(f(x),
\psi_{+,\alpha}(\lambda - i\varepsilon,x) \int_a^x dx'\,
\phi_\alpha(\lambda + i\varepsilon,x',a)^* g(x')\bigg)_{\cH}   \lb{2.32} \\
& \hspace*{2.4cm}  + \bigg(f(x),
\phi_\alpha(\lambda-i\varepsilon,x,a)
\int_x^\infty dx' \,\psi_{+,\alpha}(\lambda + i\varepsilon,x')^*
g(x')\bigg)_{\cH}\bigg]\bigg\}.   \no
\end{align}
Freely interchanging the $dx$ and $dx'$ integrals with
the limits and the $d\lambda$ integral (since all integration
domains are finite and all integrands are continuous), and inserting
expression \eqref{3.58A} for $\psi_{+,\alpha}(z,x)$ into \eqref{2.32}, one
obtains
\begin{align}
&\big(f,F(H_{+,\alpha})E_{H_{+,\alpha}}((\lambda_1,\lambda_2])
g\big)_{L^2((a,\infty);dx;\cH)}
=\int_a^\infty dx \bigg(f(x), \bigg\{\int_a^x dx'
\no
\\
& \quad \times \lim_{\delta\downarrow 0} \lim_{\varepsilon\downarrow 0}
\frac{1}{2\pi i} \int_{\lambda_1+\delta}^{\lambda_2+\delta} d\lambda \, F(\lambda)
\Big[\big[\theta_\alpha(\lambda,x,a) +
\phi_\alpha(\lambda,x,a) m_{+,\alpha}(\lambda+i\varepsilon) \big]
\phi_\alpha(\lambda,x',a)^*
\no
\\
& \hspace*{1.3cm} -\big[\theta_\alpha(\lambda,x,a) +
\phi_\alpha(\lambda,x,a) m_{+,\alpha}(\lambda-i\varepsilon) \big]
\phi_\alpha(\lambda,x',a)^* \Big]g(x')
\no
\\
& \quad +\int_x^\infty dx'\, \lim_{\delta\downarrow 0}
\lim_{\varepsilon\downarrow 0} \frac{1}{2\pi i}
\int_{\lambda_1+\delta}^{\lambda_2+\delta} d\lambda \, F(\lambda)
\lb{2.33}
\\
& \qquad \times \Big[ \phi_\alpha(\lambda,x,a) \big[\theta_\alpha(\lambda,x',a)^* +
m_{+,\alpha}(\lambda-i\varepsilon)^*\phi_\alpha(\lambda,x',a)^*\big]
\no
\\
& \hspace{1.3cm} -\phi_\alpha(\lambda,x,a)
\big[\theta_\alpha(\lambda,x',a)^* +
m_{+,\alpha}(\lambda+i\varepsilon)^*\phi_\alpha(\lambda,x',a)^*\big]\Big]g(x')\bigg\} \bigg)_{\cH}.
\no
\end{align}
Here we employed the fact that for fixed $x\in [a,\infty)$,
$\theta_\alpha(z,x,a)$ and $\phi_\alpha(z,x,a)$ are entire with
respect to  $z$, that
$\theta_\alpha(z,\cdot,a), \phi_\alpha(z,\cdot,a) \in W^{1,1}([a,c]; \cH)$ for
all $c>a$, and hence that
\begin{align}
\begin{split}
\theta_\alpha(\lambda\pm i\varepsilon,x,a)
&\underset{\varepsilon\downarrow 0}{=}
\theta_\alpha(\lambda,x,a) \pm
i\varepsilon(d/dz)\theta_\alpha(z,x,a)|_{z=\lambda} + \Oh(\ve^2),
\lb{2.33Aa} \\
\phi_\alpha(\lambda\pm i\varepsilon,x,a)
&\underset{\varepsilon\downarrow 0}{=} \phi_\alpha(\lambda,x,a)
\pm i\varepsilon(d/dz)\phi_\alpha(z,x,a)|_{z=\lambda} + \Oh(\ve^2)
\end{split}
\end{align}
with $\Oh(\varepsilon^2)$ being uniform with respect to
$(\lambda,x)$ as long as $\lambda$ and $x$ vary in
compact subsets of $\bbR\times [a,\infty)$. Moreover, we used that
for all $f,g \in \cH$ (cf.\ Theorem \ref{tA.7}\,$(vi)$),
\begin{align}
\begin{split}
&\varepsilon |(f,m_{+,\alpha}(\lambda+i\varepsilon) g)_{\cH}|
\leq C(\lambda_1,\lambda_2,\varepsilon_0,f,g) \, \text{ for } \, \lambda\in
[\lambda_1,\lambda_2], \; 0<\varepsilon\leq\varepsilon_0,
\lb{2.33a} \\
&\varepsilon (f,\Re(m_{+,\alpha}(\lambda+i\varepsilon)) g)_{\cH}
\underset{\varepsilon\downarrow 0}{=}\oh(1), \quad \lambda\in \bbR.
\end{split}
\end{align}
In particular, utilizing \eqref{2.33Aa} and \eqref{2.33a},
$\phi_\alpha(\lambda\pm i\varepsilon,x,a)$ and
$\theta_\alpha(\lambda\pm i\varepsilon,x,a)$ have been replaced by
$\phi_\alpha(\lambda,x,a)$ and $\theta_\alpha(\lambda,x,a)$ under the
$d\lambda$ integrals in \eqref{2.33}. Canceling appropriate terms in
\eqref{2.33},  simplifying the remaining terms, and using $m_{+,\alpha}(z)=m_{+,\alpha}(\ol z)^*$
then yield
\begin{align}
&\big(f,F(H_{+,\alpha})E_{H_{+,\alpha}}((\lambda_1,\lambda_2])
g\big)_{L^2((a,\infty);dx;\cH)}    \no \\
& \quad =\int_a^\infty dx \int_a^\infty dx'  \no \\
& \qquad \times \lim_{\delta\downarrow 0}\lim_{\varepsilon\downarrow 0}
\frac{1}{\pi}
\int_{\lambda_1+\delta}^{\lambda_2+\delta} d\lambda \, F(\lambda)    \lb{2.34} \\
& \hspace*{3cm} \times \big(\phi_\alpha(\lambda,x,a)^* f(x),
\Im(m_{+,\alpha}(\lambda+i\varepsilon)) \phi_\alpha(\lambda,x',a)^* g(x')\big)_{\cH}.   \no
\end{align}
Using the fact that by \eqref{A.43}
\begin{align}
& \int_{(\lambda_1,\lambda_2]} d\rho_{+,\alpha}(\lambda) h =
\rho_{+,\alpha}((\lambda_1,\lambda_2]) h =
\lim_{\delta\downarrow 0}\lim_{\varepsilon\downarrow 0}
\frac{1}{\pi}\int_{\lambda_1+\delta}^{\lambda_2+\delta} d\lambda \,
\Im(m_{+,\alpha}(\lambda+i\varepsilon)) h,  \no \\
& \hspace*{10cm} h \in \cH,  \lb{2.35}
\end{align}
and hence that
\begin{align}
\int_{\bbR} d\rho_{+,\alpha}(\lambda)\, h(\lambda) &=
\lim_{\varepsilon\downarrow 0} \frac{1}{\pi}\int_{\bbR} d\lambda \,
\Im(m_{+,\alpha}(\lambda+i\varepsilon))\, h(\lambda), \quad
h\in C_0(\bbR; \cH), \lb{2.36} \\
\int_{(\lambda_1,\lambda_2]} d\rho_{+,\alpha}(\lambda)\, k(\lambda) &=
\lim_{\delta\downarrow 0} \lim_{\varepsilon\downarrow 0} \frac{1}{\pi}
\int_{\lambda_1+\delta}^{\lambda_2+\delta} d\lambda \,
\Im(m_{+,\alpha}(\lambda+i\varepsilon))\, k(\lambda), \quad
k\in C(\bbR; \cH), \lb{2.37}
\end{align}
(with $C_0(\bbR; \cH)$ the space of continuous compactly supported
$\cH$-valued functions on $\bbR$) one concludes that
\begin{align}
&\big(f,F(H_{+,\alpha})E_{H_{+,\alpha}}((\lambda_1,\lambda_2])
g\big)_{L^2((a,\infty);dx;\cH)}    \no \\
& \quad =\int_a^\infty dx \int_a^\infty dx' \int_{(\lambda_1,\lambda_2]}
F(\lambda) \big(\phi_\alpha(\lambda,x,a)^* f(x), d\rho_{+,\alpha}(\lambda) \,
\phi_\alpha(\lambda,x',a)^* g(x')\big)_{\cH}\no \\
& \quad = \int_{(\lambda_1,\lambda_2]}
F(\lambda) \, \big(\hatt f_{+,\alpha}(\lambda),
d\rho_{+,\alpha} \, (\lambda) \, \hatt g_{+,\alpha}(\lambda)\big)_{\cH},    \lb{2.38}
\end{align}
using \eqref{2.29} and interchanging the $dx$, $dx'$ and
$d\rho_{+,\alpha}$ integrals once more. We note that
$\hatt f_{+,\alpha}, \hatt g_{+,\alpha} \in L^2(\bbR;d\rho_{+,\alpha};\cH)$ by
Lemma \ref{lD.16} and Theorem \ref{tD.17}.
\end{proof}

\begin{remark} \lb{r2.6}
Theorem \ref{t2.5} is of course well-known in the scalar case (i.e., where $\dim(\cH) =1$), see, for
instance, the extensive list of references in \cite{GZ06}. In the matrix-valued case
(i.e., if $\dim(\cH) < \infty$) we refer, for instance, to Hinton and Schneider \cite{HS98}, and in the
operator-valued case (where $\dim(\cH)=\infty$) to Gorbachuk \cite{Go68} under more restrictive
regularity assumptions on the potential $V(\cdot)$ and without providing details in the steps leading
from \eqref{2.33} to \eqref{2.38}.
\end{remark}

\begin{remark} \lb{r2.6a}
The effortless derivation of the link between the family of spectral
projections $E_{H_{+,\alpha}}(\cdot)$ and the operator-valued spectral function
$\rho_{+,\alpha}(\cdot)$ of $H_{+,\alpha}$ in Theorem \ref{t2.5} applies
equally well to half-line Dirac-type operators, Hamiltonian systems, half-lattice
Jacobi operators, and CMV operators (cf.\ \cite{GZ06}, \cite{GZ06a} and the literature
cited therein). In the context of operator-valued potential coefficients of half-line
Schr\"odinger operators this strategy has already been used by
M.\ L.\ Gorbachuk \cite{Go68} in 1966.
\end{remark}

Actually, one can improve on Theorem \ref{t2.5} and remove the
compact support restrictions on $f$ and $g$ in the usual way. To this
end one considers the map
\begin{equation}
\widetilde U_{+,\alpha} : \begin{cases} C_0^\infty((a,\infty); \cH)\to
L^2(\bbR;d\rho_{+,\alpha};\cH) \\[1mm]
h \mapsto \hatt h_{+,\alpha}(\cdot)=
\int_a^\infty dx\, \phi_\alpha(\cdot,x,a)^* h(x). \end{cases} \lb{2.39}
\end{equation}
Taking $f=g$, $F=1$, $\lambda_1\downarrow -\infty$, and
$\lambda_2\uparrow \infty$ in \eqref{2.28} then shows that
$\widetilde U_{+,\alpha}$ is a densely defined isometry in
$L^2((a,\infty);dx;\cH)$, which extends by continuity to an isometry on
$L^2((a,\infty);dx;\cH)$. The latter is denoted by $U_{+,\alpha}$ and given by
\begin{equation}
U_{+,\alpha} : \begin{cases}L^2((a,\infty);dx;\cH)\to L^2(\bbR;d\rho_{+,\alpha};\cH)
\\[1mm]
h \mapsto \hatt h_{+,\alpha}(\cdot)=
\slimes_{b\uparrow\infty}\int_a^b dx\, \phi_\alpha(\cdot,x,a)^* h(x),
\end{cases}  \lb{2.40}
\end{equation}
where $\slimes$ refers to the $L^2(\bbR;d\rho_{+,\alpha};\cH)$-limit.

The calculation in \eqref{2.38} also yields
\begin{equation}
(E_{H_{+,\alpha}}((\lambda_1,\lambda_2])g)(x)
=\int_{(\lambda_1,\lambda_2]}
\phi_\alpha(\lambda,x,a) \, d\rho_{+,\alpha}(\lambda)\,
\hatt g_{+,\alpha}(\lambda), \quad g\in C_0^\infty((a,\infty); \cH)
\lb{2.41}
\end{equation}
and subsequently, \eqref{2.41} extends to all $g\in L^2((a,\infty);dx;\cH)$
by continuity. Moreover, taking $\lambda_1\downarrow -\infty$ and
$\lambda_2\uparrow \infty$ in \eqref{2.41} using
\begin{equation}
\slim_{\lambda\downarrow -\infty} E_{H_{+,\alpha}}(\lambda)=0, \quad
\slim_{\lambda\uparrow \infty}
E_{H_{+,\alpha}}(\lambda)=I_{L^2((a,\infty);dx;\cH)},
\lb{2.42}
\end{equation}
where
\begin{equation}
E_{H_{+,\alpha}}(\lambda)=E_{H_{+,\alpha}}((-\infty,\lambda]), \quad \lambda\in\bbR, \lb{2.43}
\end{equation}
then yields
\begin{equation}
g(\cdot)=\slimes_{\mu_1\downarrow -\infty, \mu_2\uparrow\infty}
\int_{\mu_1}^{\mu_2} \phi_\alpha(\lambda,\cdot,a) \, d\rho_{+,\alpha}(\lambda)\,
\hatt g_{+,\alpha}(\lambda), \quad g\in L^2((a,\infty);dx;\cH),  \lb{2.44}
\end{equation}
where $\slimes$ refers to the $L^2([a,\infty); dx; \cH)$-limit.

In addition, one can show
that the map $U_{+,\alpha}$ in \eqref{2.40} is onto and hence that
$U_{+,\alpha}$ is unitary (i.e., $U_{+,\alpha}$ and
$U_{+,\alpha}^{-1}$ are isometric isomorphisms between
$L^2((a,\infty);dx;\cH)$ and $L^2(\bbR;d\rho_{+,\alpha};\cH)$) with
\begin{equation}
U_{+,\alpha}^{-1} : \begin{cases} L^2(\bbR;d\rho_{+,\alpha};\cH) \to
L^2((a,\infty);dx;\cH)  \\[1mm]
\hatt h \mapsto \slimes_{\mu_1\downarrow -\infty, \mu_2\uparrow\infty}
\int_{\mu_1}^{\mu_2}
\phi_\alpha(\lambda,\cdot,a) \, d\rho_{+,\alpha}(\lambda)\, \hatt h(\lambda).
\end{cases} \lb{2.45}
\end{equation}
To show this we denote the operator defined in \eqref{2.45} temporarily by
$V_{+,\alpha}$ and first claim that $V_{+,\alpha}$ is bounded: Indeed, one
computes for all $\hatt f\in C_0^\infty(\bbR; \cH)$ and $g\in C_0^\infty((a,\infty); \cH)$,
\begin{align}
\big(g,V_{+,\alpha} \hatt f \,\big)_{L^2((a,\infty);dx;\cH)} &= \int_a^\infty dx \, \bigg(g(x),
\int_\bbR \phi_\alpha(\lambda,x,a) \, d\rho_{+,\alpha}(\lambda) \, \hatt f(\lambda)\bigg)_{\cH} \no \\
&= \bigg(\int_a^\infty dx \,  \phi_\alpha(\lambda,x,a)^* g(x),
\int_\bbR d\rho_{+,\alpha}(\lambda) \, \hatt f(\lambda)\bigg)_{\cH} \no \\
&= \big(U_{+,\alpha} g, \hatt f \,\big)_{L^2(\bbR;d\rho_{+,\alpha};\cH)}.   \lb{2.44a}
\end{align}
Since $U_{+,\alpha}$ is isometric, \eqref{2.44a} extends by continuity to all
$g\in  L^2((a,\infty);dx;\cH)$. Thus,
\begin{align}
\big\|V_{+,\alpha} \hatt f \,\big\|_{L^2((a,\infty);dx;\cH)}
&= \sup_{g\in  L^2((a,\infty);dx;\cH), \, g\neq 0}
\bigg|\f{\big(g,V_{+,\alpha} \hatt f \,\big)_{L^2((a,\infty);dx;\cH)}}{\|g\|_{L^2((a,\infty);dx;\cH)}}
\bigg| \no \\
& \leq  \sup_{g\in  L^2((a,\infty);dx;\cH), \, g\neq 0}
\f{\|U_{+,\alpha} g\|_{L^2(\bbR;d\rho_{+,\alpha};\cH)}}
{\|g\|_{L^2((a,\infty);dx;\cH)}}
\big\|\hatt f \,\big\|_{L^2(\bbR;d\rho_{+,\alpha};\cH)}  \no \\
& = \big\|\hatt f \,\big\|_{L^2(\bbR;d\rho_{+,\alpha};\cH)}, \quad \hatt f\in C_0^\infty(\bbR; \cH),
\end{align}
and hence $\|V_{+,\alpha}\|\leq 1$. By \eqref{2.44},
\begin{equation}
V_{+,\alpha} U_{+,\alpha} = I_{L^2((a,\infty);dx;\cH)}.
\end{equation}
To prove that $U_{+,\alpha}$ is onto, and hence unitary, it thus suffices to prove that
$V_{+,\alpha}$ is injective.

Let $\hatt f\in L^2(\bbR;d\rho_{+,\alpha};\cH)$, $\lambda_1, \lambda_2 \in\bbR$,
$\lambda_1<\lambda_2$, and consider
\begin{align}
\begin{split}
& (H_{+,\alpha} - zI_{L^2((a,\infty);dx;\cH)}) \bigg(\int_{\lambda_1}^{\lambda_2} \phi_\alpha(\lambda,\cdot,a) \,
(\lambda - z)^{-1}d\rho_{+,\alpha}(\lambda) \, \hatt f(\lambda)\bigg)   \\
& \quad = \int_{\lambda_1}^{\lambda_2}
\phi_\alpha(\lambda,\cdot,a) \, d\rho_{+,\alpha}(\lambda) \, \hatt f(\lambda),
\quad z\in\bbC_+.
\end{split}
\end{align}
Then,
\begin{align}
\begin{split}
& \int_{\lambda_1}^{\lambda_2} \phi_\alpha(\lambda,\cdot,a) \,
(\lambda - z)^{-1}\, d\rho_{+,\alpha}(\lambda) \, \hatt f(\lambda)    \\
& \quad = (H_{+,\alpha} - zI_{L^2((a,\infty);dx;\cH)})^{-1} \bigg(\int_{\lambda_1}^{\lambda_2}
\phi_\alpha(\lambda,\cdot,a) \, d\rho_{+,\alpha}(\lambda)
\, \hatt f(\lambda)\bigg),   \quad  z\in\bbC_+.    \lb{2.63a}
\end{split}
\end{align}
Taking $\slim_{\lambda_1\downarrow -\infty, \lambda_2\uparrow \infty}$ in \eqref{2.63a} implies
\begin{equation}
V_{+,\alpha} \big((\cdot -z)^{-1} \hatt f \,\big) = (H_{+,\alpha} -zI_{L^2((a,\infty);dx;\cH)})^{-1} V_{+,\alpha} \hatt f,
\quad z\in\bbC_+.
\end{equation}
Next, suppose that $\hatt f_0 \in \ker(V_{+,\alpha})$, and let
$\big\{\hatt f_n\big\}_{n\in\bbN} \subset L^2(\bbR;d\rho_{+,\alpha};\cH)$ such that
$\supp\big(\hatt f_n\big)$ is compact for each $n\in\bbN$ and
$\lim_{n\uparrow\infty}\big\|\hatt f_0 - \hatt f_n\big\|_{L^2(\bbR;d\rho_{+,\alpha};\cH)}=0$.
Then,
\begin{align}
\begin{split}
\big(V_{+,\alpha} \big((\cdot -z)^{-1} \hatt f_n\big)\big)(x)
= \big((H_{+,\alpha} - zI_{L^2((a,\infty);dx;\cH)})^{-1} V_{+,\alpha} \hatt f_n\big)(x),&
\\
x>a, \; z\in\bbC_+, \; n\in\bbN,&
\end{split}
\end{align}
and thus for all $y\in [a,\infty)$, and arbitrary $e\in\cH$,
\begin{align}
& \int_a^y dx \int_{\bbR} \big(e,\phi_\alpha(\lambda,x,a) \, (\lambda - z)^{-1}
\, d\rho_{+,\alpha}(\lambda) \, \hatt f_n(\lambda)\big)_{\cH}   \no \\
& \quad = \int_a^y dx \int_{\bbR} \big(\phi_\alpha(\lambda,x,a) ^* e,
d\rho_{+,\alpha}(\lambda) \, (\lambda - z)^{-1} \hatt f_n(\lambda)\big)_{\cH} \no \\
& \quad =\int_{\bbR} \bigg( \int_a^y dx \, \phi_\alpha(\lambda,x,a) ^* e,
d\rho_{+,\alpha}(\lambda) \, (\lambda - z)^{-1} \hatt f_n(\lambda)\bigg)_{\cH} \no \\
& \quad = \int_a^y dx \, \big(e, \big((H_{+,\alpha}-zI_{L^2((a,\infty);dx;\cH)})^{-1} V_{+,\alpha} \hatt f_n\big)(x)\big)_{\cH}.
\lb{2.44aa}
\end{align}
Noticing that
\begin{equation}
\int_a^\infty dx \, \phi_\alpha(\cdot,x,a)^* \chi_{[a,y]}(x) e
= (U_{+,\alpha} \chi_{[a,y]} e)(\cdot) \in L^2(\bbR;d\rho_{+,\alpha};\cH),
\end{equation}
and taking $n\uparrow\infty$ in \eqref{2.44aa} then results in
\begin{align}
&\lim_{n\uparrow\infty}
\int_a^y dx \int_{\bbR}  \big(e, \phi_\alpha(\lambda,x,a)
d\rho_{+,\alpha}(\lambda) \, (\lambda - z)^{-1}  \hatt f_n(\lambda)\big)_{\cH} \no \\
& \quad = \int_a^y dx  \int_{\bbR}  \big(e, \phi_\alpha(\lambda,x,a)
d\rho_{+,\alpha}(\lambda) \, (\lambda - z)^{-1}  \hatt f_0(\lambda)\big)_{\cH}  \no \\
& \quad =  \int_{\bbR}  (\lambda - z)^{-1}  \int_a^y dx \, \big(e, \phi_\alpha(\lambda,x,a)
d\rho_{+,\alpha}(\lambda) \hatt f_0(\lambda)\big)_{\cH}   \lb{2.44b} \\
& \quad = \lim_{n\uparrow\infty}
\int_a^y dx \, \big(e, \big((H_{+,\alpha}-zI_{L^2((a,\infty);dx;\cH)})^{-1} V_{+,\alpha} \hatt f_n\big)(x)\big)_{\cH} \no \\
& \quad = \int_a^y dx \, \big(e, \big((H_{+,\alpha}-zI_{L^2((a,\infty);dx;\cH)})^{-1} V_{+,\alpha} \hatt f_0\big)(x)\big)_{\cH} =0,
\no
\\
&\hspace{5.4cm} y\in [a,\infty), \; z\in\bbC_+, \; e\in\cH.
\no
\end{align}
Applying the Stieltjes inversion formula to the (finite) complex-valued measure in
the 3rd line of \eqref{2.44b}, given by,
\begin{equation}
\int_a^y dx \, \big(e, \phi_\alpha(\lambda,x,a)
\, d\rho_{+,\alpha} (\lambda) \, \hatt f_0(\lambda)\big)_{\cH},
\end{equation}
implies for all $\lambda_1, \lambda_2 \in\bbR$, $\lambda_1 < \lambda_2$, and $e \in \cH$,
\begin{equation}
\int_{(\lambda_1,\lambda_2]}
\int_a^y dx \, \big(e, \phi_\alpha(\lambda,x,a)
\, d\rho_{+,\alpha} (\lambda) \, \hatt f_0(\lambda)\big)_{\cH} =0, \quad y\in [a,\infty).   \lb{2.44c}
\end{equation}
Differentiating \eqref{2.44c} repeatedly with respect to $y$, noting that
$\phi_\al(\la,y,a)$ and $\phi_\al'(\la,y,a)$ are continuous in
$(\la,y)\in \bbR\times [a,\infty)$, and using the dominated convergence theorem,
one concludes that for all $y\in[a,\infty)$, $e\in\cH$,
\begin{align}
\begin{split}
\int_{(\la_1,\la_2]} \big(e, \phi_\al(\la,y,a) \, d\rho_{+,\al}(\la)
\, \hatt f_0(\la)\big)_{\cH} & = 0,  \lb{2.45c} \\
\int_{(\la_1,\la_2]} \, \big(e, \phi_\al'(\la,y,a) \, d\rho_{+,\al}(\la)
\, \hatt f_0(\la)\big)_{\cH} & = 0.
\end{split}
\end{align}
Using \eqref{2.5}, the fact that $\hatt f_0, \chi_{(\la_1,\la_2]}e \in L^2(\bbR;d\rho_{+,\alpha};\cH)$, and the dominated convergence theorem once again then implies
\begin{align}
0 &= \int_{(\la_1,\la_2]} \, \big(e, \phi_\al(\la,a,a) \, d\rho_{+,\al}(\la) \, \hatt f_0(\la)\big)_{\cH}
\no
\\
& = - \int_{(\la_1,\la_2]} \big(\sin(\al) e, d\rho_{+,\al}(\la) \, \hatt f_0(\la)\big)_{\cH},  \lb{2.45d}
\\
0 &= \int_{(\la_1,\la_2]} \, \big(e, \phi_\al'(\la,a,a) \, d\rho_{+,\al}(\la) \, \hatt f_0(\la)\big)_{\cH}
\no
\\
&= \int_{(\la_1,\la_2]} \big(\cos(\al) e, d\rho_{+,\al}(\la) \, \hatt f_0(\la)\big)_{\cH}.   \lb{2.45e}
\end{align}
Taking $e=\sin(\al)e_1$ in \eqref{2.45d} and $e=\cos(\al)e_1$ in \eqref{2.45e} with an arbitrary $e_1\in\cH$ and subtracting \eqref{2.45d} from \eqref{2.45e} then gives
\begin{align}
0=\int_{(\la_1,\la_2]} \big(e_1, d\rho_{+,\al}(\la) \, \hatt f_0(\la)\big)_{\cH}. \lb{2.45f}
\end{align}
Since the interval $(\la_1,\la_2]$ was chosen arbitrary, \eqref{2.45f} implies
\begin{align}
\hatt f_0(\la) = 0 \;\, \rho_{+,\al}\text{-a.e.},
\end{align}
and hence $\ker(V_{+,\alpha})=\{0\}$. Thus $U_{+,\al}$ is onto.

We recall that the essential range of $F$ with respect to a scalar measure
$\mu$ is defined by
\begin{equation}
\essran_{\mu}(F)=\{z\in\bbC\,|\, \text{for all
$\varepsilon>0$,} \, \mu(\{\lambda\in\bbR \,|\,
|F(\lambda)-z|<\varepsilon\})>0\},  \lb{2.46c}
\end{equation}
and that $\essran_{\rho_{+,\alpha}}(F)$ for $F\in C(\bbR)$ is then defined to be
$\essran_{\nu_{+,\alpha}}(F)$ for any control measure $d\nu_{+,\alpha}$
of the operator-valued measure $d\rho_{+,\alpha}$. Given a complete orthonormal system
$\{e_n\}_{n \in \cI}$ in $\cH$ ($\cI \subseteq \bbN$ an appropriate index set), a convenient
control measure for $d\rho_{+,\alpha}$ is given by
\begin{equation}
\mu_{+,\alpha}(B)=\sum_{n\in\cI}2^{-n}(e_n, \rho_{+,\alpha}(B)e_n)_\cH, \quad
B\in\mathfrak{B}(\bbR).        \lb{2.46d}
\end{equation}

We sum up these considerations in a variant of the
spectral theorem for (functions of) $H_{+,\alpha}$.

\begin{theorem} \lb{t2.6}
Assume Hypothesis \ref{h2.7} and suppose $F\in C(\bbR)$. Then,
\begin{equation}
U_{+,\alpha} F(H_{+,\alpha})U_{+,\alpha}^{-1} = M_F I_{\cH}    \lb{2.46}
\end{equation}
in $L^2(\bbR;d\rho_{+,\alpha};\cH)$ $($cf.\ \eqref{2.30}$)$. Moreover,
\begin{align}
& \sigma(F(H_{+,\alpha}))= \essran_{\rho_{+,\alpha}}(F), \lb{2.46a} \\
& \sigma(H_{+,\alpha})=\supp(d\rho_{+,\alpha}),  \lb{2.46b}
\end{align}
and the multiplicity of the spectrum of $H_{+,\alpha}$ is at most equal to $\dim (\cH)$.
\end{theorem}
\begin{proof}
First, we note that \eqref{2.46} follows from Theorem \ref{t2.5} and the discussion following it. The fact
\eqref{2.46b} is a special case of \eqref{2.46a} and hence only the latter requires a proof.

Since $F(H_{+,\al})$ is unitarily equivalent to the operator of multiplication by $F(\cdot)$ in
$L^2(\bbR;d\rho_{+,\alpha};\cH)$, it suffices to check that $M_{(F-z)} I_{\cH}$ is not boundedly invertible whenever $z\in\essran_{\rho_{+,\alpha}}(F)$. Fix an arbitrary $z\in\essran_{\rho_{+,\alpha}}(F)$
and $\varepsilon>0$. Since $F\in C(\bbR)$, the set $\{\lambda\in\bbR \,|\,
|F(\lambda)-z|<\varepsilon\}$ is open and hence is a countable union of disjoint open intervals. By
\eqref{2.46c} there is a bounded interval $B\subset\{\lambda\in\bbR \,|\, |F(\lambda)-z|<\varepsilon\}$
such that $\rho_{+,\al}(B)\neq0$ and hence there is also a nonzero vector $h\in\cH$ such that
$(h,\rho_{+,\al}(B)h)_\cH\neq0$. Then $\chi_B h\in L^2(\bbR;d\rho_{+,\alpha};\cH)$ with
\begin{equation}
\|\chi_B h\|^2_{L^2(\bbR;d\rho_{+,\alpha};\cH)} = (h,\rho_{+\al}(B)h)_\cH>0
\end{equation}
and
\begin{equation}
\|M_{F-z}\chi_B h\|_{L^2(\bbR;d\rho_{+,\alpha};\cH)}\leq
\varepsilon\|\chi_B h\|_{L^2(\bbR;d\rho_{+,\alpha};\cH)}.
\end{equation}
Since $\varepsilon>0$ is arbitrary, this implies that $M_{F-z}I_\cH$ is not boundedly invertible in
$L^2(\bbR;d\rho_{+,\alpha};\cH)$.

Conversely, assume $z\in\bbR\bs\essran_{\rho_{+,\alpha}}(F)$. Then by \eqref{2.46c}, \eqref{2.46d},
 there exists $\varepsilon>0$ such that for any interval
$B\subset\{\lambda\in\bbR \,|\, |F(\lambda)-z|<\varepsilon\}$ one has
$\mu_{+,\alpha}(B) = \rho_{+\al}(B)=0$. Then for any
$g\in\dom(M_F)\subset L^2(\bbR;d\rho_{+,\alpha};\cH)$,
\begin{equation}
\|M_{F-z}g\|_{L^2(\bbR;d\rho_{+,\alpha};\cH)} \geq \varepsilon\|g\|_{L^2(\bbR;d\rho_{+,\alpha};\cH)},
\end{equation}
that is, $M_{F-z}I_\cH$ is boundedly invertible in this case.

Using the identity function $F(z)=z$ it follows from \eqref{2.46} that the multiplicity of the spectrum of $H_{+,\alpha}$ is equal to that of $M_zI_\cH$ which is at most $\dim (\cH)$.
\end{proof}

\section{Spectral Theory of Schr\"odinger Operators with Operator-Valued
Potentials on the Real Line} \lb{s5}

In our final section we develop basic spectral theory for full-line Schr\"odinger operators $H$
in $L^2(\bbR; dx; \cH)$, employing a $2 \times 2$ block operator representation of the
associated Weyl--Titchmarsh matrix and its $\cB\big(\cH^2\big)$-valued spectral measure,
decomposing $\bbR$ into a left and right half-line with reference point $x_0 \in \bbR$,
$(-\infty, x_0] \cup [x_0, \infty)$. The latter decomposition is familiar from the scalar and
matrix-valued ($\dim(\cH) < \infty$) special cases. Our principal new results, Theorems \ref{t2.9}
and \ref{t2.10}, again yield a diagonalization of $H$ and the corresponding generalized
eigenfunction expansion, illustrating the spectral theorem for $F(H)$ and support properties of
the underlying spectral measure.

In the special scalar case where $\dim(\cH)<\infty$, the material of this section is standard and
various parts of it can be found, for instance, in \cite{BE05}, \cite[Ch.\ 9]{CL85},
\cite[Sect.\ XIII.5]{DS88}, \cite[Ch.\ 2]{EK82}, \cite{Ev04}, \cite[Ch.\ 10]{Hi69}, \cite{HS98},
\cite{Ko49}, \cite{Le51}, \cite[Ch.\ 2]{LS75}, \cite[Ch.\ VI]{Na68}, \cite[Ch.\ 6]{Pe88},
\cite[Chs.\ II, III]{Ti62}, \cite[Sects.\ 7--10]{We87}. However, in the infinite-dimensional case,
$\dim(\cH)=\infty$, the principal results obtained in this section are new.

We make the following basic assumption throughout this section.

\begin{hypothesis} \lb{h2.8}
$(i)$ Assume that
\begin{equation}
V\in L^1_{\loc} (\bbR;dx;\cH), \quad V(x)=V(x)^* \, \text{ for a.e. } x\in\bbR
\lb{2.51}
\end{equation}
$(ii)$ Introducing the differential expression $\tau$ given by
\begin{equation}
\tau=-\f{d^2}{dx^2} + V(x), \quad x\in\bbR, \lb{2.52}
\end{equation}
we assume $\tau$ to be in the limit point case at $+\infty$ and at
$-\infty$.
\end{hypothesis}

Associated with the differential expression $\tau$ one introduces the self-adjoint Schr\"odinger operator $H$ in $L^2(\bbR;dx;\cH)$ by
\begin{align}
&Hf=\tau f,   \lb{2.53}
\\ \no
&f\in \dom(H)=\{g\in L^2(\bbR;dx;\cH) \,|\, g, g' \in
W^{2,1}_{\loc}(\bbR;dx;\cH); \, \tau g\in L^2(\bbR;dx;\cH)\}.
\end{align}

As in the half-line context we introduce the $\cB(\cH)$-valued fundamental
system of solutions $\phi_\alpha(z,\cdot,x_0)$ and
$\theta_\alpha(z,\cdot,x_0)$, $z\in\bbC$, of
\begin{equation}
(\tau \psi)(z,x) = z \psi(z,x), \quad x\in \bbR \lb{2.54}
\end{equation}
with respect to a fixed reference point $x_0\in\bbR$, satisfying the
initial conditions at the point $x=x_0$,
\begin{align}
\begin{split}
\phi_\alpha(z,x_0,x_0)&=-\theta'_\alpha(z,x_0,x_0)=-\sin(\alpha), \\
\phi'_\alpha(z,x_0,x_0)&=\theta_\alpha(z,x_0,x_0)=\cos(\alpha), \quad
\alpha=\alpha^*\in\cB(\cH). \lb{2.55}
\end{split}
\end{align}
Again we note that by Corollary 2.5\,$(iii)$, for any fixed $x, x_0\in\bbR$, the functions $\theta_{\alpha}(z,x,x_0)$ and $\phi_{\alpha}(z,x,x_0)$ as well as their strong $x$-derivatives are entire with respect to $z$ in the $\cB(\cH)$-norm. The same is true for the functions $z\mapsto\theta_{\alpha}(\ol{z},x,x_0)^*$ and
$z\mapsto\phi_{\alpha}(\ol{z},x,x_0)^*$. Moreover, by \eqref{2.7i},
\begin{equation}
W(\theta_\alpha(\ol{z},\cdot,x_0)^*,\phi_\alpha(z,\cdot,x_0))(x)=I_\cH, \quad
z\in\bbC.  \lb{2.56}
\end{equation}

Particularly important solutions of \eqref{2.54} are the
{\it Weyl--Titchmarsh solutions} $\psi_{\pm,\alpha}(z,\cdot,x_0)$,
$z\in\bbC\backslash\bbR$, uniquely characterized by
\begin{align}
\begin{split}
&\psi_{\pm,\alpha}(z,\cdot,x_0)f\in L^2([x_0,\pm\infty);dx;\cH), \quad f\in\cH,
\\
&\sin(\alpha)\psi'_{\pm,\alpha}(z,x_0,x_0)
+\cos(\alpha)\psi_{\pm,\alpha}(z,x_0,x_0)=I_\cH, \quad
z\in\bbC\backslash\bbR. \lb{2.57}
\end{split}
\end{align}
The crucial condition in \eqref{2.57} is again the $L^2$-property which
uniquely determines $\psi_{\pm,\alpha}(z,\cdot,x_0)$ up to constant
multiples by the limit point hypothesis of $\tau$ at $\pm\infty$. In
particular, for
$\alpha = \alpha^*, \beta = \beta^* \in \cB(\cH)$,
\begin{align}
\psi_{\pm,\alpha}(z,\cdot,x_0) = \psi_{\pm,\beta}(z,\cdot,x_0)C_\pm(z,\alpha,\beta,x_0)
\lb{2.58}
\end{align}
for some coefficients $C_\pm (z,\alpha,\beta,x_0)\in\cB(\cH)$. The normalization in \eqref{2.57} shows that
$\psi_{\pm,\alpha}(z,\cdot,x_0)$ are of the type
\begin{equation}
\psi_{\pm,\alpha}(z,x,x_0)=\theta_{\alpha}(z,x,x_0)
+ \phi_{\alpha}(z,x,x_0) m_{\pm,\alpha}(z,x_0),
\quad  z\in\bbC\backslash\bbR, \; x\in\bbR, \lb{2.59}
\end{equation}
for some coefficients $m_{\pm,\alpha}(z,x_0)\in\cB(\cH)$, the
{\it Weyl--Titchmarsh $m$-functions} associated with $\tau$, $\alpha$,
and $x_0$ (cf.\ Theorem \ref{t3.15}).

Next, we show that $\pm m_{\pm,\al}(\cdot,x_0)$ are operator-valued Herglotz functions. It follows from \eqref{2.54} and \eqref{2.55} that the Wronskian of $\psi_{\pm,\al}(\ol{z_1},x,x_0)^*$ and $\psi_{\pm,\al}(z_2,x,x_0)$ satisfies
\begin{align}
W(\psi_{\pm,\al}(\ol{z_1},x_0,x_0)^*,\psi_{\pm,\al}(z_2,x_0,x_0)) &= m_{\pm,\alpha}(z_2,x_0)- m_{\pm,\alpha}(\ol{z_1},x_0)^*,
\\
\f{d}{dx}W(\psi_{\pm,\al}(\ol{z_1},x,x_0)^*,\psi_{\pm,\al}(z_2,x,x_0)) &= (z_1-z_2)\psi_{\pm,\al}(\ol{z_1},x,x_0)^*\psi_{\pm,\al}(z_2,x,x_0), \no
\\
&\hspace{3cm} z_1,z_2\in\bbC\bs\bbR.
\end{align}
Hence, using the limit point hypothesis of $\tau$ at $\pm\infty$ and the $L^2$-property in \eqref{2.57} one obtains
\begin{align} \lb{2.60}
& (z_2-z_1)\int_{x_0}^{\pm\infty} dx\, \big(\psi_{\pm,\alpha}(\ol{z_{1}},x,x_0)f,\psi_{\pm,\alpha}(z_{2},x,x_0)g\big)_\cH
\no \\
&\quad = \big(f,[m_{\pm,\alpha}(z_2,x_0)- m_{\pm,\alpha}(\ol{z_1},x_0)^*]g\big)_\cH,
\quad  f,g\in\cH, \; z_1,z_2 \in\bbC\backslash\bbR.
\end{align}
Setting $z_1=z_2=z$ in \eqref{2.60}, one concludes
\begin{equation}
m_{\pm,\alpha}(z,x_0) = m_{\pm,\alpha}(\ol z,x_0)^*, \quad
z\in\bbC\backslash\bbR.  \lb{2.61}
\end{equation}
Choosing $f=g$ and $z_2=z$, $z_1=\ol z$ in \eqref{2.60}, one also infers
\begin{equation}
\Im(z)\int_{x_0}^{\pm\infty} dx\,\|\psi_{\pm,\alpha}(z,x,x_0)f\|_{\cH}^2
= \big(f,\Im[m_{\pm,\alpha}(z,x_0)]f\big)_\cH, \quad f\in\cH, \; z\in\bbC\backslash\bbR. \lb{2.62}
\end{equation}
Since $m_{\pm,\alpha}(\cdot,x_0)$ are analytic on $\bbC\backslash\bbR$, \eqref{2.62} yields that
$\pm m_{\pm,\alpha}(\cdot,x_0)$ are operator-valued Herglotz functions.

In the following we abbreviate the Wronskian of $\psi_{+,\al}(\ol{z},x,x_0)^*$ and $\psi_{-,\al}(z,x,x_0)$
by $W(z)$. It follows from the identities \eqref{2.7f}--\eqref{2.7i} and \eqref{2.61} that
\begin{align}
W(z) &= W(\psi_{+,\al}(\ol{z},x,x_0)^*,\psi_{-,\al}(z,x,x_0))
\no \\
&= m_{-,\al}(z,x_0) - m_{+,\al}(z,x_0), \quad z\in\bbC\bs\bbR. \lb{2.64}
\end{align}
The Green's function $G(z,x,x')$ of the Schr\"odinger operator $H$ then reads
\begin{align}
G(z,x,x') = \psi_{\mp,\alpha}(z,x,x_0) W(z)^{-1} \psi_{\pm,\alpha}(\ol{z},x',x_0)^*,
\quad x \lesseqgtr x', \; z\in\bbC\backslash\bbR.
\lb{2.63}
\end{align}
Thus,
\begin{equation}
((H-zI_\cH)^{-1}f)(x)
=\int_{\bbR} dx' \, G(z,x,x')f(x'), \quad z\in\bbC\backslash\bbR, \;
x\in\bbR, \; f\in L^2(\bbR;dx;\cH). \lb{2.65}
\end{equation}

Next, we introduce the $2\times 2$ block operator-valued Weyl--Titchmarsh m-function,
$M_\alpha(z,x_0)\in\cB\big(\cH^2\big)$,
\begin{align}
M_\alpha(z,x_0)&=\big(M_{\alpha,j,j'}(z,x_0)\big)_{j,j'=0,1}, \quad z\in\bbC\backslash\bbR,     \lb{2.71}
\\
M_{\alpha,0,0}(z,x_0) &= W(z)^{-1},
\\
M_{\alpha,0,1}(z,x_0) &= 2^{-1} W(z)^{-1} \big[m_{-,\alpha}(z,x_0)+m_{+,\alpha}(z,x_0)\big],
\\
M_{\alpha,1,0}(z,x_0) &= 2^{-1} \big[m_{-,\alpha}(z,x_0)+m_{+,\alpha}(z,x_0)\big] W(z)^{-1},
\\
M_{\alpha,1,1}(z,x_0) &= m_{+,\alpha}(z,x_0) W(z)^{-1} m_{-,\alpha}(z,x_0)
\no
\\
&= m_{-,\alpha}(z,x_0) W(z)^{-1} m_{+,\alpha}(z,x_0).   \lb{2.71a}
\end{align}
$M_\alpha(z,x_0)$ is a $\cB\big(\cH^2\big)$-valued Herglotz function with representation
\begin{align}
\begin{split}
& M_\alpha(z,x_0)=C_\alpha(x_0)+\int_{\bbR}
d\Omega_\alpha (\lambda,x_0)\bigg[\frac{1}{\lambda -z}-\frac{\lambda}
{1+\lambda^2}\bigg], \quad z\in\bbC\backslash\bbR, \lb{2.71b} \\
& C_\alpha(x_0)=C_\alpha(x_0)^*, \quad \int_{\bbR}
\f{\big(e,d\Omega_{\alpha}(\lambda,x_0)e\big)_{\cB\big(\cH^2\big)}}{1+\lambda^2} <\infty, \quad e\in\cK.
\end{split}
\end{align}
In addition, the Stieltjes inversion formula for the nonnegative $\cB \big(\cH^2\big)$-valued measure $d\Omega_\alpha(\cdot,x_0)$ reads
\begin{equation}
\Omega_\alpha((\lambda_1,\lambda_2],x_0)
=\f1\pi \lim_{\delta\downarrow 0}
\lim_{\varepsilon\downarrow 0} \int^{\lambda_2+\delta}_{\lambda_1+\delta}
d\lambda \, \Im(M_\alpha(\lambda +i\varepsilon,x_0)), \quad \lambda_1,
\lambda_2 \in\bbR, \; \lambda_1<\lambda_2. \lb{2.71c}
\end{equation}
In particular, $d\Omega_\alpha(\cdot,x_0)$ is a $2\times 2$ block operator-valued measure with $\cB(\cH)$-valued entries $d\Omega_{\al,\ell,\ell'}(\cdot,x_0)$, $\ell,\ell'=0,1$. Since the diagonal entries of $M_\al$ are Herglotz functions, the diagonal entries of the measure $d\Omega_\alpha(\cdot,x_0)$ are nonnegative $\cB(\cH)$-valued measures. The off-diagonal entries of the measure $d\Omega_\alpha(\cdot,x_0)$ naturally admit decompositions into a linear combination of four nonnegative measures.

We note that in formulas \eqref{2.57}--\eqref{2.71b} one can
replace $z\in\bbC\backslash\bbR$ by $z\in\bbC\backslash\sigma(H)$.

\medskip

Next, we relate the family of spectral projections,
$\{E_H(\lambda)\}_{\lambda\in\bbR}$, of the self-adjoint
operator $H$ and the $2\times 2$ operator-valued increasing spectral
function $\Omega_{\alpha}(\lambda,x_0)$, $\lambda\in\bbR$, which
generates the $\cB\big(\cH^2\big)$-valued measure $d\Omega_\alpha(\cdot,x_0)$ in the Herglotz representation \eqref{2.71b} of $M_\alpha(z,x_0)$.

We first note that for $F\in C(\bbR)$,
\begin{align}
&\big(f,F(H)g\big)_{L^2(\bbR;dx;\cH)}= \int_{\bbR}
d\,\big(f,E_H(\lambda)g\big)_{L^2(\bbR;dx;\cH)}\,
F(\lambda), \lb{2.72} \\
& f, g \in\dom(F(H)) =\bigg\{h\in L^2(\bbR;dx;\cH) \,\bigg|\,
\int_{\bbR} d \|E_H(\lambda)h\|_{L^2(\bbR;dx;\cH)}^2 \, |F(\lambda)|^2
< \infty\bigg\}. \no
\end{align}

\begin{theorem} \lb{t2.9}
Let $\alpha\in [0,\pi)$, $f,g \in C^\infty_0(\bbR;\cH)$,
$F\in C(\bbR)$, $x_0\in\bbR$, and $\lambda_1, \lambda_2 \in\bbR$,
$\lambda_1<\lambda_2$. Then,
\begin{align} \lb{2.73}
&\big(f,F(H)E_H((\lambda_1,\lambda_2])g\big)_{L^2(\bbR;dx;\cH)}
\no \\
&\quad =
\big(\hatt
f_{\alpha}(\cdot,x_0),M_FM_{\chi_{(\lambda_1,\lambda_2]}} \hatt
g_{\alpha}(\cdot,x_0)\big)_{L^2(\bbR;d\Omega_{\alpha}(\cdot,x_0);\cH^2)}
\end{align}
where we introduced the notation
\begin{align} \lb{2.74}
&\hatt h_{\alpha,0}(\lambda,x_0) = \int_\bbR dx \,
\theta_\alpha(\lambda,x,x_0)^* h(x),  \quad
\hatt h_{\alpha,1}(\lambda,x_0) = \int_\bbR dx \,
\phi_\alpha(\lambda,x,x_0)^* h(x),
\no
\\
&\hatt h_{\alpha}(\lambda,x_0) = \big(\,\hatt h_{\alpha,0}(\lambda,x_0),
\hatt h_{\alpha,1}(\lambda,x_0)\big)^\top,  \quad
\lambda \in\bbR, \;  h\in C^\infty_0(\bbR;\cH),
\end{align}
and $M_G$ denotes the maximally defined operator of multiplication
by the function $G \in C(\bbR)$ in the
Hilbert space\footnote{Again, we recall
that $L^2\big(\bbR;d\Omega_{\alpha}(\cdot,x_0);\cH^2\big)$ is a convenient
abbreviation for the Hilbert space
$L^2(\bbR;d\mu_{\alpha}(\cdot,x_0);\cM_{\Omega_{\alpha}(\cdot,x_0)})$ discussed
in detail in Appendix \ref{sD}, with $d\mu_{\alpha}(\cdot,x_0)$ a control measure for the
$\cB\big(\cH^2\big)$-valued measure $d\Omega_{\alpha}(\cdot,x_0)$. One recalls that
$\cM_{\Omega_{\alpha}(\cdot,x_0)} \subset \cS(\{\cK_{\lambda}\}_{\lambda\in \bbR})$ is
generated by ${\ul \Lambda} \big(\cH^2\big)$ (or by ${\ul \Lambda} (\{f_n\}_{n\in\cI})$ for
any complete orthonormal system $\{f_n\}_{n\in\cI}$, $\cI \subseteq \bbN$.}
$L^2\big(\bbR;d\Omega_{\alpha}(\cdot,x_0);\cH^2\big)$,
\begin{align}
\begin{split}
& \big(M_G\hatt h\big)(\lambda)=G(\lambda)\hatt h(\lambda)
=\big(G(\lambda) \hatt h_0(\lambda), G(\lambda) \hatt h_1(\lambda)\big)^\top
\, \text{ for \ $\Omega_{\alpha}(\cdot,x_0)$-a.e.\ $\lambda\in\bbR$}, \lb{2.75} \\
& \hatt h\in\dom(M_G)=\big\{\hatt k \in L^2\big(\bbR;d\Omega_{\alpha}(\cdot,x_0);\cH^2\big) \,\big|\,
G\hatt k \in L^2(\bbR;d\Omega_{\alpha}\big(\cdot,x_0);\cH^2\big)\big\}.
\end{split}
\end{align}
\end{theorem}
\begin{proof}
The point of departure for deriving \eqref{2.73} is again Stone's
formula \eqref{2.26a} applied to $T=H$,
\begin{align}
&\big(f,F(H)E_H((\lambda_1,\lambda_2])g\big)_{L^2(\bbR;dx;\cH)}
\no \\
& \quad = \lim_{\delta\downarrow 0}\lim_{\varepsilon\downarrow 0}
\frac{1}{2\pi i} \int_{\lambda_1+\delta}^{\lambda_2+\delta}
d\lambda \, F(\lambda) \big[\big(f,(H-(\lambda+i\varepsilon)I_\cH)^{-1}g\big)_{L^2(\bbR;dx;\cH)}
\no \\
& \hspace*{4.9cm} - \big(f,(H-(\lambda-i\varepsilon)I_\cH)^{-1}g\big)_{L^2(\bbR;dx;\cH)}\big]. \lb{2.82}
\end{align}
Insertion of \eqref{2.63} and \eqref{2.65} into \eqref{2.82} then
yields the following:
\begin{align} \lb{2.83}
&\big(f,F(H)E_H((\lambda_1,\lambda_2])g\big)_{L^2(\bbR;dx;\cH)} =
\lim_{\delta\downarrow 0}\lim_{\varepsilon\downarrow 0}
\frac{1}{2\pi i}
\int_{\lambda_1+\delta}^{\lambda_2+\delta} d\lambda \, F(\lambda) \int_\bbR dx
\\
& \ \times \bigg\{
\Big(f(x), \psi_{+,\alpha}(\lambda+i\varepsilon,x,x_0) W(\lambda+i\varepsilon)^{-1} \int_{-\infty}^x dx'\,
\psi_{-,\alpha}(\lambda-i\varepsilon,x',x_0)^* g(x')\Big)_\cH
\no
\\
& \quad + \Big(f(x),
\psi_{-,\alpha}(\lambda+i\varepsilon,x,x_0)
W(\lambda+i\varepsilon)^{-1} \int_x^\infty dx' \, \psi_{+,\alpha}(\lambda-i\varepsilon,x',x_0)^* g(x')\Big)_\cH
\no
\\
& \quad -
\Big(f(x),
\psi_{+,\alpha}(\lambda-i\varepsilon,x,x_0)
W(\lambda-i\varepsilon)^{-1} \int_{-\infty}^x dx' \, \psi_{-,\alpha}(\lambda+i\varepsilon,x',x_0)^* g(x')\Big)_\cH
\no
\\
& \quad - \Big(f(x),
\psi_{-,\alpha}(\lambda-i\varepsilon,x,x_0) W(\lambda-i\varepsilon)^{-1}
\int_x^\infty dx'\, \psi_{+,\alpha}(\lambda+i\varepsilon,x',x_0)^*
g(x')\Big)_\cH \bigg\}.
\no
\end{align}
Freely interchanging the $dx$ and $dx'$ integrals with the limits and the
$d\lambda$ integral (since all integration domains are finite and all
integrands are continuous), and inserting the expressions \eqref{2.59} for
$\psi_{\pm,\alpha}(z,x,x_0)$ into \eqref{2.83}, one obtains
\begin{align}
&\big(f,F(H)E_H((\lambda_1,\lambda_2])g\big)_{L^2(\bbR;dx;\cH)}
=\int_{\bbR} dx \bigg(f(x), \bigg\{ \int_{-\infty}^x dx' \, \no
\\
& \quad \times \lim_{\delta\downarrow
0}\lim_{\varepsilon\downarrow 0} \frac{1}{2\pi i}
\int_{\lambda_1+\delta}^{\lambda_2+\delta} d\lambda \, F(\lambda)
\Big[\big[\theta_\alpha(\lambda,x,x_0) +
\phi_\alpha(\lambda,x,x_0) m_{+,\alpha}(\lambda+i\varepsilon,x_0) \big] \no
\\
& \hspace*{1.2cm} \times W(\lambda+i\varepsilon)^{-1}
\big[\theta_\alpha(\lambda,x',x_0)^* + m_{-,\alpha}(\lambda+i\varepsilon,x_0) \phi_\alpha(\lambda,x',x_0)^* \big]
g(x') \no
\\
& \qquad -\big[\theta_\alpha(\lambda,x,x_0) +
\phi_\alpha(\lambda,x,x_0) m_{+,\alpha}(\lambda-i\varepsilon,x_0) \big] \no
\\
& \hspace*{1.2cm} \times W(\lambda-i\varepsilon)^{-1}
\big[\theta_\alpha(\lambda,x',x_0)^* + m_{-,\alpha}(\lambda-i\varepsilon,x_0) \phi_\alpha(\lambda,x',x_0)^* \big]
g(x') \Big] \no
\\
& \quad +\int_x^\infty dx'\, \lim_{\delta\downarrow 0}
\lim_{\varepsilon\downarrow 0} \frac{1}{2\pi i}
\int_{\lambda_1+\delta}^{\lambda_2+\delta} d\lambda \, F(\lambda)
\lb{2.85} \\
& \qquad \times \Big[\big[\theta_\alpha(\lambda,x,x_0) +
\phi_\alpha(\lambda,x,x_0) m_{-,\alpha}(\lambda+i\varepsilon,x_0) \big]
\no \\
& \hspace*{1.2cm} \times W(\lambda+i\varepsilon)^{-1}
\big[\theta_\alpha(\lambda,x',x_0)^* +
m_{+,\alpha}(\lambda+i\varepsilon,x_0) \phi_\alpha(\lambda,x',x_0)^* \big]
g(x') \no
\\
& \qquad -\big[\theta_\alpha(\lambda,x,x_0) +
\phi_\alpha(\lambda,x,x_0) m_{-,\alpha}(\lambda-i\varepsilon,x_0) \big]
\no
\\
& \hspace*{1.2cm} \times W(\lambda-i\varepsilon)^{-1} \big[\theta_\alpha(\lambda,x',x_0)^* +
m_{+,\alpha}(\lambda-i\varepsilon,x_0) \phi_\alpha(\lambda,x',x_0)^* \big]
g(x') \Big]\bigg\}\bigg)_\cH.
\no
\end{align}
Here we employed \eqref{2.61}, the fact that for fixed $x\in\bbR$,
$\theta_\alpha(z,x,x_0)$ and $\phi_\alpha(z,x,x_0)$ are entire
with respect to $z$, that $\theta_\alpha(z,\cdot,x_0), \phi_\alpha(z,\cdot,x_0) \in W^{1,1}_{\loc}(\bbR;\cH)$, and hence that
\begin{align}
\begin{split}
\theta_\alpha(\lambda\pm i\varepsilon,x,x_0)
&\underset{\varepsilon\downarrow 0}{=}
\theta_\alpha(\lambda,x,x_0) \pm
i\varepsilon (d/dz)\theta_\alpha(z,x,x_0)|_{z=\lambda} + \Oh(\ve^2),
\\
\phi_\alpha(\lambda\pm i\varepsilon,x,x_0)
&\underset{\varepsilon\downarrow 0}{=} \phi_\alpha(\lambda,x,x_0)
\pm i\varepsilon (d/dz)\phi_\alpha(z,x,x_0)|_{z=\lambda}
+ \Oh(\ve^2)
\end{split} \lb{2.85a}
\end{align}
with $\Oh(\varepsilon^2)$ being uniform with respect to $(\lambda,x)$ as long as $\lambda$ and $x$ vary in compact subsets of $\bbR$. Moreover, we used that
\begin{align}
&\varepsilon\|M_{\al}(\lambda+i\varepsilon,x_0)\|_{\cB(\cH^2)}\leq
C(\lambda_1,\lambda_2,\varepsilon_0,x_0), \quad \lambda\in
[\lambda_1,\lambda_2], \; 0<\varepsilon\leq\varepsilon_0, \no \\
&\varepsilon
\|\Re(M_{\al}(\lambda+i\varepsilon,x_0))\|_{\cB(\cH^2)}
\underset{\varepsilon\downarrow 0}{=}\oh(1), \quad \lambda\in\bbR, \lb{2.86}
\end{align}
In particular, utilizing
\eqref{2.61}, \eqref{2.85a},
\eqref{2.86}, and the elementary facts
\begin{align} \lb{2.88}
&\Im\big[m_{\pm,\alpha}(\lambda+i\varepsilon,x_0) W(\lambda+i\varepsilon)^{-1}\big]
\no
\\
&\quad = \f{1}{2}
\Im\big[[m_{-,\alpha}(\lambda+i\varepsilon,x_0) + m_{+,\alpha}(\lambda+i\varepsilon,x_0)] W(\lambda+i\varepsilon)^{-1}\big],
\no
\\
&\Im\big[W(\lambda+i\varepsilon)^{-1} m_{\pm,\alpha}(\lambda+i\varepsilon,x_0)\big] \\
&\quad = \f{1}{2}
\Im\big[W(\lambda+i\varepsilon)^{-1} [m_{-,\alpha}(\lambda+i\varepsilon,x_0) + m_{+,\alpha}(\lambda+i\varepsilon,x_0)]\big],
\quad \lambda\in\bbR, \; \varepsilon >0,
\no
\end{align}
$\phi_\alpha(\lambda\pm i\varepsilon,x,x_0)$ and
$\theta_\alpha(\lambda\pm i\varepsilon,x,x_0)$ under the $d\lambda$
integrals in \eqref{2.85} have immediately been replaced by
$\phi_\alpha(\lambda,x,x_0)$ and $\theta_\alpha(\lambda,x,x_0)$. Collecting
appropriate terms in
\eqref{2.85} then yields
\begin{align} \lb{2.89}
&\big(f,F(H)E_{H}((\lambda_1,\lambda_2])g\big)_{L^2(\bbR;dx;\cH)}
= \int_\bbR dx \bigg(f(x), \int_\bbR dx'
\lim_{\delta\downarrow 0}\lim_{\varepsilon\downarrow 0}
\frac{1}{\pi} \int_{\lambda_1+\delta}^{\lambda_2+\delta} d\lambda
\, F(\lambda)
\no \\
& \ \times\Big\{
\theta_\alpha(\lambda,x,x_0)
\Im\big[W(\lambda+i\varepsilon)^{-1}\big]
\theta_\alpha(\lambda,x',x_0)^*
\\
& \quad
+ 2^{-1} \theta_\alpha(\lambda,x,x_0)
\no \\
& \qquad
\times \Im\big[W(\lambda+i\varepsilon)^{-1} [m_{-,\alpha}(\lambda+i\varepsilon,x_0) + m_{+,\alpha}(\lambda+i\varepsilon,x_0)]\big]
\phi_\alpha(\lambda,x',x_0)^*
\no \\
& \quad
+ 2^{-1} \phi_\alpha(\lambda,x,x_0)
\no \\
& \qquad
\times \Im\big[[m_{-,\alpha}(\lambda+i\varepsilon,x_0)+
m_{+,\alpha}(\lambda+i\varepsilon,x_0)] W(\lambda+i\varepsilon)^{-1}\big] \theta_\alpha(\lambda,x',x_0)^*
\no \\
& \quad
+\phi_\alpha(\lambda,x,x_0)
\no \\
& \qquad
\times
\Im\big[m_{-,\alpha}(\lambda+i\varepsilon,x_0)
W(\lambda+i\varepsilon)^{-1}
m_{+,\alpha}(\lambda+i\varepsilon,x_0)\big]
\phi_\alpha(\lambda,x',x_0)^*\Big\} g(x') \bigg)_\cH. \no
\end{align}
Since by \eqref{2.71c} ($\ell, \ell'=0,1$)
\begin{align}
\begin{split}
&\int_{(\lambda_1,\lambda_2]} d\Omega_{\alpha,\ell,\ell'}(\lambda,x_0)
= \Omega_{\alpha,\ell,\ell'}((\lambda_1,\lambda_2],x_0) \\
& \quad =
\lim_{\delta\downarrow 0}\lim_{\varepsilon\downarrow 0}
\frac{1}{\pi}\int_{\lambda_1+\delta}^{\lambda_2+\delta} d\lambda \,
\Im(M_{\alpha,\ell,\ell'}(\lambda+i\varepsilon,x_0)), \lb{2.90}
\end{split}
\end{align}
one also has
\begin{align}
&\int_{\bbR} d\Omega_{\alpha,\ell,\ell'}(\lambda,x_0)\, h(\lambda) =
\lim_{\varepsilon\downarrow 0} \frac{1}{\pi}\int_{\bbR} d\lambda \,
\Im(M_{\alpha,\ell,\ell'}(\lambda+i\varepsilon,x_0))\, h(\lambda),
\quad h\in C_0(\bbR;\cH),  \lb{2.91} \\
&\int_{(\lambda_1,\lambda_2]} d\Omega_{\alpha,\ell,\ell'}(\lambda,x_0)
\, k(\lambda) =
\lim_{\delta\downarrow 0} \lim_{\varepsilon\downarrow 0} \frac{1}{\pi}
\int_{\lambda_1+\delta}^{\lambda_2+\delta} d\lambda \,
\Im(M_{\alpha,\ell,\ell'}(\lambda+i\varepsilon,x_0))\, k(\lambda), \no
\\
& \hspace*{9.4cm} k\in C(\bbR;\cH), \lb{2.92}
\end{align}
Then using \eqref{2.71}--\eqref{2.71a}, \eqref{2.74}, and interchanging the $dx$, $dx'$ and $d\Omega_{\alpha,\ell,\ell'}(\cdot,x_0)$, $\ell,\ell'=0,1$, integrals
once more, one concludes from \eqref{2.89}
\begin{align}
&\big(f,F(H)E_H((\lambda_1,\lambda_2])g\big)_{L^2(\bbR;dx;\cH)}
=\int_\bbR dx\, \bigg(f(x), \int_\bbR dx' \int_{(\lambda_1,\lambda_2]} F(\lambda)
\no
\\
& \qquad \times
\Big\{\theta_\alpha(\lambda,x,x_0) \, d\Omega_{\alpha,0,0}(\lambda,x_0) \, \theta_\alpha(\lambda,x',x_0)^*
\no
\\
& \hspace*{1.25cm}
+ \theta_\alpha(\lambda,x,x_0) \, d\Omega_{\alpha,0,1}(\lambda,x_0) \, \phi_\alpha(\lambda,x',x_0)^*
\no
\\
& \hspace*{1.25cm}
+ \phi_\alpha(\lambda,x,x_0) \, d\Omega_{\alpha,1,0}(\lambda,x_0) \, \theta_\alpha(\lambda,x',x_0)^*
\no \\
& \hspace*{1.25cm}
+\phi_\alpha(\lambda,x,x_0) \, d\Omega_{\alpha,1,1}(\lambda,x_0) \, \phi_\alpha(\lambda,x',x_0)^* \Big\} g(x') \bigg)_\cH
\no
\\
& \quad = \int_{(\lambda_1,\lambda_2]} F(\lambda) \, \big(\hatt
f_{\alpha}(\lambda,x_0), d\Omega_{\alpha}(\lambda,x_0) \,
\hatt g_{\alpha}(\lambda,x_0) \big)_{\cH^2}. \lb{2.93}
\end{align}
\end{proof}

\begin{remark} \lb{r2.13a}
Again we emphasize that the idea of a straightforward derivation of the
link between the family of spectral projections $E_{H}(\cdot)$ and the
$2\times 2$ block operator-valued spectral function $\Omega_{\alpha}(\cdot)$ of
$H$ in Theorem \ref{t2.9} can already be found in \cite{HS98} as pointed
out in Remark \ref{r2.6a}. It applies equally well to Dirac-type operators
and Hamiltonian systems on $\bbR$ (see the extensive literature cited,
e.g., in \cite{CG02}) and to Jacobi and CMV operators on $\bbZ$
(cf.\ \cite{Be68} and \cite{GZ06a}).
\end{remark}

As in the half-line case, one can improve on Theorem \ref{t2.9}
and remove the compact support restrictions on $f$ and $g$ in the usual
way. To this end one considers the map
\begin{align}
&\widetilde U_{\alpha}(x_0) : \begin{cases} C_0^\infty(\bbR)\to
L^2\big(\bbR;d\Omega_{\alpha}(\cdot,x_0);\cH^2\big)
\\[1mm]
h \mapsto \hatt h_{\alpha}(\cdot,x_0)
=\big(\,\hatt h_{\alpha,0}(\lambda,x_0),
\hatt h_{\alpha,1}(\lambda,x_0)\big)^\top, \end{cases} \lb{2.95}
\\
&  \hatt h_{\alpha,0}(\lambda,x_0)=\int_\bbR dx \,
\theta_\alpha(\lambda,x,x_0)^* h(x), \quad \hatt h_{\alpha,1}(\lambda,x_0)=\int_\bbR dx \, \phi_\alpha(\lambda,x,x_0)^* h(x).  \no
\end{align}
Taking $f=g$, $F=1$, $\lambda_1\downarrow -\infty$, and
$\lambda_2\uparrow \infty$ in \eqref{2.73} then shows that $\widetilde
U_{\alpha}(x_0)$ is a densely defined isometry in $L^2(\bbR;dx;\cH)$,
which extends by continuity to an isometry on $L^2(\bbR;dx;\cH)$. The latter
is denoted by $U_{\alpha}(x_0)$ and given by
\begin{align}
&U_{\alpha}(x_0) : \begin{cases} L^2 (\bbR;dx;\cH)\to
L^2\big(\bbR;d\Omega_{\alpha}(\cdot,x_0);\cH^2\big) \\[1mm]
h \mapsto \hatt h_{\alpha}(\cdot,x_0)
= \big(\,\hatt h_{\alpha,0}(\cdot,x_0),
\hatt h_{\alpha,1}(\cdot,x_0)\big)^\top, \end{cases} \lb{2.96} \\
& \hatt h_\alpha(\cdot,x_0)=\begin{pmatrix}
\hatt h_{\alpha,0}(\cdot,x_0) \\
\hatt h_{\alpha,1}(\cdot,x_0) \end{pmatrix}=
\slimes_{a\downarrow -\infty, b \uparrow\infty} \begin{pmatrix}
\int_{a}^b dx \, \theta_\alpha(\cdot,x,x_0)^* h(x) \\
\int_{a}^b dx \, \phi_\alpha(\cdot,x,x_0)^* h(x) \end{pmatrix}, \no
\end{align}
where $\slimes$ refers to the
$L^2\big(\bbR;d\Omega_{\alpha}(\cdot,x_0);\cH^2\big)$-limit.

The calculation in \eqref{2.93} also yields
\begin{align}
&(E_{H}((\lambda_1,\lambda_2])g)(x) = \int_{(\lambda_1,\lambda_2]}
(\theta_\alpha(\lambda,x,x_0), \phi_\alpha(\lambda,x,x_0)) \,
d\Omega_{\alpha}(\lambda,x_0)\,
\hatt g_{\alpha}(\lambda,x_0)
\no \\
&\quad = \int_{(\lambda_1,\lambda_2]}
\big[\theta_\alpha(\lambda,x,x_0)  \,
d\Omega_{\alpha,0,0}(\lambda,x_0) \,
\hatt g_{\alpha,0}(\lambda,x_0)
\no \\
&\qquad + \theta_\alpha(\lambda,x,x_0) \,
d\Omega_{\alpha,0,1}(\lambda,x_0) \,
\hatt g_{\alpha,1}(\lambda,x_0)
 +
\phi_\alpha(\lambda,x,x_0) \,
d\Omega_{\alpha,1,0}(\lambda,x_0) \,
\hatt g_{\alpha,0}(\lambda,x_0)
\no \\
& \qquad + \phi_\alpha(\lambda,x,x_0) \,
d\Omega_{\alpha,1,1}(\lambda,x_0) \,
\hatt g_{\alpha,1}(\lambda,x_0)
\big], \quad g\in C_0^\infty(\bbR) \lb{2.97}
\end{align}
and subsequently, \eqref{2.97} extends to all $g\in L^2(\bbR;dx;\cH)$
by continuity. Moreover, taking $\lambda_1\downarrow -\infty$ and
$\lambda_2\uparrow \infty$ in \eqref{2.97} and using
\begin{equation}
\slim_{\lambda\downarrow -\infty} E_{H}(\lambda)=0, \quad
\slim_{\lambda\uparrow \infty} E_{H}(\lambda)=I_{L^2(\bbR;dx;\cH)},
\lb{2.98}
\end{equation}
where
\begin{equation}
E_{H}(\lambda)=E_{H}((-\infty,\lambda]), \quad \lambda\in\bbR, \lb{2.99}
\end{equation}
then yield
\begin{align}
g(\cdot) &=\slimes_{\mu_1\downarrow -\infty, \mu_2\uparrow \infty}
\int_{(\mu_1,\mu_2]} (\theta_\alpha(\lambda,\cdot,x_0),
\phi_\alpha(\lambda,\cdot,x_0)) \, d\Omega_{\alpha}(\lambda,x_0)\,
\hatt g_{\alpha}(\lambda,x_0) \no \\
&= \slimes_{\mu_1\downarrow -\infty, \mu_2\uparrow \infty}
\int_{\mu_1}^{\mu_2}
\big[\theta_\alpha(\lambda,\cdot,x_0) \,
d\Omega_{\alpha,0,0}(\lambda,x_0) \,
\hatt g_{\alpha,0}(\lambda,x_0)
\no \\
&\quad +
\theta_\alpha(\lambda,\cdot,x_0) \,
d\Omega_{\alpha,0,1}(\lambda,x_0) \,
\hatt g_{\alpha,1}(\lambda,x_0)
+\phi_\alpha(\lambda,\cdot,x_0) \,
d\Omega_{\alpha,1,0}(\lambda,x_0) \,
\hatt g_{\alpha,0}(\lambda,x_0)
\no \\
& \quad + \phi_\alpha(\lambda,\cdot,x_0) \,
d\Omega_{\alpha,1,1}(\lambda,x_0) \,
\hatt g_{\alpha,1}(\lambda,x_0)
\big], \quad g\in L^2(\bbR;dx;\cH), \lb{2.100}
\end{align}
where $\slimes$ refers to the $L^2(\bbR;dx;\cH)$-limit. In addition, one can show
that the map $U_{\alpha}(x_0)$ in \eqref{2.96} is onto and hence that
$U_{\alpha}(x_0)$ is unitary with
\begin{align}
&U_{\alpha}(x_0)^{-1} : \begin{cases}
L^2\big(\bbR;d\Omega_{\alpha}(\cdot,x_0);\cH^2\big) \to  L^2(\bbR;dx;\cH)  \\[1mm]
\hatt h \mapsto  h_\alpha, \end{cases} \lb{2.101} \\
& h_\alpha(\cdot)= \slimes_{\mu_1\downarrow -\infty, \mu_2\uparrow
\infty} \int_{\mu_1}^{\mu_2}(\theta_\alpha(\lambda,\cdot,x_0),
\phi_\alpha(\lambda,\cdot,x_0)) \, d\Omega_{\alpha}(\lambda,x_0)\,
\hatt h(\lambda). \no
\end{align}
Indeed, denoting the operator defined in \eqref{2.101} temporarily by
$V_\alpha(x_0)$, one can closely follow the arguments in the corresponding
half-line case in \eqref{2.44a}--\eqref{2.44c}. After first proving that
$V_\alpha(x_0)$ is bounded, one then assumes that
$\hatt f=(f_0,\; f_1)^{\top}\in \ker(V_{\alpha}(x_0)) \subset
L^2\big(\bbR;d\Omega_{\alpha}(\cdot,x_0);\cH^2\big)$. As in the half-line case, one can
show using the dominated convergence theorem that this implies that
for all $x\in\bbR$ and $e\in\cH$,
\begin{align} \lb{2.101c}
\begin{split}
& \int_{(\la_1,\la_2]} \big(e, (\theta_\al(\la,x,x_0),\phi_\al(\la,x,x_0)) \, d\Omega_{\al}(\la,x_0) \, \hatt f(\la) \big)_\cH = 0,  \\
& \int_{(\la_1,\la_2]} \big(e, (\theta_\al'(\la,x,x_0),\phi_\al'(\la,x,x_0)) \,
d\Omega_{\al}(\la,x_0) \, \hatt f(\la) \big)_\cH  = 0.
\end{split}
\end{align}
Using the fact that $\hatt f, \, \chi_{(\la_1,\la_2]}(e_0,e_1)^\top \in
L^2\big(\bbR;d\Omega_{\al}(\la,x_0);\cH^2\big)$ for all $\la_1<\la_2$, $e_0,e_1\in\cH$,
that $\theta_\al(\la,x,x_0),\theta_\al'(\la,x,x_0), \phi_\al(\la,x,x_0),\phi_\al'(\la,x,x_0)$
are continuous with respect to $(\la,x)\in\bbR^2$, and the dominated convergence theorem once again, one finally concludes that for all $e\in\cH$,
\begin{align}
&\int_{(\la_1,\la_2]} \big(e, (\theta_\al(\la,x_0,x_0),\phi_\al(\la,x_0,x_0)) \,
d\Omega_{\al}(\la,x_0) \, \hatt f(\la) \big)_\cH
\no \\
& \quad =
\int_{(\la_1,\la_2]} \big((\cos(\al)e,-\sin(\al)e)^\top, \,
d\Omega_{\al}(\la,x_0) \, \hatt f(\la) \big)_{\cH^2} = 0,  \lb{2.101e}
\\
&\int_{(\la_1,\la_2]} \big(e, (\theta_\al'(\la,x_0,x_0),\phi_\al'(\la,x_0,x_0)) \,
d\Omega_{\al}(\la,x_0) \, \hatt f(\la) \big)_\cH
\no \\
& \quad =
\int_{(\la_1,\la_2]} \big((\sin(\al)e,\cos(\al)e)^\top, \,
d\Omega_{\al}(\la,x_0) \, \hatt f(\la) \big)_{\cH^2} = 0. \lb{2.101f}
\end{align}
The sum of \eqref{2.101e} with $e=\cos(\al)e_1-\sin(\al)e_2$ and \eqref{2.101f} with $e=\sin(\al)e_1+\cos(\al)e_2$ then yields,
\begin{align}
\int_{(\la_1,\la_2]} \big((e_1,e_2)^\top, \,
d\Omega_{\al}(\la,x_0) \, \hatt f(\la) \big)_{\cH^2} = 0, \quad (e_1,e_2)^\top\in\cH^2. \lb{2.101g}
\end{align}
Since the interval $(\la_1,\la_2]$ was chosen arbitrary, \eqref{2.101g}
implies
\begin{align}
\hatt f(\la) = 0 \;\, \Omega_{\al}(\cdot,x_0)\text{-a.e.},
\end{align}
and hence that $V_{\al}(x_0)$ is injective and thus $U_{\al}(x_0)$ is onto.

We sum up these considerations in a variant of the spectral theorem for
(functions of) $H$.

\begin{theorem} \lb{t2.10}
Let $F\in C(\bbR)$ and $x_0\in\bbR$. Then,
\begin{equation}
U_{\alpha}(x_0) F(H)U_{\alpha}(x_0)^{-1} = M_F \lb{2.102}
\end{equation}
in $L^2\big(\bbR;d\Omega_{\alpha}(\cdot,x_0);\cH^2\big)$ $($cf.\ \eqref{2.75}$)$.
Moreover,
\begin{align}
& \sigma(F(H))= \essran_{\Omega_{\alpha}}(F),  \lb{2.103a} \\
& \sigma(H)=\supp (d\Omega_\alpha(\cdot,x_0)),    \lb{2.103b}
\end{align}
and the multiplicity of the spectrum of $H$ is at most equal to $\dim\big(\cH^2\big)=2\dim(\cH)$.
\end{theorem}
\begin{proof}
The proof of the theorem is analogous to the one given for Theorem \ref{t2.6}.
\end{proof}

\appendix
\section{Basic Facts on Operator-Valued Herglotz Functions} \lb{sA}
\setcounter{theorem}{0}
\setcounter{equation}{0}

In this appendix we review some basic facts on (bounded) operator-valued Herglotz
functions (also called Nevanlinna, Pick, $R$-functions, etc.),
applicable to $m_{\alpha}$ and $G_{\alpha}(\cdot,x,x)$, $x\in (a,b)$,
discussed in the bulk of this paper. For additional details concerning the material in this appendix
we refer to \cite{GWZ13}.

In the remainder of this appendix, let $\cH$ be a separable, complex Hilbert space
with inner product denoted by $(\cdot,\cdot)_{\cH}$.

\begin{definition}\label{dA.4}
The map $M: \bbC_+ \rightarrow \cB(\cH)$ is called a bounded
operator-valued Herglotz function in $\cH$ (in short, a bounded Herglotz operator in
$\cH$) if $M$ is analytic on $\bbC_+$ and $\Im (M(z))\geq 0$ for all $z\in \bbC_+$.
\end{definition}

Here we follow the standard notation
\begin{equation} \lb{A.37}
\Im (M) = (M-M^*)/(2i),\quad \Re (M) = (M+M^*)/2, \quad M \in \cB(\cH).
\end{equation}

Note that $M$ is a bounded Herglotz operator if and only if the scalar-valued functions
$(u,Mu)_\cH$ are Herglotz for all $u\in\cH$.

As in the scalar case one usually extends $M$ to $\bbC_-$ by
reflection, that is, by defining
\begin{equation}
M(z)=M(\overline z)^*, \quad z\in \bbC_-.   \lb{A.36}
\end{equation}
Hence $M$ is analytic on $\bbC\backslash\bbR$, but $M\big|_{\bbC_-}$
and $M\big|_{\bbC_+}$, in general, are not analytic
continuations of each other.

Of course, one can also consider unbounded operator-valued
Herglotz functions, but they will not be used in this paper.

In contrast to the scalar case, one cannot generally expect strict
inequality in $\Im(M(\cdot))\geq 0$. However, the kernel of $\Im(M(\cdot))$
has simple properties:

\begin{lemma} \lb{lA.5}
Let $M(\cdot)$ be a bounded operator-valued Herglotz function in $\cH$.
Then the kernel $\cH_0 = \ker(\Im(M(z)))$ is independent of $z\in\bbC_+$.
Consequently, upon decomposing $\cH = \cH_0 \oplus \cH_1$,
$\cH_1 = \cH_0^\bot$, $\Im(M(\cdot))$ takes on the form
\begin{equation}
\Im(M(z))= \begin{pmatrix} 0 & 0 \\ 0 & N_1(z) \end{pmatrix},
\quad z \in \bbC_+,     \lb{A.38}
\end{equation}
where $N_1(\cdot) \in \cB(\cH_1)$ satisfies
\begin{equation}
N_1(z) > 0, \quad z\in\bbC_+.    \lb{A.39}
\end{equation}
\end{lemma}
For a proof of Lemma \ref{lA.5} see, for instance,
\cite[Proposition\ 1.2\,$(ii)$]{DM97} (alternatively, the
proof of \cite[Lemma\ 5.3]{GT00} in the matrix-valued context extends to the present
infinite-dimensional situation).

Next we recall the definition of a bounded operator-valued measure (see, also
\cite[p.\ 319]{Be68}, \cite{MM04}, \cite{PR67}):

\begin{definition} \lb{dA.6}
Let $\cH$ be a separable, complex Hilbert space.
A map $\Sigma:\mathfrak{B}(\bbR) \to\cB(\cH)$, with $\mathfrak{B}(\bbR)$ the
Borel $\sigma$-algebra on $\bbR$, is called a {\it bounded, nonnegative,
operator-valued measure} if the following conditions $(i)$ and $(ii)$ hold: \\
$(i)$ $\Sigma (\emptyset) =0$ and $0 \leq \Sigma(B) \in \cB(\cH)$ for all
$B \in \mathfrak{B}(\bbR)$. \\
$(ii)$ $\Sigma(\cdot)$ is strongly countably additive (i.e., with respect to the
strong operator  \hspace*{5mm} topology in $\cH$), that is,
\begin{align}
& \Sigma(B) = \slim_{N\to \infty} \sum_{j=1}^N \Sigma(B_j)   \lb{A.40} \\
& \quad \text{whenever } \, B=\bigcup_{j\in\bbN} B_j, \, \text{ with } \,
B_k\cap B_{\ell} = \emptyset \, \text{ for } \, k \neq \ell, \;
B_k \in \mathfrak{B}(\bbR), \; k, \ell \in \bbN.    \no
\end{align}
$\Sigma(\cdot)$ is called an {\it $($operator-valued\,$)$ spectral
measure} (or an {\it orthogonal operator-valued measure}) if additionally the following
condition $(iii)$ holds: \\
$(iii)$ $\Sigma(\cdot)$ is projection-valued (i.e., $\Sigma(B)^2 = \Sigma(B)$,
$B \in \mathfrak{B}(\bbR)$) and $\Sigma(\bbR) = I_{\cH}$. \\
$(iv)$ Let $f \in \cH$ and $B \in \mathfrak{B}(\bbR)$. Then the vector-valued
measure $\Sigma(\cdot) f$ has {\it finite variation on $B$}, denoted by
$V(\Sigma f;B)$, if
\begin{equation}
V(\Sigma f; B) = \sup\bigg\{\sum_{j=1}^N \|\Sigma(B_j)f\|_{\cH} \bigg\} < \infty,
\end{equation}
where the supremum is taken over all finite sequences $\{B_j\}_{1\leq j \leq N}$
of pairwise disjoint subsets on $\bbR$ with $B_j \subseteq B$, $1 \leq j \leq N$.
In particular, $\Sigma(\cdot) f$ has {\it finite total variation} if
$V(\Sigma f;\bbR) < \infty$.
\end{definition}

We recall that due to monotonicity considerations, taking the limit
in the strong operator topology in \eqref{A.40} is equivalent to taking the limit with
respect to the weak operator topology in $\cH$.

We also note that integrals of the type \eqref{A.41a}--\eqref{A.42a} below are now
taken with respect to an operator-valued measure, as opposed to the Bochner integrals
we used in the bulk of this paper, Sections \ref{s2}--\ref{s5}.

For relevant material in connection with the following result we refer the reader, for instance, to \cite{AL95}, \cite{AN75}, \cite{AN76},
\cite{BT92}, \cite[Sect.\ VI.5,]{Be68}, \cite[Sect.\ I.4]{Br71}, \cite{Bu97}, \cite{Ca76},
\cite{De62}, \cite{DM91}--\cite{DM97}, \cite[Sects.\ XIII.5--XIII.7]{DS88},
\cite{HS98}, \cite{KO77}, \cite{KO78}, \cite{Ma62}, \cite{MM02}, \cite{MM04},
\cite[Ch.\ VI]{Na68}, \cite{Na74}, \cite{Na77}, \cite{NA75}, \cite{Sh71}, \cite{Ts92},
\cite[Sects.\ 8--10]{We87}.

\begin{theorem}
\rm {(\cite{AN76}, \cite[Sect.\ I.4]{Br71}, \cite{Sh71}.)} \lb{tA.7}
Let $M$ be a bounded operator-valued Herglotz function in $\cH$.
Then the following assertions hold: \\
$(i)$ For each $f \in \cH$, $(f,M(\cdot) f)_{\cH}$ is a $($scalar$)$ Herglotz function. \\
$(ii)$ Suppose that $\{e_j\}_{j\in\bbN}$ is a complete orthonormal system in $\cH$
and that for some subset of $\bbR$ having positive Lebesgue measure, and for all
$j\in\bbN$, $(e_j,M(\cdot) e_j)_{\cH}$ has zero normal limits. Then $M\equiv 0$. \\
$(iii)$ There exists a bounded, nonnegative $\cB(\cH)$-valued measure
$\Omega$ on $\bbR$ such that the Nevanlinna representation
\begin{align}
& M(z) = C + D z + \int_{\bbR}
\f{d\Omega (\lambda)}{1+\lambda^2} \, \frac{1 + \lambda z}{\lambda -z}    \lb{A.41a} \\
& \qquad \; = C + D z + \int_{\bbR} d\Omega (\lambda ) \bigg[\f{1}{\lambda-z}
- \f{\lambda}{1+\lambda ^2}\bigg],
\quad z\in\bbC_+,    \lb{A.42} \\
& \wti \Omega((-\infty, \lambda]) = \slim_{\varepsilon \downarrow 0}
\int_{-\infty}^{\lambda + \varepsilon} \f{d \Omega (t)}{1 + t^2},  \quad
\lambda \in \bbR,  \\
& \wti \Omega(\bbR) = \Im(M(i))
= \int_{\bbR} \f{d\Omega(\lambda)}{1+\lambda^2} \in \cB(\cH),   \lb{A.42a} \\
& C=\Re(M(i)),\quad D=\slim_{\eta\uparrow \infty} \,
\frac{1}{i\eta}M(i\eta) \geq 0,      \lb{A.42b}
\end{align}
holds in the strong sense in $\cH$. Here
$\wti\Omega (B) = \int_{B}  \big(1+\lambda^2\big)^{-1}d\Omega(\lambda)$, $B \in \mathfrak{B}(\bbR)$.   \\
$(iv)$ Let $\lambda _1,\lambda_2\in\bbR$, $\lambda_1<\lambda_2$. Then the
Stieltjes inversion formula for $\Omega $ reads
\begin{equation}\lb{A.43}
\Omega ((\lambda_1,\lambda_2]) f =\pi^{-1} \slim_{\delta\downarrow 0}
\slim_{\varepsilon\downarrow 0}
\int^{\lambda_2 + \delta}_{\lambda_1 + \delta}d\lambda \,
\Im (M(\lambda+i\varepsilon)) f, \quad f \in \cH.
\end{equation}
$(v)$ Any isolated poles of $M$ are simple and located on the
real axis, the residues at poles being nonpositive bounded operators in $\cB(\cH)$.  \\
$(vi)$ For all $\lambda \in \bbR$,
\begin{align}
& \slim_{\varepsilon \downarrow 0} \, \varepsilon
\Re(M(\lambda +i\varepsilon ))=0,    \lb{A.45} \\
& \, \Omega (\{\lambda \}) = \slim_{\varepsilon \downarrow 0} \,
\varepsilon \Im (M(\lambda + i \varepsilon ))
=- i \slim_{\varepsilon \downarrow 0} \,
\varepsilon M(\lambda +i\varepsilon).     \lb{A.46}
\end{align}
$(vii)$ If in addition $M(z) \in \cB_{\infty} (\cH)$, $z \in \bbC_+$, then the measure
$\Omega$ in \eqref{A.41a} is countably additive with respect to the $\cB(\cH)$-norm,
and the Nevanlinna representation \eqref{A.41a}, \eqref{A.42} and the
Stieltjes inversion formula \eqref{A.43} as well as \eqref{A.45}, \eqref{A.46} hold
with the limits taken with respect to the $\|\cdot\|_{\cB(\cH)}$-norm. \\
$(viii)$ Let $f \in \cH$ and assume in addition that $\Omega(\cdot) f$ is of finite total
variation. Then for a.e.\ $\lambda \in \bbR$, the normal limits $M(\lambda + i0) f$
exist in the strong sense and
\begin{equation}
\slim_{\varepsilon \downarrow 0} M(\lambda +i\varepsilon) f
= M(\lambda +i 0) f = H(\Omega(\cdot) f) (\lambda) + i \pi \Omega'(\lambda) f,
\end{equation}
where $H(\Omega(\cdot) f)$ denotes the $\cH$-valued Hilbert transform
\begin{equation}
H(\Omega(\cdot) f) (\lambda) = \text{p.v.}\int_{- \infty}^{\infty} d \Omega (t) f \,
\f{1}{t - \lambda}
= \slim_{\delta \downarrow 0} \int_{|t-\lambda|\geq \delta} d \Omega (t) f \,
\f{1}{t - \lambda}.
\end{equation}
\end{theorem}

As usual, the normal limits in Theorem \ref{tA.7} can be replaced by nontangential
ones.

The nature of the boundary values of $M(\cdot + i 0)$ when for some $p>0$,
$M(z) \in \cB_p(\cH)$, $z \in \bbC_+$, was clarified in detail in \cite{BE67},
\cite{Na89}, \cite{Na90}, \cite{Na94}.

Using an approach based on operator-valued Stieltjes integrals, a special case
of Theorem \ref{tA.7} was proved by Brodskii \cite[Sect.\ I.4]{Br71}. In particular,
he proved the analog of the Herglotz representation for operator-valued
Caratheodory functions. More precisely, if $F$ is analytic on $\D$ (the open
unit disk in $\bbC$) with nonnegative real part $\Re(F(w)) \geq 0$, $w \in \D$,
then $F$ is of the form
\begin{align}
\begin{split}
& F (w) = i \Im(F(0)) + \oint_{\dD} d\Upsilon (\zeta) \, \f{\zeta + w}{\zeta - w},
\quad w \in \D,      \lb{A.54} \\
& \Re (F (0)) = \Upsilon(\dD),
\end{split}
\end{align}
with $\Upsilon$ a bounded, nonnegative $\cB(\cH)$-valued measure on $\dD$. The result \eqref{A.54} can also be derived by an application of Naimark's dilation
theory (cf.\ \cite{AN76} and \cite[p.\ 68]{Fi70}), and it can also be used to derive the Nevanlinna representation \eqref{A.41a}, \eqref{A.42} (cf.\ \cite{AN76}, and in a
special case also \cite[Sect.\ I.4]{Br71}). Finally, we also mention that
Shmuly'an \cite{Sh71} discusses the Nevanlinna representation \eqref{A.41a},
\eqref{A.42}; moreover, certain special classes of Nevanlinna functions, isolated
by Kac and Krein \cite{KK74} in the scalar context, are studied by Brodskii \cite[Sect.\ I.4]{Br71} and Shmuly'an \cite{Sh71}.

For a variety of applications of operator-valued Herglotz functions, see, for
instance, \cite{AL95}, \cite{ABMN05}, \cite{ABT11}, \cite{BMN02}, \cite{Ca76},
\cite{DM91}--\cite{DM97}, \cite{GKMT01},
\cite{MM02}--\cite{MN11a}, \cite{Sh71}, and the literature cited therein.

\section{Direct Integrals and the Construction of the Model Hilbert Space
$L^2(\bbR;d\Sigma;\cK)$} \lb{sD}
\setcounter{theorem}{0}
\setcounter{equation}{0}

In this appendix  we recall the construction of a model Hilbert space $L^2(\bbR; d\Sigma; \cK)$
(and related Banach spaces $L^p(\bbR; w d\Sigma; \cK)$, $p \geq 1$, $w$ an appropriate scalar
nonnegative weight function) as discussed in detail in \cite{GKMT01} and slightly extended in
\cite{GWZ13a}. Variants of this construction are of importance in the bulk of this paper.

For proofs of the results in this appendix we refer to \cite{GKMT01} and \cite{GWZ13a};
as general background literature for the topic to follow, we refer to the theory of
direct integrals of Hilbert spaces as presented, for instance, in \cite[Ch.\ 4]{BW83},
\cite[Ch.\ 7]{BS87}, \cite[Ch.\ II]{Di96}, \cite[Ch.\ XII]{vN51}.

Throughout this section we make the following assumptions:

\begin{hypothesis} \lb{hD.0}
Let $\mu$ denote a $\sigma$-finite Borel measure on $\bbR$,
$\mathfrak{B}(\bbR)$ the Borel $\sigma$-algebra on $\bbR$,
and suppose that $\cK$ and $\cK_{\lambda}$, $\lambda\in\bbR$, denote
separable, complex Hilbert spaces such that the dimension function
$\bbR \ni \lambda \mapsto \dim (\cK_{\lambda}) \in \bbN \cup \{\infty\}$
is $\mu$-measurable.
\end{hypothesis}

Assuming Hypothesis \ref{hD.0}, let
$\cS(\{\cK_\lambda\}_{\lambda\in\bbR})$
be the vector space associated with the Cartesian product
$\prod_{\lambda\in\bbR}\cK_\lambda$ equipped with
the obvious linear structure. Elements of
$\cS(\{\cK_\lambda\}_{\lambda\in\bbR})$ are maps
\begin{equation}
f \in \cS(\{\cK_\lambda\}_{\lambda\in\bbR}), \quad
\bbR\ni\lambda\mapsto f(\lambda)\in\cK_{\lambda},
\lb{D.1}
\end{equation}
in particular, we identify $f = \{f(\lambda)\}_{\lambda \in \bbR}$.

\begin{definition} \lb{dD.1}
Assume Hypothesis \ref{hD.0}.
A {\it measurable family of Hilbert spaces $\cM$
modeled on $\mu$ and
$\{\cK_\lambda\}_{\lambda\in\bbR}$} is a linear
subspace $\cM\subset\cS(\{\cK_\lambda\}
_{\lambda\in\bbR})$ such that $f\in\cM$ if and
only if the map $\bbR\ni\lambda\mapsto (f(\lambda),
g(\lambda))_{\cK_\lambda}\in\bbC$ is
$\mu$-measurable for all $g\in\cM$. 
Moreover, $\cM$ is said to be {\it generated by some
subset $\cF$, $\cF\subset\cM$}, if for every
$g\in\cM$ we can
find a sequence of functions $h_n\in{\rm lin.span}
\{\chi_Bf\in \cS(\{\cK_\lambda\}) \,|\, B\in\mathfrak{B}(\bbR), f\in\cF\}$ with
$\lim_{n\to\infty}\|g(\lambda)-h_n(\lambda)
\|_{\cK_{\lambda}}=0$
$\mu$-a.e.
\end{definition}

We note that we shall identify functions in $\cM$ which coincide $\mu$-a.e.; thus 
$\cM$ is more precisely a set of equivalence classes of functions. 
The definition of $\cM$ was chosen with its maximality
in mind and we refer to Lemma \ref{lD.3} and for more details
in this respect. An explicit construction of an example of
$\cM$ will be given in Theorem \ref{tD.7}.

\begin{remark} \lb{rD.2}
The following properties are proved in a standard manner:\\
$(i)$ If $f\in\cM$, $g\in\cS(\{\cK_\lambda\}
_{\lambda\in\bbR})$ and $g=f$ $\mu$-a.e.\ then $g\in\cM$.\\
$(ii)$ If $\{f_n\}_{n\in\bbN}\in\cM$,
$g\in\cS(\{\cK_\lambda\}_{\lambda\in\bbR})$
and $f_n(\lambda)\to g(\lambda)$ as $n\to\infty$
$\mu$-a.e.\ (i.e., $\lim_{n\to\infty}\|f_n(\lambda)-
g(\lambda)\|
_{\cK_\lambda}=0$ $\mu$-a.e.) then $g\in\cM$.\\
$(iii)$ If $\phi$ is a scalar-valued $\mu$-measurable
function and $f\in\cM$ then $\phi f\in\cM$.\\
$(iv)$ If $f\in\cM$ then $\bbR\ni\lambda\mapsto
\|f(\lambda)\|_{\cK_\lambda}\in [0,\infty)$ is
$\mu$-measurable.
\end{remark}

\begin{lemma} [\cite{GKMT01}] \lb{lD.3}
Assume Hypothesis \ref{hD.0}. Suppose that
$\{f_n\}_{n\in\bbN} \subset \cS(\{\cK_\lambda\}_{\lambda\in\bbR})$
is such that\\
$($$\alpha$$)$ $\bbR\ni\lambda\mapsto
(f_m(\lambda),f_n(\lambda))_{\cK_\lambda}\in\bbC$ is
$\mu$-measurable for
all $m,n\in\bbN$.\\
$($$\beta$$)$ For $\mu$-a.e. $\lambda\in\bbR$,
$\ol{{\rm lin.span}\{f_n(\lambda)\}_{n\in\bbN}} = \cK_\lambda$. \\
\noindent In particular, any orthonormal basis $\{e_n (\lambda)\}_{n\in\bbN}$
in $\cK_{\lambda}$ will satisfy $(\alpha)$ and $(\beta)$. Setting
\begin{equation}
\cM=\{g\in\cS(\{\cK_\lambda\}_{\lambda\in\bbR})
 \, |\, (f_n(\lambda),g(\lambda))_{\cK_\lambda}
\text{ is }
\mu\text{-measurable for all }n\in\bbN\}, \lb{D.3}
\end{equation}
one has the following facts: \\
$(i)$ $\cM$ is a measurable family of Hilbert spaces.\\
$(ii)$ $\cM$ is generated by $\{f_n\}_{n\in\bbN}$. \\
$(iii)$ $\cM$ is the unique measurable family of
Hilbert spaces
containing the sequence $\{f_n\}_{n\in\bbN}$. \\
$(iv)$ If $\{g_n\}_{n\in\bbN} \subset \cM$ is any sequence satisfying
$(\beta)$ then
$\cM$ is generated by $\{g_n\}_{n\in\bbN}$.
\end{lemma}

Next, let $w$ be a $\mu$-measurable function,
$w>0$ $\mu$-a.e., and consider the space
\begin{equation}
\dot L^2(\bbR; w d\mu; \cM)= \bigg\{f\in\cM \, \bigg| \,
\int_\bbR w(\lambda)
d\mu(\lambda) \, \|f(\lambda)\|^2_{\cK_\lambda}
<\infty\bigg\}
\lb{D.6}
\end{equation}
with its obvious linear structure. On
$\dot L^2(\bbR; w d\mu; \cM)$ one defines a semi-inner
product
$(\cdot,\cdot)_{\dot L^2(\bbR; w d\mu; \cM)}$ (and hence a
semi-norm $\|\cdot\|_{\dot L^2(\bbR; w d\mu; \cM)}$) by
\begin{equation}
(f,g)_{\dot L^2(\bbR; w d\mu; \cM)}=\int_\bbR
w(\lambda) d\mu(\lambda) \,
(f(\lambda),g(\lambda))_{\cK_\lambda}, \quad
f,g\in \dot L^2(\bbR; w d\mu; \cM). \lb{D.7}
\end{equation}
That \eqref{D.7} defines a semi-inner product
immediately follows
from the corresponding properties of
$(\cdot,\cdot)_{\cK_\lambda}$ and the linearity of
the integral. Next, one defines the equivalence relation
$\sim$, for elements $f,g\in \dot L^2(\bbR; w d\mu; \cM)$ by
\begin{equation}
f\sim g \text{ if and only if } f=g \quad \text{$\mu$-a.e.}
\lb{D.7a}
\end{equation}
and hence introduces the set of equivalence classes of
$\dot L^2(\bbR; w d\mu; \cM)$ denoted by
\begin{equation}
L^2(\bbR; w d\mu; \cM)=\dot L^2(\bbR; w d\mu; \cM)/\sim .
\lb{D.7b}
\end{equation}
In particular, introducing the subspace of null functions
\begin{align}
\cN(\bbR; w d\mu; \cM)&= \big\{f\in \dot L^2(\bbR; w d\mu; \cM)
\,\big|\, \|f(\lambda)\|_{\cK_\lambda}=0 \text{ for }
\text{$\mu$-a.e. }\lambda\in\bbR\big\} \no \\
&= \big\{f\in \dot L^2(\bbR; w d\mu; \cM) \,\big|\,
\|f\|_{\dot L^2(\bbR; w d\mu; \cM)}=0\big\}, \lb{D.7c}
\end{align}
$L^2(\bbR; w d\mu; \cM)$ is precisely the quotient space
$\dot L^2(\bbR; w d\mu; \cM)/\cN(\bbR; w d\mu; \cM)$.
Denoting the equivalence class of
$f\in \dot L^2(\bbR; w d\mu; \cM)$ temporarily by $[f]$, the
semi-inner product on $L^2(\bbR; w d\mu; \cM)$
\begin{equation}
([f],[g])_{L^2(\bbR; w d\mu; \cM)}=\int_\bbR
w(\lambda)d\mu(\lambda) \,
(f(\lambda),g(\lambda))_{\cK_\lambda} \lb{D.7d}
\end{equation}
is well-defined (i.e., independent of the chosen representatives of the equivalence classes) and actually an inner product. Thus, $L^2(\bbR; w d\mu; \cM)$ is a normed space and by the usual abuse of notation we denote its elements in the
following again by $f,g$, etc.\ Moreover, $L^2(\bbR; w d\mu; \cM)$ is also complete:

\begin{theorem} \lb{tD.4}
Assume Hypothesis \ref{hD.0}. Then the normed space
$L^2(\bbR; w d\mu; \cM)$ is complete and hence a Hilbert space. In addition,
$L^2(\bbR; w d\mu; \cM)$ is separable.
\end{theorem}
That $L^2(\bbR; w d\mu; \cM)$ is complete was shown in
\cite[Subsect.\ 4.1.2]{BW83}, \cite[Sect.\ 7.1]{BS87}, and more recently, in
\cite{GKMT01}. Separability of $L^2(\bbR; w d\mu; \cM)$ is proved in
\cite[Sect.\ 7.1]{BS87} (see also \cite[Subsect.\ 4.3.2]{BW83}).

\begin{remark} \lb{rD.4}
A similar construction defines the
Banach spaces $L^p(\bbR; w d\mu; \cM)$, $p\geq 1$.
\end{remark}

Thus, $L^2(\bbR; w d\mu; \cM)$ corresponds precisely to the direct
integral of the Hilbert spaces $\cK_\lambda$ with respect to
the measure $wd\mu$ (see, e.g., \cite[Ch.\ 4]{BW83},
\cite[Ch.\ 7]{BS87}, \cite[Ch.\ II]{Di96}, \cite[Ch.\ XII]{vN51}) and is
frequently denoted by
$\int_{\bbR}^{\oplus} w(\lambda) d\mu(\lambda) \, \cK_{\lambda}$.

\smallskip

Having reviewed the construction of $L^2(\bbR; w d\mu; \cM)=
\int_{\bbR}^{\oplus} w(\lambda) d\mu(\lambda) \, \cK_{\lambda}$ in
connection with a scalar measure $w d\mu$, we now turn to the case
of operator-valued measures and recall the following definition (we refer, for
instance, to \cite[Sects.\ 1.2, 3.1, 5.1]{BW83}, \cite[Sect.\ VII.2.3]{Be68},
\cite[Ch.\ 6]{BS87}, \cite[Ch.\ I]{DU77}, \cite[Ch.\ X]{DS88}, \cite{MM04} for vector-valued
and operator-valued measures):

\begin{definition} \lb{dA.6a}
Let $\cH$ be a separable, complex Hilbert space.
A map $\Sigma:\mathfrak{B}(\bbR) \to\cB(\cH)$, with $\mathfrak{B}(\bbR)$ the
Borel $\sigma$-algebra on $\bbR$, is called a {\it bounded, nonnegative,
operator-valued measure} if the following conditions $(i)$ and $(ii)$ hold: \\
$(i)$ $\Sigma (\emptyset) =0$ and $0 \leq \Sigma(B) \in \cB(\cH)$ for all
$B \in \mathfrak{B}(\bbR)$. \\
$(ii)$ $\Sigma(\cdot)$ is strongly countably additive (i.e., with respect to the
strong operator  \hspace*{5mm} topology in $\cH$), that is,
\begin{align}
& \Sigma(B) = \slim_{N\to \infty} \sum_{j=1}^N \Sigma(B_j)   \lb{A.40a} \\
& \quad \text{whenever } \, B=\bigcup_{j\in\bbN} B_j, \, \text{ with } \,
B_k\cap B_{\ell} = \emptyset \, \text{ for } \, k \neq \ell, \;
B_k \in \mathfrak{B}(\bbR), \; k, \ell \in \bbN.    \no
\end{align}
Moreover, $\Sigma(\cdot)$ is called an {\it $($operator-valued\,$)$ spectral
measure} (or an {\it orthogonal operator-valued measure}) if additionally
the following condition $(iii)$ holds: \\
$(iii)$ $\Sigma(\cdot)$ is projection-valued (i.e., $\Sigma(B)^2 = \Sigma(B)$,
$B \in \mathfrak{B}(\bbR)$) and $\Sigma(\bbR) = I_{\cH}$.
\end{definition}

In the following, let $\Sigma:\mathfrak{B}(\bbR) \to\cB(\cK)$ be a bounded
nonnegative measure, that is, $\Sigma$ satisfies requirements $(i)$ and $(ii)$ in
Definition \ref{dA.6a}. Denoting $T = \Sigma(\bbR)$, one has
\begin{equation}
0 \leq \Sigma(B) \leq T \in \cB(\cK), \quad B \in \mathfrak{B} (\bbR),    \lb{D.8}
\end{equation}
and hence
\begin{equation}
\big\|\Sigma(B)^{1/2} \xi \big\|_{\cK} \leq \big\|T^{1/2} \xi \big\|_{\cK}, \quad \xi \in \cK,    \lb{D.8A}
\end{equation}
shows that
\begin{equation}
\ker(T) = \ker\big(T^{1/2}\big) \subseteq \ker\big(\Sigma(B)^{1/2}\big) = \ker(\Sigma(B)), \quad
B \in \mathfrak{B} (\bbR).    \lb{D.8B}
\end{equation}
We will use the orthogonal decomposition
\begin{equation}
\cK = \cK_0 \oplus \cK_1, \quad \cK_0 = \ker(T), \; \,\cK_1 = \ker(T)^\bot = \ol{\ran(T)},    \lb{D.8C}
\end{equation}
and identify $f_0 = (f_0 \;\; 0)^\top \in \cK_0$ and $f_1 = (0 \; \; f_1)^\top \in \cK_1$. In particular, with
$f = (f_0 \; \; f_1)^\top$, one has $\|f\|_{\cK}^2 = \|f_0\|_{\cK_0}^2 + \|f_1\|_{\cK_1}^2$.
Then $T$ permits the $2\times 2$ block operator representation
\begin{equation}
T = \begin{pmatrix} 0 & 0 \\ 0 & T_1 \end{pmatrix}, \, \text{ with } \,
0 \leq T_1 \in \cB(\cK_1), \;\, \ker(T_1) = \{0\},     \lb{D.8D}
\end{equation}
with respect to the decomposition \eqref{D.8C}. By \eqref{D.8B} one concludes that
$\Sigma(B)$, $B \in \mathfrak{B} (\bbR)$, is necessarily of the form
\begin{equation}
\Sigma(B) = \begin{pmatrix} 0 & D^* \\ D & \Sigma_1 (B) \end{pmatrix},
\, \text{ for some } \, 0 \leq \Sigma_1 (B) \in \cB(\cK_1), \; D \in \cB(\cK_0,\cK_1),
\lb{D.8F}
\end{equation}
with respect to the decomposition \eqref{D.8C}. The computation
\begin{equation}
0 = \Sigma(B) \begin{pmatrix} f_0 \\ 0 \end{pmatrix}
= \begin{pmatrix} 0 & D^* \\ D
& \Sigma_1 (B)\end{pmatrix} \begin{pmatrix} f_0 \\ 0 \end{pmatrix}
= \begin{pmatrix} 0 \\ D f_0 \end{pmatrix},
\quad f_0 \in \cK_0,    \lb{D.8G}
\end{equation}
yields $D=0$ as $f_0 \in \cK_0$ was arbitrary. Thus, $\Sigma(B)$,
$B \in \mathfrak{B} (\bbR)$, is actually also of diagonal form
\begin{equation}
\Sigma(B) = \begin{pmatrix} 0 & 0 \\ 0 & \Sigma_1 (B) \end{pmatrix},
\, \text{ for some } \, 0 \leq \Sigma_1 (B) \in \cB(\cK_1),    \lb{D.8H}
\end{equation}
with respect to the decomposition \eqref{D.8C}.

Moreover, let
$\mu$ be a control measure for $\Sigma$ (equivalently, for $\Sigma_1$), that is,
\begin{equation}
\mu(B)=0 \text{ if and only if } \Sigma(B)=0
\text{ for all } B\in\mathfrak{B}(\bbR). \lb{D.8a}
\end{equation}
(E.g., $\mu(B)=\sum_{n\in\cI}2^{-n}(e_n, \Sigma(B)e_n)_\cK$, $B\in\mathfrak{B}(\bbR)$,
with $\{e_n\}_{n\in\cI}$ a complete orthonormal system in $\cK$, $\cI\subseteq\bbN$,
an appropriate index set.)

The following theorem was first stated in \cite{GKMT01} under the implicit assumption
that $\Sigma (\bbR) = T = I_{\cK}$. The general case
$T \in \cB(\cK)$, explicitly permitting the existence of a nontrivial kernel of $T$ was recently 
discussed in \cite{GWZ13a}:

\begin{theorem} \lb{tD.7}
Let $\cK$ be a separable, complex Hilbert space,
$\Sigma:\mathfrak{B}(\bbR) \to\cB(\cK)$ a bounded, nonnegative
operator-valued measure, and $\mu$ a control measure for $\Sigma$.
Then there are separable, complex Hilbert spaces
$\cK_\lambda$,
$\lambda\in\bbR$, a measurable family of Hilbert spaces
${\cM}_\Sigma$
modelled on $\mu$ and $\{\cK_\lambda\}
_{\lambda\in\bbR}$,
and a bounded linear map $\ul \Lambda\in\cB\big(\cK,
L^2(\bbR; d\mu; {\cM}_\Sigma)\big)$, satisfying
\begin{equation}
\|\ul \Lambda\|_{\cB(\cK,L^2(\bbR; d\mu; {\cM}_\Sigma))}
= \big\|T^{1/2}\big\|_{\cB(\cK)},      \lb{D.8b}
\end{equation}
and
\begin{equation}
\ker(\ul \Lambda) = \ker(T),    \lb{D.8c}
\end{equation}
so that the following assertions $(i)$--$(iii)$ hold: \\
$(i)$ For all $B \in \mathfrak{B} (\bbR)$, $\xi,\eta\in\cK$,
\begin{equation}
(\eta,\Sigma(B) \xi)_\cK=\int_B d\mu(\lambda) \,
((\ul \Lambda \eta)(\lambda),(\ul \Lambda \xi)(\lambda))
_{\cK_\lambda},     \lb{D.9}
\end{equation}
in particular,
\begin{equation}
(\eta, T \xi)_\cK=\int_{\bbR} d\mu(\lambda) \,
((\ul \Lambda \eta)(\lambda),(\ul \Lambda \xi)(\lambda))
_{\cK_\lambda}. \lb{D.9a}
\end{equation}
$(ii)$ Let $\cI = \{1,\dots,N\}$ for some $N\in\bbN$, or $\cI = \bbN$.
$\ul \Lambda(\{e_n\}_{n\in\cI})$ generates ${\cM}_\Sigma$, where
$\{e_n\}_{n\in\cI}$ denotes any
sequence of linearly independent elements in $\cK$
with the property $\ol{{\rm lin.span} \{e_n\}_{n\in\cI}}=\cK$.
In particular, $\ul \Lambda(\cK)$ generates
${\cM}_\Sigma$.\\
$(iii)$
For all $B \in \mathfrak{B} (\bbR)$ and $\xi\in\cK$,
\begin{equation}
\ul \Lambda(S(B)\xi\big)
= \{\chi_B(\lambda) (\ul \Lambda \xi)(\lambda)\}_{\lambda\in\bbR},      \lb{D.10}
\end{equation}
where $($cf.\ \eqref{D.8D} and \eqref{D.8H}$)$
\begin{equation}
S(B) = \begin{pmatrix} I_{\cK_0} & 0 \\ 0 & T_1^{-1/2} \Sigma_1 (B)^{1/2} \end{pmatrix},
\quad S(\bbR) = I_{\cK},
\end{equation}
with respect to the decomposition \eqref{D.8C}.
\end{theorem}

Next, we recall that the construction in Theorem \ref{tD.7} is essentially unique:

\begin{theorem} [\cite{GKMT01}] \lb{tD.9}
Suppose $\cK'_{\lambda},\ \lambda\in\bbR$ is a
family of separable
complex Hilbert spaces, $\cM'$ is a measurable
family of Hilbert spaces
modelled on $\mu$ and $\{\cK'_{\lambda}\}$, and
$\ul \Lambda'\in\cB\big(\cK,L^2(\bbR; d\mu; \cM')\big)$ is
a map satisfying $(i)$, $(ii)$, and $(iii)$ of
Theorem \ref{tD.7}.
Then for $\mu$-a.e. $\lambda\in\bbR$ there is a
unitary operator
$U_{\lambda}:\cK_{\lambda}\to\cK'_{\lambda}$
such that $f=\{f(\lambda)\}_{\lambda\in\bbR}
\in {\cM}_{\Sigma}$ if and only if
$\{U_{\lambda}f(\lambda)\}_{\lambda \in \bbR}\in\cM'$ and for all
$\xi\in\cK$,
\begin{equation}
(\ul \Lambda' \xi)(\lambda)= U_{\lambda}
(\ul \Lambda \xi)(\lambda) \quad \text{$\mu$-a.e.}
\lb{D.17g}
\end{equation}
\end{theorem}

\begin{remark} \lb{rD.11}
$(i)$ Without going into further details, we note that
${\cM}_{\Sigma}$ depends of course on the control measure $\mu$.
However, a change in $\mu$ merely effects a change in
density and so ${\cM}_{\Sigma}$ can essentially be
viewed as $\mu$-independent. \\
$(ii)$ With $0 < w$ a $\mu$-measurable weight function, one can also consider the
Hilbert space $L^2(\bbR; w d\mu; {\cM}_\Sigma)$.
In view of our comment in item $(i)$ concerning the mild
dependence on the control measure $\mu$ of
${\cM}_{\Sigma}$, one typically puts more emphasis on the
operator-valued measure $\Sigma$ and hence uses the more suggestive
notation $L^2(\bbR; w d\Sigma; \cK)$ instead of the more
precise $L^2(\bbR; w d\mu; {\cM}_\Sigma)$ in this case.
\end{remark}

Next, let
\begin{equation}
\cV={\rm lin.span}\{e_n \in \cK \,|\, n\in\cI\}, \quad  \ol \cV = \cK,
\lb{D.17h}
\end{equation}
and define
\begin{equation}
{\ul \cV}_\Sigma={\rm lin.span}\big\{\chi_B(\cdot) \, \ul \Lambda e_n\in
L^2(\bbR; d\mu; {\cM}_{\Sigma}) \, \big| \, B \in \mathfrak{B}(\bbR),
\,n\in\cI\big\}. \lb{D.18}
\end{equation}
The fact that $\{\ul \Lambda e_n\}_{n\in\cI}$ generates
${\cM}_{\Sigma}$ then implies that ${\ul \cV}_\Sigma$
is dense in the Hilbert space $L^2(\bbR; d\mu; {\cM}_\Sigma)$, that is,
\begin{equation}
\ol{{\ul \cV}_{\Sigma}}=L^2(\bbR; d\mu; {\cM}_\Sigma). \lb{D.19}
\end{equation}

Since the operator-valued distribution function $\Sigma(\cdot)$ has at most
countably many discontinuities on $\bbR$, denoting by $\mathfrak{S}_{\Sigma}$ the
corresponding set of discontinuities of $\Sigma(\cdot)$, introducing the set of intervals
\begin{equation}
\cB_{\Sigma} = \{(\alpha, \beta] \subset \bbR\,|\, \alpha, \beta \in \bbR\backslash\mathfrak{S}_{\Sigma}\},
\end{equation}
the minimal $\sigma$-algebra generated by $\cB_{\Sigma}$ coincides with the Borel algebra
$\mathfrak{B}(\bbR)$. Hence one can introduce
\begin{equation}
{\wti {\ul \cV}}_\Sigma = {\rm lin.span}\big\{\chi_{(\alpha,\beta]}(\cdot)
 \, \ul \Lambda e_n\in
L^2(\bbR; d\mu; {\cM}_{\Sigma}) \, \big| \,
\alpha, \beta \in \bbR\backslash\mathfrak{S}_{\Sigma}, \,n\in\cI\big\},    \lb{D.19a}
\end{equation}
which still retains the density property in \eqref{D.19}, that is,
\begin{equation}
\ol{{\wti{\ul \cV}}_{\Sigma}}=L^2(\bbR; d\mu; {\cM}_\Sigma). \lb{D.19b}
\end{equation}

In the following we briefly describe an alternative construction of $L^2(\bbR; d\Sigma; \cK)$ used by Berezanskii \cite[Sect.\ VII.2.3]{Be68} in order to identify the two constructions.

Introduce
\begin{align}
&C_{0,0}(\bbR;\cK) = \bigg\{u:\bbR\to\cK\,\bigg|\, u (\cdot) \text{ is strongly continuous in $\cK$}, 
\, \supp(u) \text{ is compact},   \no \\
& \hspace*{4.1cm}
\dim\bigg(\bigcup_{\lambda\in\bbR} \ran (u(\lambda))\bigg) < \infty\bigg\}. 
\end{align}
On $C_{0,0}(\bbR;\cK)$ one can introduce the semi-inner product
\begin{equation}
(u,v)_{L^2(\bbR; d\Sigma; \cK)}
= \int_{\bbR} (u(\lambda), d \Sigma(\lambda) v(\lambda))_{\cK},
\quad u, v \in C_{0,0}(\bbR;\cK),    \lb{D.35}
\end{equation}
where the integral on the right-hand side of \eqref{D.35} is well-defined in the
Riemann--Stieltjes sense.
Introducing the kernel of this semi-inner product by
\begin{equation}
\cN = \{u \in C_{0,0}(\bbR;\cK) \,|\, (u, u)_{L^2(\bbR; d\Sigma; \cK)} = 0\},
\lb{D.36}
\end{equation}
Berezanskii \cite[Sect.\ VII.2.3]{Be68} obtains the separable Hilbert space
$L^2(\bbR; d\Sigma; \cK)$ as the completion of $C_{0,0}(\bbR;\cK)/\cN$ with respect
to the inner product in \eqref{D.35} as
\begin{equation}
\hatt {L^2(\bbR; d\Sigma; \cK)} = \ol{C_{0,0}(\bbR;\cK)/\cN}.      \lb{D.37}
\end{equation}
In particular,
\begin{equation}
([u], [v])_{\hatt{L^2(\bbR; d\Sigma; \cK)}}
= \int_{\bbR} (u(\lambda), d \Sigma(\lambda) v(\lambda))_{\cK},
\quad u, v \in C_{0,0}(\bbR;\cK),      \lb{D.38}
\end{equation}
and (cf.\ also \cite[Corollary\ 2.6]{MM04}) \eqref{D.38} extends to piecewise continuous $\cK$-valued functions with compact support as long as the discontinuities of $\ul u$ and $\ul v$ are disjoint from the set $\mathfrak{S}_{\Sigma}$ (the set of discontinuities of $\Sigma(\cdot)$).

Since Kats' work in the case of a finite-dimensional Hilbert space $\cK$ (cf.\ \cite{Ka50}, \cite{Ka03} and also Fuhrman \cite[Sect.\ II.6]{Fu81} and Rosenberg \cite{Ro64}),
and especially in the work of Malamud and Malamud \cite{MM04}, who studied the general case
$\dim(\cK) \leq \infty$, it has become customary to interchange the order of taking the quotient with
respect to the semi-inner product and completion in this process of constructing
$\hatt{L^2(\bbR; d\Sigma; \cK)}$. More precisely, in this context one first completes $C_{0,0}(\bbR,\cK)$
with respect to the semi-inner product \eqref{D.35} to obtain a semi-Hilbert space
\begin{equation}
\wti {L^2(\bbR; d\Sigma; \cK)} = \ol{C_{0,0}(\bbR;\cK)},     \lb{D.39}
\end{equation}
and then takes the quotient with respect to the kernel of the underlying semi-inner product, as
described in method $\mathbf{(I)}$ of \cite[Appendix\ A]{GWZ13a}. Berezanskii's approach 
in \cite[Sect.\ VII.2.3]{Be68} corresponds to method $\mathbf{(II)}$ discussed in 
\cite[Appendix\ A]{GWZ13a}. The equivalence of these two methods is not stated in these sources, but was spelled out explicitly in \cite{GWZ13a}.

Next we will recall that Berezanskii's construction of $\hatt{L^2(\bbR; d\Sigma; \cK)}$
(and hence the corresponding construction by Kats (if $\dim(\cK)<\infty$) and by
Malamud and Malamud (if $\dim(\cK) \leq \infty$) is equivalent to the one in \cite{GKMT01}
and hence to that outlined in Theorem \ref{tD.7}. For this purpose we recall that 
it was shown in the proof of Theorem\ 2.14 in \cite{MM04} that
\begin{equation} 
{\hatt{\cV}}_\Sigma = {\rm lin.span}\Big\{\chi_{(\alpha,\beta]}(\cdot) \, e_n \in
\hatt{L^2(\bbR; d\Sigma; \cK)} \, \Big| \,
\alpha, \beta \in \bbR\backslash\mathfrak{S}_{\Sigma}, \,n\in\cI\Big\},    \lb{D.19A}
\end{equation}
is dense in $\hatt{L^2(\bbR; d\Sigma; \cK)}$. 

\begin{theorem} \lb{tD.14}
The Hilbert spaces $\hatt{L^2(\bbR; d\Sigma; \cK)}$ and $L^2(\bbR; d\mu; \cM_{\Sigma})$ are
isometrically isomorphic with isomorphism $\cU_{\Sigma}$ defined as follows:  
\begin{equation}
\dot \cU_{\Sigma}: \begin{cases} {\hatt{\cV}}_\Sigma \to {\wti {\ul \cV}}_\Sigma, \\
\chi_{(\alpha,\beta]}(\cdot) \, e_n \mapsto \chi_{(\alpha,\beta]}(\cdot) \, {\ul \Lambda} e_n,
\end{cases} \quad
\alpha, \beta \in \bbR\backslash \mathfrak{S}_{\Sigma}, \; n \in \cI,    \lb{U0}
\end{equation}
establishes the densely defined isometry between the Hilbert spaces
$\hatt{L^2(\bbR; d\Sigma; \cK)}$ and $L^2(\bbR; d\mu; \cM_{\Sigma})$  
which extends by continuity to the unitary map 
\begin{equation} 
\cU_{\Sigma} = \ol{\dot \cU_{\Sigma}}: \hatt{L^2(\bbR; d\Sigma; \cK)} 
\to L^2(\bbR; d\mu; \cM_{\Sigma}).   \lb{U}
\end{equation} 
\end{theorem}

As a result, dropping the additional ``hat'' on the left-hand side of \eqref{D.37}, and hence
just using the notation $L^2(\bbR; d\Sigma; \cK)$ for both Hilbert space constructions is consistent.

We continue this section by yet another approach originally due to Gel'fand and
Kostyuchenko \cite{GK55} and Berezanskii \cite[Ch.\ V]{Be68}. In this context
we also refer to Berezankii \cite[Sect.\ 2.2]{Be86},
Berezansky, Sheftel, and Us \cite[Ch.\ 15]{BSU96},
Birman and Entina \cite{BE67}, Gel'fand and Shilov \cite[Ch.\ IV]{GS67},
and M.\ Malamud and S.\ Malamud \cite{MM02}, \cite{MM04}: Introducing an
operator $K \in \cB_2(\cH)$ with $\ker(K) = \ker(K^*) = \{0\}$, one has the existence
of the weakly $\mu$-measurable nonnegative operator-valued function
$\Psi_K(\cdot)$ with values in $\cB_1(\cH)$, such that
\begin{align}
\begin{split}
(f, \Sigma (B) g)_{\cH} = \int_B d\mu(t) \,
\big(\Psi_K (t)^{1/2} K^{-1} f, \Psi_K (t)^{1/2} K^{-1} g\big)_{\cH},&   \\
f, g \in \dom\big(K^{-1}\big), \; B \in \mathfrak{B} (\bbR), \; \text{$B$ bounded,}&
\end{split}
\end{align}
with
\begin{equation}
\Psi_K (\cdot) = \f{d K^*\Sigma K}{d \mu}(\cdot) \, \text{ $\mu$-a.e.}
\end{equation}
In fact, the derivative $\Psi_K (\cdot)$ exists in the $\cB_1(\cH)$-norm (cf.\ \cite{BE67}
and \cite{MM02}, \cite{MM04}). Introducing the semi-Hilbert space
$\wti \cH_t$, $t\in\bbR$, as the completion of $\dom\big(K^{-1}\big)$ with respect to
the semi-inner product
\begin{equation}
(f,g)_{\wti \cH_t} = \big(\Psi_K (t)^{1/2} K^{-1} f, \Psi_K (t)^{1/2} K^{-1} g\big)_{\cH},
\quad f, g \in \dom\big(K^{-1}\big), \; t \in \bbR,
\end{equation}
factoring $\wti \cH_t$ by the kernel of the corresponding semi-norm
$\ker(\|\cdot\|_{\wti \cH_t})$ then yields the Hilbert space
$\cH_t = \wti \cH_t / \ker(\|\cdot\|_{\wti \cH_t})$, $t\in\bbR$. One can show
(cf.\ \cite{MM02}, \cite{MM04}) that
\begin{equation}
L^2(\bbR; d\Sigma; \cK) \, \text{ and } \, \int_{\bbR}^{\oplus} d \mu(t) \, \cH_t
\, \text{ are isometrically isomorphic,}
\end{equation}
yielding yet another construction of $L^2(\bbR; d\Sigma; \cK)$.

Finally, we will discuss one more characterization of $L^2(\bbR; d\Sigma; \cK)$ which is
used in Sections \ref{s4} and \ref{s5} and closely patterned after work by Sait{\= o} \cite{Sa71}.

\begin{definition} \lb{dD.12}
Let $\lambda_1, \lambda_2 \in \bbR$, $\lambda_1 < \lambda_2$. \\
$(i)$ Assume that
$Q: [\lambda_1, \lambda_2] \to \cB(\cK)$, $u: [\lambda_1, \lambda_2] \to \cK$, and
$\rho: [\lambda_1, \lambda_2] \to \cB(\cK)$.
Denote by $\Delta$ a finite subdivision of $[\lambda_1,\lambda_2]$ of the form
$\lambda_1 = \eta_0 < \eta_1 < \cdots < \eta_n = \lambda_2$. The norm of $\Delta$,
denoted by $\|\Delta\|$, is defined by $\|\Delta\| = \max_{0 \leq j \leq n-1} [\eta_{j+1} - \eta_j]$.
If the limit
\begin{equation}
\wlim_{\|\Delta\| \to 0} \sum_{j=0}^{n-1} Q(\eta_j') [\rho(\eta_{j+1}) - \rho(\eta_j)] u(\eta_j'),  \lb{D.42}
\end{equation}
or
\begin{equation}
\wlim_{\|\Delta\| \to 0} \sum_{j=0}^{n-1} Q(\eta_j') [u(\eta_{j+1}) - u(\eta_j)],    \lb{D.43}
\end{equation}
with $\eta_j' \in [\eta_{j}, \eta_{j+1}]$, $0 \leq j \leq n-1$, exists in the sense of weak convergence
in $\cK$ independently of the choice of subdivision $\Delta$ and the choice of  $\eta_j'$,
$0 \leq j \leq n-1$, then the limit is denoted by
\begin{equation}
\int_{[\lambda_1,\lambda_2]} Q(\lambda) \, d\rho(\lambda) \, u(\lambda),
\end{equation}
or
\begin{equation}
\int_{[\lambda_1,\lambda_2]} Q(\lambda) \, d u(\lambda),
\end{equation}
respectively. \\
$(ii)$ Suppose that for any $\lambda_1 \in (\lambda_0,\lambda_2]$ the integral \eqref{D.42} or
\eqref{D.43} exists in the sense described in item $(i)$, and that
\begin{equation}
\wlim_{\lambda_1 \downarrow \lambda_0} \int_{[\lambda_1,\lambda_2]}
Q(\lambda) \, d\rho(\lambda) \, u(\lambda),
\end{equation}
or
\begin{equation}
\wlim_{\lambda_1 \downarrow \lambda_0} \int_{[\lambda_1,\lambda_2]} Q(\lambda) \, d u(\lambda),
\end{equation}
exist in the sense of weak convergence in $\cK$. Then one defines the integral over the interval
$(\lambda_1,\lambda_2]$ by
\begin{equation}
\int_{(\lambda_1,\lambda_2]} Q(\lambda) \, d\rho(\lambda) \, u(\lambda) =
\wlim_{\lambda \downarrow \lambda_1} \int_{[\lambda,\lambda_2]}
Q(\lambda') \, d\rho(\lambda') \, u(\lambda'),
\end{equation}
or
\begin{equation}
\int_{(\lambda_1,\lambda_2]} Q(\lambda) \, d u(\lambda) =
\wlim_{\lambda \downarrow \lambda_1} \int_{[\lambda,\lambda_2]} Q(\lambda') \, d u(\lambda').
\end{equation}
\end{definition}

\begin{lemma} [\cite{Sa71}] \lb{lD.13}
 Let $\lambda_1, \lambda_2 \in \bbR$, $\lambda_1 < \lambda_2$,
 $Q \in C^1([\lambda_1,\lambda_2], \cB(\cK))$, $u \in C^1([\lambda_1, \lambda_2],\cK)$, and
 $\rho: [\lambda_1, \lambda_2] \to \cB(\cK)$, such that  for some constant $C>0$,
 $\|\rho(\lambda)\|_{\cB(\cK)} \leq C$, $\lambda \in [\lambda_1,\lambda_2]$. Suppose, in addition, that
 \begin{align}
 \begin{split}
& \text{ for all $f \in\cK$, } \, \big(f, [Q'(\cdot) \rho(\cdot)
u(\cdot) + Q(\cdot) \rho(\cdot) u'(\cdot)]\big)_{\cK} \,
\text{ is Riemann integrable}       \lb{D.50} \\
& \quad \text{on $[\lambda_1,\lambda_2]$.}
 \end{split}
 \end{align}
Then $\int_{[\lambda_1,\lambda_2]} Q(\lambda) \, d\rho(\lambda) \, u(\lambda)$ exists in the sense of Definition \ref{dD.12}\,$(i)$, and
\begin{align}
& \bigg(f, \int_{[\lambda_1,\lambda_2]} Q(\lambda) \, d\rho(\lambda) \, u(\lambda)\bigg)_{\cK} =
(f, Q(\lambda_2) \rho(\lambda_2) u(\lambda_2))_{\cK} -
(f, Q(\lambda_1) \rho(\lambda_1) u(\lambda_1))_{\cK}    \no \\
& \quad - \int_{\lambda_1}^{\lambda_2}
d\lambda \, \big(f, [Q'(\lambda) \rho(\lambda) u(\lambda)
+ Q(\lambda) \rho(\lambda) u'(\lambda)]\big)_{\cK},
\quad f \in \cK.
\end{align}
\end{lemma}

\begin{lemma} [\cite{Sa71}] \lb{lD.14}
 Let $\lambda_1, \lambda_2 \in \bbR$, $\lambda_1 < \lambda_2$,
 $Q \in C^1([\lambda_1,\lambda_2], \cB(\cK))$, and
 $v: [\lambda_1, \lambda_2] \to \cK$, such that  for some constant $C>0$,
 $\|v(\lambda)\|_{\cK} \leq C$, $\lambda \in [\lambda_1,\lambda_2]$. Suppose, in addition, that
 \begin{align}
\text{ for all $f \in\cK$, } \, \big(f, Q'(\cdot) v(\cdot)\big)_{\cK} \,
\text{ is Riemann integrable on $[\lambda_1,\lambda_2]$.}       \lb{D.52}
 \end{align}
Then $\int_{[\lambda_1,\lambda_2]} Q(\lambda) \, dv(\lambda)$ exists in the sense of
Definition \ref{dD.12}\,$(i)$, and
\begin{align}
\begin{split}
\bigg(f, \int_{[\lambda_1,\lambda_2]} Q(\lambda) \, dv(\lambda)\bigg)_{\cK} &=
(f, Q(\lambda_2) v(\lambda_2))_{\cK} - (f, Q(\lambda_1) v(\lambda_1))_{\cK}     \\
& \quad - \int_{\lambda_1}^{\lambda_2}
d\lambda \, \big(f, Q'(\lambda) v(\lambda)\big)_{\cK},  \quad f \in \cK.
\end{split}
\end{align}
\end{lemma}

\begin{definition} \lb{dD.15}
Let $\lambda_1, \lambda_2 \in \bbR$, $\lambda_1 < \lambda_2$. \\
$(i)$ Assume that
$v, w: [\lambda_1, \lambda_2] \to \cK$, and that 
$\Sigma: \mathfrak{B}(\bbR) \to \cB(\cK)$ is a bounded operator-valued measure as defined in
Definition \ref{dA.6a}. Denote by $\Delta$ a finite subdivision of $[\lambda_1,\lambda_2]$ of the form
$\lambda_1 = \eta_0 < \eta_1 < \cdots < \eta_n = \lambda_2$ as in Definition \ref{dD.12} with norm
$\|\Delta\|$. If the limit
\begin{equation}
\lim_{\|\Delta\| \to 0} \sum_{j=0}^{n-1} (v(\eta_j'),
[\Sigma(\eta_{j+1}) - \Sigma(\eta_j)] w(\eta_j'))_{\cK},  \lb{D.54}
\end{equation}
with $\eta_j' \in [\eta_{j}, \eta_{j+1}]$, $0 \leq j \leq n-1$, exists independently of the choice of
subdivision $\Delta$ and the choice of  $\eta_j'$, $0 \leq j \leq n-1$, then the limit is denoted by
\begin{equation}
\int_{[\lambda_1,\lambda_2]} (v(\lambda), \, d\Sigma(\lambda) \, w(\lambda))_{\cK},    \lb{D.55}
\end{equation}
$(ii)$ Suppose that for any $\lambda_1 \in (\lambda_0,\lambda_2]$ the integral \eqref{D.55} exists,
and that
\begin{equation}
\lim_{\lambda_1 \downarrow \lambda_0} \int_{[\lambda_1,\lambda_2]}
(v(\lambda), \, d\Sigma(\lambda) \, w(\lambda))_{\cK} \, \text{ exists.}
\end{equation}
Then one defines the integral over the interval
$(\lambda_1,\lambda_2]$ by
\begin{equation}
\int_{(\lambda_1,\lambda_2]} (v(\lambda), \, d\Sigma(\lambda) \, w(\lambda))_{\cH} =
\lim_{\lambda \downarrow \lambda_1} \int_{[\lambda,\lambda_2]}
(v(\lambda'), \, d\Sigma(\lambda') \, w(\lambda'))_{\cH}.
\end{equation}
\end{definition}

\begin{lemma} [\cite{Sa71}] \lb{lD.16}
 Let $\lambda_1, \lambda_2 \in \bbR$, $\lambda_1 < \lambda_2$,
 $v, w \in C^1([\lambda_1,\lambda_2], \cK)$, and
$\Sigma: \mathfrak{B}(\bbR) \to \cB(\cK)$ a bounded operator-valued measure as defined in
Definition \ref{dA.6a}. Then
$\int_{[\lambda_1,\lambda_2]} (v(\lambda) \, d\Sigma(\lambda) \, w(\lambda))_{\cH}$ exists and
\begin{align}
& \int_{[\lambda_1,\lambda_2]} (v(\lambda), \, d\Sigma(\lambda) \, w(\lambda))_{\cK} =
(v(\lambda_2), \Sigma(\lambda_2) w(\lambda_2))_{\cK} -
(v(\lambda_1), \Sigma(\lambda_1) w(\lambda_1))_{\cK}    \no \\
& \quad - \int_{\lambda_1}^{\lambda_2}
d\lambda \, \big[\big(v'(\lambda), \Sigma(\lambda) w(\lambda)\big)_{\cK}
+ \big(v(\lambda), \Sigma(\lambda) w'(\lambda)]\big)_{\cK}\big].
\end{align}
\end{lemma}

Next, consider the vector space $\cD_0$ defined by
\begin{align}
& \cD_0 = \Bigg\{u: \bbR \to \cK \, \Bigg|\, \supp (u) \, \text{compact}; \, u \, \text{ is left-continuous
and has only} \no \\
& \quad \text{discontinuities of the first kind; there exists
$\{\lambda_j(u)\}_{1\leq j \leq N}$, $\lambda_1 < \lambda_2 < \lambda_N$,} \no \\
& \quad \text{such that
$\begin{cases} u(\lambda), & \lambda \in (\lambda_j, \lambda_{j+1}], \\
\slim_{\lambda' \downarrow \lambda_j} u(\lambda'), & \lambda = \lambda_j, \end{cases}$
is strongly continuously}    \no \\
& \quad \text{differentiable on $[\lambda_j, \lambda_{j+1}]$, $1 \leq j \leq N-1$, and
$u(\lambda) = 0$ for $\lambda \in \bbR \backslash [\lambda_1, \lambda_N]$.}\Bigg\}.
\end{align}
Given $\Sigma$ as in Definition \ref{dD.15}, Sait{\= o} \cite{Sa71} then introduces the semi-inner product
on $\cD_0 \times \cD_0$ by
\begin{equation}
(v,w)_{\Sigma,S} = \sum_{j=1}^{K-1} \int_{(\gamma_k,\gamma_{k+1}]}
(v(\lambda), \, d\Sigma(\lambda) w(\lambda))_{\cK}, \quad v, w \in \cD_0,
\end{equation}
where $\{\gamma_k\}_{1\leq k \leq K}$ represents the discontinuities of $v$ and $w$, appropriately
ordered with respect to magnitude. Introducing the subspace of null functions by
$\cN_{\Sigma,S} = \{u \in \cD_0 \, | \, (u,u)_{\Sigma,S} = 0\}$, the completion of $\cD_0/\cN_{\Sigma,S}$ becomes a Hilbert space denoted by
\begin{equation}
L^2(\bbR; d\Sigma; \cK)_S = \ol{\cD_0/\cN_{\Sigma,S}}.
\end{equation}
Next, we recall that $u:\bbR \to \cK$ is called a step function if
\begin{equation}
u(\lambda) = \begin{cases} u_j \in \cK, & \lambda \in (\alpha_j, \beta_j], \; 1 \leq j \leq N, \\
0, & \lambda \in \bbR \backslash \bigcup_{j=1}^N (\alpha_j, \beta_j], \end{cases}.
\end{equation}
Denoting the set of step functions by $\cD_{\rm step}$, then clearly $\cD_{\rm step} \subset \cD_0$, and
one can prove that
\begin{equation}
\ol{\cD_{\rm step}} = L^2(\bbR; d\Sigma; \cK)_S.
\end{equation}
It remains to show that Sait{\= o}'s space $L^2(\bbR; d\Sigma; \cK)_S$ and $L^2(\bbR; d\Sigma; \cK)$
discussed in Theorem \ref{tD.4} and Remark \ref{rD.11}\,$(ii)$ are isometrically isomorphic. We will 
show this by proving that Sait{\= o}'s construction $L^2(\bbR; d\Sigma; \cK)_S$, actually, coincides 
with Berezanskii's construction, $\hatt{L^2(\bbR; d\Sigma; \cK)}$ in \eqref{D.37}:

\begin{theorem} \lb{tD.17}
The Hilbert spaces $L^2(\bbR; d\Sigma; \cK)_S$ and $\hatt{L^2(\bbR; d\Sigma; \cK)}$ coincide.
\end{theorem}
\begin{proof}
By \eqref{D.19A}, $\cD_{\rm step}$ is dense in $\hatt{L^2(\bbR; d\Sigma; \cK)}$. 
Since $\cD_{\rm step}$ is
also dense in $L^2(\bbR; d\Sigma; \cK)_S$, the elementary fact
\begin{align}
\|\chi_{(\alpha,\beta]} \, e_n\|_{L^2(\bbR; d\Sigma; \cK)_S}^2
& = \int_{\bbR} (\chi_{(\alpha,\beta]}(\lambda) \, e_n,
d \Sigma(\lambda) \chi_{(\alpha,\beta]}(\lambda) \, e_n)_{\cK}   \no \\
& = \int_{(\alpha,\beta]} d(e_n, \Sigma(\lambda) e_n)_{\cK}   \no \\
& = (e_n, \Sigma((\alpha, \beta]) \, e_n)_{\cK}    \no \\
& = \|\chi_{(\alpha,\beta]} \, e_n\|_{\hatt{L^2(\bbR; d\Sigma; \cK)}}^2,  \quad
\alpha, \beta \in \bbR, \; n \in \cI,   \lb{D.64}
\end{align}
completes the proof.
\end{proof}

\medskip

\noindent {\bf Acknowledgments.}
We are indebted to Mark Malamud for numerous discussions on this topic.



\end{document}